\documentclass[11pt]{article}
\usepackage[a4paper,margin=29mm]{geometry}
\usepackage{amsmath,amssymb,amsthm,mathtools,microtype}
\usepackage[hidelinks]{hyperref}
\usepackage{booktabs,array}
\usepackage[T1]{fontenc}

\newtheorem{theorem}{Theorem}[section]
\newtheorem{lemma}[theorem]{Lemma}
\newtheorem{proposition}[theorem]{Proposition}
\newtheorem{corollary}[theorem]{Corollary}
\newtheorem{remark}[theorem]{Remark}

\newcommand{\SU}{\mathrm{SU}}
\newcommand{\SO}{\mathrm{SO}}
\newcommand{\Sp}{\mathrm{Sp}}
\newcommand{\Isom}{\operatorname{Isom}}
\newcommand{\Diff}{\operatorname{Diff}}
\newcommand{\Lie}{\operatorname{Lie}}

\title{\textbf{New Homogeneous Einstein Metrics from Detection}}
\author{Anna Siffert}
\date{}
\begin{document}
\maketitle
\begin{abstract}
We use a detection method for invariant Einstein equations to construct and
distinguish asymmetric homogeneous Einstein metrics on several families of
compact homogeneous spaces.  The method retains a geometrically meaningful
symmetry-breaking parameter, eliminates the remaining Einstein equations,
and then reconstructs positive metrics from the resulting detector.

We prove a sharp threshold in a two-factor symplectic family, construct two
asymmetric Einstein metrics on each admissible space
$\SU(n)^3/\Delta\SO(n)$, obtain a sharp threshold for
$\SU(2n)^3/\Delta\Sp(n)$, and construct asymmetric metrics on three
exceptional three-factor families.  The four-factor case exhibits a further
phenomenon.  On
\[
 M_9=\SU(18)^4/\Delta\Sp(9)
\]
the two symmetry types $3+1$ and $2+2$ are both realized: up to homothety
and internal block permutations, the corresponding fixed families contain
exactly three asymmetric Einstein metrics, two with three equal horizontal
scales and one distinct scale (type $3+1$), and one with two equal pairs of
horizontal scales (type $2+2$).  They are pairwise non-isometric,
Riemannian irreducible, and not naturally reductive.  The $2+2$ metric is
globally unique in the full positive $2+2$ fixed family up to block
interchange.
\end{abstract}
\section{Introduction}
The invariant Einstein equation is finite dimensional, but its solutions can
be hidden by an ansatz that identifies geometrically different invariant
directions too early.  We use a detection strategy: retain one such
distinction, eliminate the remaining equations, and reconstruct every
positive solution from a one-variable detector.

The purpose here is concrete.  We establish sharp existence and
nonexistence statements in two- and three-factor families and then show that
the four-factor space
\[
 \SU(18)^4/\Delta\Sp(9)
\]
already carries asymmetric Einstein metrics of two different
symmetry-breaking types.  The response and bifurcation theory explaining why
these types occur is developed in a companion manuscript; none of the
existence or uniqueness results below depends on that theory.

The main conclusions can be summarized as follows.  In the first four parts,
``asymmetric'' refers to the indicated detected branch; in the four-factor
statement the scope is the two fixed families specified below.

\begin{theorem}[Main results]\label{thm:main-results-intro}
The following hold.
\begin{enumerate}
\item Let
\[
 H=\Sp(p+q),\qquad K=\Sp(p)\times\Sp(q),\qquad
 M_{p,q}=(H\times H)/\Delta K,
\]
with $1\le p<q$.  The asymmetric two-factor family obtained by allowing the two horizontal
scales to vary independently contains an Einstein metric if and only if
\[
 q\ge15\quad (p=1),\qquad
 q\ge8p^2+5p+1\quad (p\ge2).
\]
For every admissible pair it is unique in that branch up to homothety and
interchange of the two horizontal factors.

\item For every $n=3$ or $n\ge5$, the space
\[
 \SU(n)^3/\Delta\SO(n)
\]
has exactly two Einstein metrics in the detected asymmetric branch up to
homothety.  They are pairwise non-isometric; neither is naturally reductive
with respect to any connected transitive group of isometries.

\item For
\[
 \SU(2n)^3/\Delta\Sp(n),\qquad n\ge2,
\]
there is a sharp threshold at $n=22$: the detected asymmetric branch contains
no positive Einstein metric for $2\le n\le21$ and exactly two up to homothety
for $n\ge22$.  For $n\ge22$ the two metrics are pairwise non-isometric.

\item Each of
\[
 E_6^3/\Delta\Sp(4),\qquad
 E_7^3/\Delta\SU(8),\qquad
 E_8^3/\Delta\SO(16)
\]
has exactly two Einstein metrics in the detected asymmetric branch up to
homothety.  In each case the twto metrics are pairwise non-isometric and are
not naturally reductive with respect to any connected transitive group of
isometries.

\item On
\[
 M_9=\SU(18)^4/\Delta\Sp(9),
\]
the $3+1$ and $2+2$ fixed families studied in this paper contain, up to
homothety and their internal block permutations, exactly three asymmetric
Einstein metrics: two of type $3+1$ and one of type $2+2$.  The three metrics
are pairwise non-isometric under arbitrary Riemannian isometries, are
Riemannian irreducible, and are not naturally reductive with respect to any
connected transitive group of isometries.  The $2+2$ metric is globally
unique in the positive $2+2$ fixed family up to homothety and interchange of
the two blocks.
\end{enumerate}
\end{theorem}

The individual statements, including the exact reconstruction and full
isometry results, are proved in the corresponding sections below.

The paper is organized as follows.  Section~\ref{sec:detection-framework}
formulates the invariant Einstein equations and the detection reduction.
Section~\ref{sec:twofactor-symplectic} proves the sharp threshold and
uniqueness theorem for the two-factor symplectic family.
Section~\ref{sec:AI-threefactor} treats the $AI$ family
$\SU(n)^3/\Delta\SO(n)$, proving existence, uniqueness, and geometric
distinction of the two asymmetric metrics.  Section~\ref{sec:AII-threefactor}
proves the sharp $n=22$ threshold for $\SU(2n)^3/\Delta\Sp(n)$ and
determines the isometry groups of the resulting metrics.
Section~\ref{sec:predictor} derives a representation-theoretic criterion
that predicts the three-factor crossing before elimination.
Section~\ref{sec:exceptional-threefactor} applies the method to the three
exceptional families and then develops the four-factor analysis, including
the $3+1$ and $2+2$ metrics on $\SU(18)^4/\Delta\Sp(9)$ and their
geometric distinction.  Appendix~\ref{app:H4balanced-coefficients} records
the exact coefficient system for the $2+2$ branch, and
Appendix~\ref{app:H4balanced-n9-elimination} gives its exact elimination
certificate.

\subsection*{Acknowledgements}

The author used OpenAI's ChatGPT as an interactive assistant during the
preparation of this manuscript for exposition, organization, creating the pictures,
proofreading, and suggestions for presenting and checking arguments.
All these contributions were independently verified by the author,
who takes full responsibility for the mathematical content and
conclusions of the paper.

\section{What do we actually solve?}\label{sec:detection-framework}

Before introducing the detector, let us make the underlying Einstein problem
explicit.  Let
\[
 M=G/K
\]
be a compact homogeneous space, choose a reductive decomposition
\[
 \mathfrak g=\mathfrak k\oplus\mathfrak m,
\]
and fix an $\operatorname{Ad}(G)$-invariant inner product $Q$ on
$\mathfrak g$.  We use a $Q$-orthogonal decomposition into
$\operatorname{Ad}(K)$-invariant irreducible summands,
\[
 \mathfrak m=\mathfrak m_1\oplus\cdots\oplus\mathfrak m_r.
\]
In a diagonal invariant subfamily for which symmetry forces the Ricci tensor
to have no off-diagonal terms between these summands, an invariant metric has
the form
\[
 g=x_1Q|_{\mathfrak m_1}+\cdots+x_rQ|_{\mathfrak m_r},
 \qquad x_i>0.
\]
This is the situation in the first families below.  Geometrically, the
numbers $x_i$ are the relative sizes assigned to the different invariant
directions.  Multiplying all $x_i$ by the same constant changes only the
overall scale, so one scale may be normalized.

There is one qualification that will matter later.  If equivalent isotropy
summands occur, a general invariant metric need not be diagonal in a fixed
decomposition, and the invariant Ricci tensor may contain off-diagonal
components.  In the $H^s/\Delta K$ part of the paper these equivalent
directions are therefore encoded by a positive matrix
$P\in\operatorname{Sym}^+(V_s)$ rather than by one scalar for each copy.
The scalar discussion in the present section is the simplest model of the
same detection mechanism, not a restriction imposed on the later
multi-factor analysis.

For a diagonal metric as above, let $\operatorname{Ric}_g^\sharp$ be the Ricci
endomorphism defined by
\[
 g(\operatorname{Ric}_g^\sharp X,Y)=\operatorname{Ric}_g(X,Y).
\]
Under the stated diagonal hypothesis,
\[
 \operatorname{Ric}_g^\sharp|_{\mathfrak m_i}
 =r_i\,\mathrm{Id}_{\mathfrak m_i}.
\]
The metric is Einstein precisely when all these eigenvalues agree:
\begin{equation}\label{eq:Einstein-eigenvalue-equality}
 r_1=r_2=\cdots=r_r.
\end{equation}
Thus, after fixing the overall scale, the diagonal homogeneous Einstein
problem becomes a finite system of algebraic equations for positive numbers
$x_i$.  For example, one may solve
\begin{equation}\label{eq:Einstein-differences-general}
 r_1-r_r=0,\qquad\ldots,\qquad r_{r-1}-r_r=0.
\end{equation}
In a non-diagonal invariant family the same principle holds, but these scalar
equalities are supplemented by the vanishing of the relevant off-diagonal
Ricci components.  This is exactly what the matrix-valued response blocks
later record.

For reference, in the diagonal setting the $r_i$ are given by the standard
homogeneous Ricci formula; see, for example, \cite[Chapter~7]{Besse}.  Put
$d_i=\dim\mathfrak m_i$, write
\[
 B|_{\mathfrak m_i}=-b_iQ|_{\mathfrak m_i},
\]
where $B$ is the Killing form, and define
\[
 [ijk]=
 \sum_{\alpha,\beta,\gamma}
 Q([e_\alpha^i,e_\beta^j],e_\gamma^k)^2
\]
for $Q$-orthonormal bases of the corresponding summands.  Then
\begin{equation}\label{eq:homogeneous-Ricci-formula}
 r_i=
 \frac{b_i}{2x_i}
 -\frac{1}{2d_i}\sum_{j,k}[ijk]\frac{x_k}{x_i x_j}
 +\frac{1}{4d_i}\sum_{j,k}[ijk]\frac{x_i}{x_jx_k}.
\end{equation}
Nothing in the detection viewpoint changes this formula or the Einstein
equations.  The difference is how we organize the equations and which
symmetry-breaking variable we insist on keeping visible.

\subsection*{Where detection enters}

Suppose a more symmetric ansatz has required two geometrically meaningful
scales to be equal, say
\[
 x_1=x_2.
\]
To look for metrics missed by that ansatz, the obvious first step is to allow
$x_1$ and $x_2$ to differ.  The role of detection begins after this
enlargement: it organizes the larger system while keeping the restored
distinction visible.

Instead of treating all remaining variables on an equal footing, we keep the
restored distinction visible.  In the simplest scalar situation one may use
the ratio
\[
 u=\frac{x_1}{x_2}.
\]
Then
\[
 u=1
 \quad\Longleftrightarrow\quad
 x_1=x_2,
\]
while $u\ne1$ records the asymmetry we are looking for.  We now ask a very
specific question:

\begin{quote}
\emph{Do the Einstein equations force $u=1$, or do they allow $u\ne1$?}
\end{quote}

This is the detection question.  We simplify the other Einstein equations
while keeping $u$ visible.  In favorable examples elimination and exact
reconstruction reduce the problem to a scalar necessary condition
\[
 D(u)=0.
\]
We call such a scalar condition a \emph{detector}.  A zero with $u\ne1$ is a
candidate asymmetric Einstein metric, not yet a solution: elimination can
introduce extraneous roots, and the remaining metric variables must be
reconstructed, checked against the original Einstein equations, and proved
positive.  Whenever a later theorem counts detector roots as Einstein
metrics, those reconstruction and positivity steps are part of the proof.

The ratio $u$ is only the simplest example of a hidden variable.  In the
multi-factor problem the hidden directions form representation spaces, and
the detector becomes a scalar determinant or a representation-resolved
response block.  The underlying question is unchanged: does the Einstein
response remain invertible in the hidden direction, or does it lose rank?

A useful detector need not exist for every homogeneous Einstein problem.
Elimination may become too large, may introduce extraneous algebraic
components, or may fail to leave a manageable one-variable condition.  The
point is therefore not a universal algorithm.  It is a way of choosing
\emph{which geometric distinction to preserve} while simplifying the
Einstein equations.

Several outcomes can occur.  Detector roots may persist as a structural
parameter varies, appear or disappear at a threshold, or fail to produce an
admissible reconstructed metric.  A branch may also terminate away from the
symmetric state when two detector roots merge.  We introduce the more precise
threshold and bifurcation language only when these phenomena occur in the
examples.  In particular, a zero of an detector is not by itself called a
bifurcation point; that terminology is used only after the corresponding
linearized and nonlinear hypotheses have been verified.

\subsection*{How to read the proofs}

Each family below follows the same three questions:
\begin{enumerate}
\item \textbf{Expose:} which equality or symmetry has hidden the distinction
we want to test?
\item \textbf{Detect:} do the Einstein equations force the restored
distinction to disappear, or do they allow it to persist?
\item \textbf{Reconstruct:} do the detected candidates satisfy the original
Einstein equations with all metric variables positive?
\end{enumerate}

The first example makes all three steps concrete and also shows why this
organization can be much simpler than attacking the full Einstein system at
once.  Later sections use the same logic with matrix-valued hidden variables
and representation-resolved response channels.

\section{A first threshold example: a two-factor symplectic family}\label{sec:twofactor-symplectic}

The first example shows the detection mechanism in its simplest threshold
form.  Rather than eliminate the full enlarged Einstein system at once, we
keep the restored distinction
\[
 x_1\stackrel{?}{=}x_2
\]
visible.  The Ricci difference for these two scales separates the symmetric
locus from an asymmetric branch; on that branch the remaining equations reduce
to one polynomial condition and one admissibility inequality.  Their
interaction gives the sharp threshold.

\subsection{What is known}
Let $H=\Sp(p+q)$, $K=\Sp(p)\times\Sp(q)$, and
$M_{p,q}=(H\times H)/\Delta K$.  Interchanging the two factors of $K$
identifies the pairs with parameters $(p,q)$ and $(q,p)$, so throughout this
section we order the unequal factors and assume
\[
 1\le p<q.
\]
Here $H/K$ is the irreducible symmetric
space of Cartan type CII; for the standard terminology and classification of compact symmetric spaces, see \cite{Helgason}.  Lauret--Will study the general Einstein problem on
$(H\times H)/\Delta K$ and classify the irreducible symmetric case when $K$
is simple \cite{LWHH}.  Our family lies outside that subclass because $K$
has two simple factors.

For $p\ne q$ these factors have different Killing constants,
\[
 B_{\mathfrak{sp}(p)}
 =\frac{p+1}{p+q+1}B_{\mathfrak h}|_{\mathfrak{sp}(p)},\qquad
 B_{\mathfrak{sp}(q)}
 =\frac{q+1}{p+q+1}B_{\mathfrak h}|_{\mathfrak{sp}(q)}.
\]
This is the distinction retained below.  Related aligned-space work gives
broader context \cite{LWaligned,LWtwo}.  Our result concerns an asymmetric diagonal branch for unequal $p,q$ and its
sharp existence threshold.  The results cited above do not contain this branch
or the threshold established below.

\subsection{The hidden parameter}
Let $H=\Sp(p+q)$ and $K=\Sp(p)\times\Sp(q)$.  For
$M=(H\times H)/\Delta K$, the reductive tangent space has the splitting
\[
 \mathfrak m=\mathfrak q_1\oplus\mathfrak q_2\oplus
 \mathfrak{sp}(p)^-\oplus\mathfrak{sp}(q)^-.
\]
The last two summands are inequivalent when $p\ne q$.  There is also a
canonical reason for taking the two horizontal copies to be orthogonal.  Let
$\sigma$ be the Cartan involution of the symmetric pair $H/K$.  The
automorphism $(\sigma,\operatorname{id})$ of $H\times H$ preserves
$\Delta K$, acts by $-1$ on $\mathfrak q_1$ and by $+1$ on
$\mathfrak q_2\oplus\mathfrak k^-$, and hence kills every mixed term
between $\mathfrak q_1$ and the other summands.  Applying the analogous
involution in the second factor kills the remaining horizontal mixed terms.
Thus the following diagonal family is the fixed-point family of a finite group
of automorphisms.  In particular its Ricci tensor has the same block form, and
by symmetric criticality a solution of the scalar Einstein equations below is
a genuine $H\times H$-invariant Einstein metric, not merely a critical point
of an arbitrary diagonal ansatz.

We therefore consider
\begin{equation}\label{TF-metric}
 g=x_1Q|_{\mathfrak q_1}+x_2Q|_{\mathfrak q_2}
   +x_3Q|_{\mathfrak{sp}(p)^-}+x_4Q|_{\mathfrak{sp}(q)^-},
 \qquad x_i>0,
\end{equation}
where $Q=-B_H$ on each $H$-factor with the usual induced normalization.
Up to homothety set $x_4=1$.

The dimensions of $\mathfrak q_i$, $\mathfrak{sp}(p)^-$, and
$\mathfrak{sp}(q)^-$ are, respectively,
\[
 d=4pq,\qquad e=p(2p+1),\qquad f=q(2q+1).
\]
The bracket coefficients $[ijk]$ from
\eqref{eq:homogeneous-Ricci-formula} that enter the calculation are encoded by
\[
 A=\frac{pq(2p+1)}{2(p+q+1)},\qquad
 B=\frac{pq(2q+1)}{2(p+q+1)}.
\]
Substitution into~\eqref{eq:homogeneous-Ricci-formula} gives

\[
 r_1=\frac{dx_1-Ax_3-Bx_4}{2dx_1^2},\qquad
 r_2=\frac{dx_2-Ax_3-Bx_4}{2dx_2^2},
\]
\[
 r_3=\frac1{2x_3}-\frac{A}{ex_3}
 +\frac{Ax_3}{4e}\left(x_1^{-2}+x_2^{-2}\right),
\]
\[
 r_4=\frac1{2x_4}-\frac{B}{fx_4}
 +\frac{Bx_4}{4f}\left(x_1^{-2}+x_2^{-2}\right).
\]

Thus the Einstein problem in this symmetry-fixed family is completely explicit.  After fixing
$x_4=1$, one must find $x_1,x_2,x_3>0$ such that
\begin{equation}\label{eq:TF-direct-Einstein-system}
 r_1-r_2=0,\qquad
 r_2-r_3=0,\qquad
 r_3-r_4=0.
\end{equation}
Equivalently, one may solve $r_1=r_2=r_3=r_4$ and recover the common value as
the Einstein constant.  This is the direct route: three coupled nonlinear
rational equations in the three remaining metric variables.

The first equation in \eqref{eq:TF-direct-Einstein-system} is the Ricci
difference associated with the restored distinction $x_1\ne x_2$.  Its
factorization separates the symmetric locus from the asymmetric branch.  On
the latter branch, the remaining Einstein equations reduce to a one-variable
polynomial condition together with a positivity condition, from which the
sharp threshold follows.

\subsection{The detection reduction}
The first response factors:
\begin{equation}\label{TF-detector}
 r_1-r_2=-\frac{x_1-x_2}{2dx_1^2x_2^2}
 \{dx_1x_2-(x_1+x_2)(Ax_3+Bx_4)\}.
\end{equation}
Thus an Einstein metric is either on the symmetric locus $x_1=x_2$ or on a hidden branch.  Put
\[
 z=x_3,\qquad s=x_1+x_2,\qquad u=x_1x_2.
\]
On the hidden branch,
\begin{equation}\label{TF-u}
 u=\frac{s((2p+1)z+2q+1)}{8(p+q+1)}.
\end{equation}
Moreover $r_1=r_2=1/(2s)$.  Eliminating $s$ from $r_3=r_4=1/(2s)$ gives
\begin{equation}\label{TF-cubic}
 C_{p,q}(z)=c_3z^3+c_2z^2+c_1z+c_0=0,
\end{equation}
where
\[
 c_3=(2p+1)(q+1)(2p+4q+1),
\]
\[
 c_2=-4p^3-16p^2q-8p^2-24pq^2-20pq-p-8q^3-16q^2-6q+1,
\]
\[
 c_1=8p^3+24p^2q+16p^2+16pq^2+20pq+6p+4q^3+8q^2+q-1,
\]
\[
 c_0=-(p+1)(2q+1)(4p+2q+1).
\]
In particular $c_3,c_1>0$ and $c_2,c_0<0$.

On the hidden branch the reconstruction formulas give $s>0$ and $u>0$.
Hence the roots $x_1,x_2$ of $X^2-sX+u$ are distinct positive reals exactly
when its discriminant is positive.  Direct substitution of the Einstein
expression for $s$ gives
\[
 s^2-4u=-\Lambda_{p,q}(z)R_{p,q}(z),
 \qquad \Lambda_{p,q}(z)>0\quad(z>0),
\]
where all factors in $\Lambda_{p,q}$ are positive in the present parameter
range.  Thus positive distinct reconstruction is equivalent to
\begin{equation}\label{TF-admiss}
 R_{p,q}(z)<0,
\end{equation}
where
\[
 R_{p,q}(z)=(2p+1)^2(q+1)z^2+B_{p,q}z+D_{p,q},
\]
\[
 B_{p,q}=-4p^3-8p^2q-10p^2+4pq^2-4p+2q^2+2q,
\]
\[
 D_{p,q}=8p^3+12p^2q+14p^2+4pq+4p-2q^2-3q-1.
\]
Hence the four-variable Einstein problem has become
\[
 C_{p,q}(z)=0,\qquad R_{p,q}(z)<0,\qquad z>0.
\]

\subsection{The detection resultant and the sharp threshold}
The resultant factors as
\begin{equation}\label{TF-resultant}
 \operatorname{Res}_z(C,R)
 =-64(2p+1)^2(q+1)(p+q+1)^4L_{p,q}H(p,q),
\end{equation}
where $L_{p,q}=p^2+3pq+3p+q^2+2q+1>0$ and
\begin{align}\label{TF-H}
H(p,q)={}&4(2p-1)q^4-4(16p^3+2p^2-5p+1)q^3\\
&-(2p-1)(4p^2+2p-1)q^2
+2p(2p+1)(2p^2+4p-1)q-p^2(2p+1)^2.\nonumber
\end{align}

\begin{lemma}[root ordering]\label{TF-rootorder}
Whenever $R_{p,q}<0$ somewherer on $(0,\infty)$, it has a unique positive zero $\rho$, and
\[
 R(z)<0\iff0<z<\rho.
\]
For $p\ge2$, the cubic $C_{p,q}$ is strictly increasing on $(0,\rho)$.
For $p=1$ the same monotonicity holds throughout the finite subthreshold
range $q\le14$.  The superthreshold range $q\ge15$ is treated separately
in the proof of Theorem~\ref{TF-main}.
\end{lemma}
\begin{proof}
The discriminant of $R$ is
\[
4(2p+1)^2(p+q+1)^2
\bigl(p^2-6pq-6p+q^2+2q+1\bigr).
\]
We first show that an admissible positive interval forces $R(0)<0$.
Write $q=p+t$ and $y=t+1>0$.  The linear coefficient of $R$ becomes
\[
 B_{p,q}=2(2p+1)\bigl(y^2-y-2p^2-p\bigr),
\]
whereas positivity of the discriminant of $R$ is equivalent to
\[
 y^2-4py-4p^2>0.
\]
If $R(0)\ge0$ and the upward-opening quadratic $R$ were negative somewhere
on $(0,\infty)$, then it would have two positive roots, so necessarily
$B_{p,q}<0$ and $\operatorname{disc}R>0$.  The first inequality gives
$y^2-y<2p^2+p$, while the second gives
$y^2-y>4py+4p^2-y\ge2p^2+p$, a contradiction.  Hence $R(0)<0$, and there is
exactly one positive zero $\rho$.
For the cubic set $z_0=c_1/(-2c_2)$.  Then
\[
 C'(z)=3c_3z^2+2c_2z+c_1>0\qquad(0<z<z_0).
\]
After substitution, $4c_2^2R(z_0)$ is a polynomial in $t=q-p$ all of whose coefficients are strictly positive for $p\ge2$.  Hence $R(z_0)>0$ and $\rho<z_0$.  For $p=1$, the corresponding numerator is
\[
-128q^8+2192q^7+35856q^6+170408q^5+369496q^4
+410937q^3+240819q^2+71883q+9441,
\]
which is positive for $1\le q\le14$ since the first two terms already have positive sum there and all remaining terms are positive.
\end{proof}

Since $C(z)<0$ for $z<0$, evaluation of the resultant at the two roots of $R$ gives
\begin{equation}\label{TF-sign}
 \operatorname{sgn}C(\rho)=\operatorname{sgn}H(p,q).
\end{equation}
Together with $C(0)<0$ and Lemma~\ref{TF-rootorder}, this converts existence into the sign of $H$.

\begin{theorem}[sharp symplectic threshold]\label{TF-main}
The family \eqref{TF-metric} contains an Einstein metric with $x_1\ne x_2$ if and only if
\[
 q\ge15\quad(p=1),\qquad
 q\ge8p^2+5p+1\quad(p\ge2).
\]
For every admissible pair $(p,q)$ the admissible zero of $C_{p,q}$ is unique.  Hence the metric is unique in this asymmetric branch up to homothety and interchange of $x_1,x_2$.
\end{theorem}
\begin{proof}
Write $q=p+t$.  The polynomial $H(p,p+t)$ has positive leading coefficient and all remaining coefficients negative:
\begin{align*}
H(p,p+t)={}&(8p-4)t^4+(-64p^3+24p^2+4p-4)t^3\\
&+(-192p^4+16p^3+36p^2-8p-1)t^2\\
&+(-192p^5+64p^3-4p)t-64p^6+32p^4-4p^2.
\end{align*}
Descartes' rule gives at most one positive zero.  Since $H(p,p)<0$ and the leading coefficient is positive, $H(p,p+t)\to+\infty$ as $t\to\infty$; hence there is exactly one positive zero.  For $p\ge2$, put $Q_p=8p^2+5p$.  Then
\[
H(p,Q_p)=-4p^2(2p+1)^2(32p^3+132p^2+84p+9)<0,
\]
whereas
\[
H(p,Q_p+1)=3584p^7+3008p^6-1216p^5-3104p^4-2032p^3-680p^2-120p-9>0.
\]
The last positivity follows, for example, by writing $p=2+t$, when all coefficients are positive.  Hence the unique crossing lies between these consecutive integers.  Equation \eqref{TF-sign} and Lemma~\ref{TF-rootorder} then give existence and uniqueness for $p\ge2$.

For $p=1$, direct evaluation gives $H(1,14)<0<H(1,15)$, and the same
one-crossing argument shows
\[
 H(1,q)<0\quad(q\le14),\qquad H(1,q)>0\quad(q\ge15).
\]
For $q\le14$, Lemma~\ref{TF-rootorder} and \eqref{TF-sign} exclude an
admissible zero.  For $q\ge15$, one has
\[
 R_{1,q}(0)=-2q^2+13q+25<0,
\]
so the admissible interval is nonempty; since \eqref{TF-sign} gives
$C_{1,q}(\rho)>0$ while $C_{1,q}(0)<0$, the intermediate value theorem
gives an admissible zero.

It remains only to record uniqueness on the superthreshold branch $p=1$, $q\ge15$.
Here
\[
 C'_{1,q}(1)=-4(q+1)(q+2)(3q-4)<0,
 \qquad
 C'_{1,q}(0)=4q^3+24q^2+45q+29>0.
\]
Since $C'_{1,q}$ is an upward-opening quadratic, it has two positive zeros
\[
 0<\alpha<1<\beta,
\]
and therefore $C_{1,q}$ is strictly increasing on $(0,\alpha)$ and strictly
decreasing on $(\alpha,1)$.

On the reconstruction side,
\[
 R_{1,q}(0)=-2q^2+13q+25<0,\qquad
 R_{1,q}(1)=4(q+2)^2>0,
\]
and
\[
 R'_{1,q}(z)=18(q+1)z+6(q^2-q-3)>0
 \qquad(z>0,\ q\ge15).
\]
Hence the unique positive zero $\rho$ of $R_{1,q}$ satisfies
\[
 0<\rho<1.
\]
The resultant sign relation and $H(1,q)>0$ for $q\ge15$ give
\[
 C_{1,q}(\rho)>0.
\]
If $\rho\le\alpha$, strict increase on $(0,\rho)$ together with
$C_{1,q}(0)<0<C_{1,q}(\rho)$ gives exactly one zero.  If
$\alpha<\rho<1$, then $C_{1,q}(\alpha)\ge C_{1,q}(\rho)>0$, so there is
exactly one zero in $(0,\alpha)$ and none in $[\alpha,\rho)$.  Thus the
admissible zero is unique also for $p=1$.
\end{proof}

The threshold polynomial $H(p,q)$ therefore records the compatibility of
the asymmetric branch with positive reconstruction.

\subsection{Orthogonal contrast}\label{subsec:orthogonal-contrast}
For $2\le p<q$, consider
\[
 \frac{\SO(p+q)\times\SO(p+q)}{\Delta(\SO(p)\times\SO(q))}.
\]
(The case $p=1$ has a trivial $\SO(1)$ factor and is not the same four-scale
comparison.)  The same branch response exposes a formal asymmetric branch.
Put
\[
 a=p-1,\qquad b=q-1=a+c,\qquad c\ge1.
\]
On that branch the discriminant of the reconstructed pair $(x_1,x_2)$ has the sign opposite to a quadratic $Q_{a,b}(z)$.  Its leading and constant coefficients are positive.  The linear coefficient, after division by the positive factor $a$, is
\[
 \ell_{a,c}=c^2-2c-2a^2-2a.
\]
We claim that
\[
 Q_{a,b}(z)>0\qquad(z>0).
\]
There are two cases.  If $\ell_{a,c}\ge0$, then all three coefficients of $Q_{a,b}$ are non-negative and the constant coefficient is strictly positive, so the claim is immediate.

Assume now that $\ell_{a,c}<0$.  A direct calculation gives
\[
 \operatorname{disc}Q_{a,b}=a^2(a+b)^2F_{a,c},
 \qquad
 F_{a,c}=c^2-4ac-4a^2-8c+8.
\]
Using the definition of $\ell_{a,c}$,
\[
 F_{a,c}=\ell_{a,c}-4ac-2a^2-6c+2a+8.
\]
Set
\[
 f(a,c):=-4ac-2a^2-6c+2a+8.
\]
On $a\ge1$ and $c\ge1$,
\[
 \partial_a f=-4c-4a+2<0,
 \qquad
 \partial_c f=-4a-6<0.
\]
Thus $f$ is decreasing in both variables on this region, and therefore
\[
 f(a,c)\le f(1,1)=-2.
\]
Since $\ell_{a,c}<0$, it follows that $F_{a,c}<0$.  Thus $Q_{a,b}$ has negative discriminant and positive leading coefficient, so it is positive on all of $\mathbb R$.  This proves the claim in both cases.

Consequently the asymmetric orthogonal branch has negative reconstruction discriminant and cannot produce a Riemannian metric.  The contrast is useful conceptually: the same detection scheme that reveals a genuine branch in the symplectic case detects and rejects the corresponding formal branch in the orthogonal case.

\subsection{Full isometry group}

\label{TF-sec:isometry}
Write $n=p+q$, $H=\Sp(n)$ and $G=H\times H$.  The action of $G$ on
\[
 M_{p,q}=G/\Delta K
\]
has finite ineffective kernel $\Delta Z(H)$.  We first determine the identity component of the full isometry group of a detected asymmetric metric.

\begin{lemma}[normalizer and centralizer]\label{TF-lem:normalizer}
For a Lie subalgebra $\mathfrak l\subset\mathfrak g$, write
\[
 N_{\mathfrak g}(\mathfrak l)
 =\{X\in\mathfrak g:[X,\mathfrak l]\subset\mathfrak l\}
\]
for its Lie-algebra normalizer.  Then
\[
 N_{\mathfrak h\oplus\mathfrak h}(\Delta\mathfrak k)=\Delta\mathfrak k.
\]
Consequently the connected centralizer of the effective $G$-action on $M_{p,q}$ is trivial.
\end{lemma}
\begin{proof}
The symmetric subalgebra
\[
 \mathfrak k=\mathfrak{sp}(p)\oplus\mathfrak{sp}(q)
\]
is self-normalizing in $\mathfrak h=\mathfrak{sp}(n)$ and has zero Lie-algebra centralizer.  If $(X,Y)$ normalizes $\Delta\mathfrak k$, then for every $A\in\mathfrak k$,
\[
 ([X,A],[Y,A])\in\Delta\mathfrak k.
\]
Hence $[X-Y,A]=0$ for all $A$, so $X=Y$; self-normalization then gives $X=Y\in\mathfrak k$.  The standard description of the centralizer of a transitive action by $N_G(K)/K$ gives the final assertion.
\end{proof}

\begin{lemma}[right stabilizer of the free fibre]\label{TF-lem:rightfibre}
Consider either free $H$-orbit obtained from one simple factor of
$G=H\times H$.  The induced metric is left invariant on $H$ and, with
respect to the $Q=-B_H$ decomposition
\[
 \mathfrak h=\mathfrak q\oplus\mathfrak{sp}(p)\oplus\mathfrak{sp}(q),
\]
its metric endomorphism is scalar on each of the three displayed summands.
For a detected asymmetric metric one has $x_3\ne x_4$.  Consequently the
identity component of the subgroup of right translations preserving the
induced fibre metric is contained in $K=\Sp(p)\times\Sp(q)$.
\end{lemma}
\begin{proof}
A right translation by $a\in H$ preserves a left-invariant metric precisely
when $\operatorname{Ad}(a)$ preserves its metric endomorphism.  Infinitesimally, if
$X\in\mathfrak h$ generates such right isometries, then $\operatorname{ad}_X$ commutes with
that endomorphism.  Since the two eigenvalues on
$\mathfrak{sp}(p)$ and $\mathfrak{sp}(q)$ are different, at least one of these
two simple summands is an isolated eigenspace (even if the horizontal
eigenvalue coincides with the other vertical eigenvalue).  Hence $\operatorname{ad}_X$
preserves that isolated simple summand.  Its Lie-algebra normalizer in
$\mathfrak{sp}(p+q)$ is
$\mathfrak{sp}(p)\oplus\mathfrak{sp}(q)=\mathfrak k$, so $X\in\mathfrak k$.
\end{proof}

\begin{lemma}[diagonal peeling for $C_n$]\label{TF-lem:diagonal-peeling}
Let $\mathfrak h=C_n$ with $n\ge4$, let $r\ge1$, and let
$\Delta\mathfrak h\subset\mathfrak h^r$ be a diagonal copy.  If a compact
reductive subalgebra $\mathfrak b\subset\mathfrak h^r$ satisfies
\[
 \mathfrak h^r=\Delta\mathfrak h+\mathfrak b,
\]
then, after permuting the simple factors and applying automorphisms of them,
\[
 \mathfrak b=\mathfrak h^{r-1}\oplus\mathfrak a
\]
for a compact subalgebra $\mathfrak a\subset\mathfrak h$; the case
$\mathfrak a=\mathfrak h$ is allowed.
\end{lemma}
\begin{proof}
We argue by induction on $r$.  The assertion is tautological for $r=1$.
Assume $r\ge2$.  Since the ambient algebra $\mathfrak h^r$ and
$\Delta\mathfrak h$ are strongly semisimple, Theorem~3.3 of
\cite{Onishchik1969} replaces $\mathfrak b$ by its maximal strongly
semisimple subalgebra without changing the equality
$\mathfrak h^r=\Delta\mathfrak h+\mathfrak b$ at the strongly semisimple
level.

Suppose first that $\mathfrak b$ contains no simple ideal of the ambient
product.  The decomposition is then effective.  It is also irreducible:
a nontrivial splitting of the ambient product into two ideals would force
the simple diagonal algebra $\Delta\mathfrak h$, which projects
nontrivially to every ambient factor, to split as a direct sum.  Hence
Theorem~4.3 of \cite{Onishchik1969} makes the decomposition primitive.
But the compact simple factorization list contains no nontrivial proper
factorization whose ambient simple algebra is of type $C_n$.  This is a
contradiction.

Thus $\mathfrak b$ contains an ambient simple ideal, say the first
$\mathfrak h$ factor.  Write
\[
 \mathfrak h^r=\mathfrak h_1\oplus\widehat{\mathfrak h}.
\]
Because $\mathfrak h_1\subset\mathfrak b$, subtraction of the
$\mathfrak h_1$ component shows
\[
 \mathfrak b=\mathfrak h_1\oplus
 (\mathfrak b\cap\widehat{\mathfrak h}).
\]
Projecting the factorization to $\widehat{\mathfrak h}\simeq
\mathfrak h^{r-1}$ gives
\[
 \widehat{\mathfrak h}
 =
 \Delta\mathfrak h+
 (\mathfrak b\cap\widehat{\mathfrak h}),
\]
where the projected copy of $\Delta\mathfrak h$ is again diagonal.
The induction hypothesis now yields
\[
 \mathfrak b
 =
 \mathfrak h^{r-1}\oplus\mathfrak a
\]
after permuting factors and applying automorphisms.  This proves the lemma.
\end{proof}

\begin{theorem}[connected full isometry group]
\label{TF-thm:isom0}
Let $1\le p<q$ and let $g$ be an asymmetric Einstein metric of
Theorem~\ref{TF-main}.  Then
\[
 \operatorname{Isom}_0(M_{p,q},g)
 =\frac{\Sp(n)\times\Sp(n)}{\Delta Z(\Sp(n))},
 \qquad n=p+q.
\]
\end{theorem}
\begin{proof}
Let $I_0=\operatorname{Isom}_0(M_{p,q},g)$, let $\mathfrak i$ be its Lie
algebra, and let $\mathfrak l$ be the isotropy algebra at the origin.  Put
$\mathfrak g=C_n\oplus C_n$.  Since the effective image of $G$ is transitive,
\[
 \mathfrak i=\mathfrak g+\mathfrak l,
 \qquad
 \mathfrak g\cap\mathfrak l=\Delta(C_p\oplus C_q).
\]
By Lemma~\ref{TF-lem:normalizer}, $\mathfrak i$ has zero centre.  It has no
simple ideal of type $A_1$: every homomorphism $C_n\to A_1$ is zero for
$n\ge4$, so such an ideal would centralize $\mathfrak g$, contrary to the same
lemma.  Thus $\mathfrak i$ is strongly semisimple.  We exclude primitive simple
enlargements in two cases.

If $p,q\ge2$, the isotropy $C_p\oplus C_q$ is strongly semisimple, and
Onishchik's strongly semisimple reduction retains
\[
 \mathfrak g\cap\mathfrak l_s=\Delta(C_p\oplus C_q)
\]
for the maximal strongly semisimple subalgebra
$\mathfrak l_s\subset\mathfrak l$.  A primitive chain issuing from either
$C_n$ ideal of $\mathfrak g$ could only end in one of the two proper simple
compact factorizations in which $C_n$ occurs,
\[
 A_{2n-1}=C_n+A_{2n-2},\qquad
 D_{2n}=C_n+B_{2n-1}.
\]
In both cases the intersection with $C_n$ is $C_{n-1}$.  Along the strongly
semisimple reduction the full-rank algebra $C_p\oplus C_q$ is carried in the
intersection, but
\[
 \operatorname{rk}(C_p\oplus C_q)=p+q=n>n-1=\operatorname{rk}C_{n-1}.
\]
Hence no primitive simple enlargement is possible.

It remains to justify the exceptional case $p=1$.  Here
\[
 \mathfrak k=C_1\oplus C_{n-1},\qquad n=q+1\ge16.
\]
Consider an effective primitive component containing one of the two
$C_n$ ideals, say $\mathfrak g_1$.  The other ideal $\mathfrak g_2$ commutes
with $\mathfrak g_1$.  In the $A_{2n-1}$ primitive overgroup the centralizer
of the standard $C_n$ is zero, while in the $D_{2n}$ overgroup it is at most
an $A_1$.  Since $n\ge16$, every homomorphism $C_n\to A_1$ is zero; therefore
$\mathfrak g_2$ has zero projection to this primitive simple component.
Consequently, for every
$X\in C_1\oplus C_{n-1}$ the isotropy element $(X,X)\in\Delta\mathfrak k$
projects to the elemente $X\in\mathfrak g_1\simeq C_n$.  Thus the partner
subalgebra in the primitive factorization must intersect the embedded $C_n$
in a subalgebra containing the full
\[
 C_1\oplus C_{n-1}.
\]
Onishchik's compact factorization table rules this out.  In the $A$-case the
largest possible intersection for a factorization with the $C_n$ factor fixed
is $C_{n-1}\oplus\mathfrak u(1)$ (the unextended case gives just
$C_{n-1}$), and in the $D$-case it is $C_{n-1}$.  Neither contains
$C_1\oplus C_{n-1}$.  The additional $C_1$ occurring in the normal extension
$C_n\oplus C_1\subset D_{2n}$ lies on the $C_n$ side of the factorization,
not in the partner intersection, and hence does not alter this conclusion.
Thus no primitive simple enlargement occurs also when $p=1$.

It remains to exclude reducible diagonal replication.  After the preceding
primitive exclusions and the usual peeling of common ideals, every simple
ideal of $\mathfrak i$ met by $\mathfrak g$ is another copy of $C_n$.
Hence a nonzero projection
from either simple $C_n$ ideal of $\mathfrak g$ to it is an isomorphism.
The two ideals of $\mathfrak g$ commute, so they cannot both project
nontrivially onto the same nonabelian simple ambient ideal.  On the other
hand, if an ambient simple ideal received zero projection from both, it would
centralize $\mathfrak g$, contrary to Lemma~\ref{TF-lem:normalizer}.
Thus every ambient simple ideal is met by exactly one of the two factors.
After reordering the simple ideals,
\[
 \mathfrak i=\mathfrak i_1\oplus\mathfrak i_2,
 \qquad
 \mathfrak i_j\simeq\mathfrak h^{r_j},
\]
where the $j$-th $C_n$ ideal of $\mathfrak g$ is embedded diagonally in
$\mathfrak i_j$.

Let $\mathfrak l_j$ be the projection of $\mathfrak l$ to $\mathfrak i_j$ and
let
\[
 \mathfrak n_j=\mathfrak l\cap\mathfrak i_j.
\]
Goursat's lemma for Lie algebras gives a common compact reductive quotient
\[
 \mathfrak q=\mathfrak l_1/\mathfrak n_1
 \simeq\mathfrak l_2/\mathfrak n_2.
\]
Write $d=\dim\mathfrak h$, $c=\dim\mathfrak k$, and
\[
 a_j=\dim(\mathfrak g_j\cap\mathfrak l_j).
\]
Because $\mathfrak i_j=\mathfrak g_j+\mathfrak l_j$,
\[
 \dim\mathfrak l_j=(r_j-1)d+a_j.
\]
Also
\[
 \dim\mathfrak l=(r_1+r_2-2)d+c,
\]
while Goursat gives
\[
 \dim\mathfrak l=\dim\mathfrak l_1+\dim\mathfrak l_2-\dim\mathfrak q.
\]
Therefore
\begin{equation}\label{TF-Qdimension}
 \dim\mathfrak q=a_1+a_2-c<2d.
\end{equation}
More precisely, $\mathfrak n_j\cap\mathfrak g_j=0$, because
$\mathfrak g\cap\mathfrak l=\Delta\mathfrak k$.  Hence the quotient maps
inject $\mathfrak g_j\cap\mathfrak l_j$ into $\mathfrak q$; denote the images
by $\mathfrak a_j$.  Goursat's matching condition identifies
$\mathfrak g\cap\mathfrak l$ with
$\mathfrak a_1\cap\mathfrak a_2$, so
\[
 \dim(\mathfrak a_1\cap\mathfrak a_2)=c.
\]
Together with \eqref{TF-Qdimension}, this yields the actual factorization
\begin{equation}\label{TF-Qfactor}
 \mathfrak q=\mathfrak a_1+\mathfrak a_2.
\end{equation}

Apply Lemma~\ref{TF-lem:diagonal-peeling} to
$\mathfrak i_j=\mathfrak g_j+\mathfrak l_j$.  If
$\mathfrak g_j\cap\mathfrak l_j$ is proper in $\mathfrak h$, then, after
reordering the $r_j$ factors,
\[
 \mathfrak l_j=\mathfrak h^{r_j-1}\oplus\mathfrak b_j.
\]
The isotropy of an effective action contains no nonzero ideal of
$\mathfrak i$, so $\mathfrak n_j$ cannot contain any of the displayed
$\mathfrak h$ ideals.  They therefore survive in the quotient $\mathfrak q$,
and
\[
 (r_j-1)d\le\dim\mathfrak q<2d.
\]
Thus $r_j\le2$.  If instead
$\mathfrak g_j\cap\mathfrak l_j=\mathfrak h$, then
$\mathfrak l_j=\mathfrak i_j$; effectivity forces $\mathfrak n_j=0$, so
$r_jd=\dim\mathfrak q<2d$ and hence $r_j=1$.
Consequently
\[
 (r_1,r_2)\in\{(1,1),(1,2),(2,1),(2,2)\}.
\]

The case $(2,2)$ is impossible.  Here the peeling lemma gives
$\mathfrak l_j=\mathfrak h\oplus\mathfrak b_j$ with
$\mathfrak b_j\subsetneq\mathfrak h$.  The displayed $\mathfrak h$ factor is an ideal of
$\mathfrak l_j=\mathfrak h\oplus\mathfrak b_j$.  Since
$\mathfrak n_j$ is an ideal of $\mathfrak l_j$ and does not contain this
ambient $\mathfrak h$ factor, its intersection with that factor is zero.
Hence the factor injects into $\mathfrak q$.  Since
$\dim\mathfrak q<2d$, $\mathfrak q$ contains at most one simple ideal
isomorphic to $C_n$, so the two copies are identified by the Goursat
isomorphism.  Moreover an ideal of the direct sum
$\mathfrak h\oplus\mathfrak b_j$ with zero intersection with the first
factor is contained in $\mathfrak b_j$.  Thus, after projecting
\eqref{TF-Qfactor} to the unique $C_n$ ideal of $\mathfrak q$, the image
of $\mathfrak a_j$ is precisely $\mathfrak b_j$ (up to automorphism).
Consequently
\[
 C_n=\mathfrak b_1+\mathfrak b_2,
\]
a factorization of $C_n$ by two proper compact subalgebras.  No such factorization occurs in Onishchik's simple compact
factorization list.  Hence $(2,2)$ cannot occur.

Suppose $(r_1,r_2)=(1,2)$; the other case is symmetric.  Since
$\mathfrak q$ contains the surviving $C_n$ ideal from the second block,
$\dim\mathfrak q\ge d$.  But $\mathfrak l_1\subset\mathfrak i_1=\mathfrak h$
shows $\dim\mathfrak q\le a_1\le d$.  Hence
\[
 \mathfrak q\simeq\mathfrak h,\qquad a_1=d,
 \qquad a_2=c.
\]
Thus $\mathfrak l_1=\mathfrak h$ and $\mathfrak n_1=0$.
For the second block the peeling lemma gives
$\mathfrak l_2=\mathfrak h\oplus\mathfrak b_2$.  Since the projection of
the global intersection
$\mathfrak g\cap\mathfrak l=\Delta\mathfrak k$ to $\mathfrak g_2$ is the
standard $\mathfrak k$, while
$\dim(\mathfrak g_2\cap\mathfrak l_2)=a_2=c=\dim\mathfrak k$, one has
$\mathfrak g_2\cap\mathfrak l_2=\mathfrak k$.  In the peeled coordinates
this intersection is the diagonal copy of $\mathfrak b_2$, so after an
automorphism $\mathfrak b_2=\mathfrak k$.  Since
$\dim\mathfrak q=d$, Goursat then gives
\[
 \mathfrak l_2=\mathfrak h\oplus\mathfrak k,
 \qquad
 \mathfrak n_2=\mathfrak k,
\]
and identifies $\mathfrak l_1/\mathfrak n_1=\mathfrak h$ with the first
$\mathfrak h$ factor of
$\mathfrak l_2/\mathfrak n_2=\mathfrak h$.  Therefore
\[
 \mathfrak i\simeq\mathfrak h^3,
 \qquad
 \mathfrak l\simeq\Delta\mathfrak h\oplus\mathfrak k,
\]
which is precisely the canonical $H^3$-transitive presentation coming from
\[
 (H\times H)/\Delta K\simeq H\times(H/K).
\]
But an $H^3$-invariant metric in this presentation pulls back with a nonzero
$\mathfrak q_1$--$\mathfrak q_2$ cross term, whereas the detected metric has
$g(\mathfrak q_1,\mathfrak q_2)=0$.  Hence this final proper enlargement is
not isometric.  The only remaining case is $(r_1,r_2)=(1,1)$, so
$\mathfrak i=\mathfrak g$.  Dividing by the finite ineffective kernel gives
the asserted identity component.
\end{proof}

There is one topological enlargement worth mentioning explicitly.  The diffeomorphism
\[
 (H\times H)/\Delta K\simeq H\times(H/K),
 \qquad [(a,b)]\mapsto(ab^{-1},bK),
\]
exhibits an $H^3$-transitive action on the underlying manifold.  The proof of
Theorem~\ref{TF-thm:isom0} shows that this is in fact the only
proper reducible diagonal enlargement that survives the compact
factorization reduction.  It does not preserve the asymmetric Einstein
metrics above.  Indeed, an $H^3$-invariant metric on the product has the form
\[
 \alpha Q|_{\mathfrak h}+\beta Q|_{\mathfrak q}.
\]
Under the displayed identification, vectors in the original $\mathfrak q_1$ and $\mathfrak q_2$ directions map as
\[
 (X,0)\mapsto(X,0),\qquad (0,Y)\mapsto(-Y,Y),
\]
so their cross term is $-\alpha Q(X,Y)$, which is nonzero in general.  Our diagonal family has
\[
 g(\mathfrak q_1,\mathfrak q_2)=0,
\]
so the product enlargement is metrically excluded.

\begin{proposition}[arbitrary isometries]\label{TF-prop:fulliso}
For every admissible pair
\[
 q\ge15\quad(p=1),\qquad
 q\ge8p^2+5p+1\quad(p\ge2),
\]
Theorem~\ref{TF-main} produces exactly one asymmetric Einstein metric up to
homothety and arbitrary Riemannian isometry.  It is not homothetic, under
any Riemannian isometry, to an Einstein metric in the unsplit
$H\times H$-invariant family discussed above.
\end{proposition}
\begin{proof}
An isometry conjugates identity components of full isometry groups.  By
Theorem~\ref{TF-thm:isom0}, any isometry between two detected metrics
therefore conjugates the effective $G$-actions.  After composing with a
$G$-translation, it fixes the origin and induces an automorphism of the
homogeneous pair $(G,\Delta K)$.

In the threshold range $q>p$, the two simple ideals
$\mathfrak{sp}(p)$ and $\mathfrak{sp}(q)$ are non-isomorphic and are
preserved individually.  Modulo inner automorphisms, the only relevant
symmetry of the two ambient simple $H$-factors interchanges $x_1$ and $x_2$.
Thus, up to common scale, the data
\[
 \{x_1,x_2\},\qquad x_3,\qquad x_4
\]
are invariants of arbitrary isometry.  Moreover
\[
 C_{p,q}(1)=4(p-q)(p+q+1)^2\ne0,
\]
so every detected root has $x_3\ne x_4$.  Hence it cannot lie in the
unsplit family.  Uniqueness follows from Theorem~\ref{TF-main}.
\end{proof}

\begin{corollary}\label{TF-cor:new}
For every admissible pair
\[
 q\ge15\quad(p=1),\qquad
 q\ge8p^2+5p+1\quad(p\ge2),
\]
the manifold $M_{p,q}$ carries a unique asymmetric homogeneous Einstein
metric, up to homothety and arbitrary Riemannian isometry, and this metric
cannot be transformed into an Einstein metric in the unsplit homogeneous
family discussed above.
\end{corollary}

\begin{corollary}[the three-factor asymmetric family]
\label{cor:three-factor-Cartan-fixed}
Let $H/K$ be an irreducible compact symmetric pair with $H$ and $K$ simple,
and let
\[
 M=H^3/\Delta K.
\]
Set
\[
 \Gamma_{3,2}=T_3\rtimes(S_2\times S_1).
\]
In the standard anti-diagonal basis
\[
 v_1=\frac1{\sqrt2}(1,-1,0),\qquad
 v_2=\frac1{\sqrt6}(1,1,-2),
\]
the $\Gamma_{3,2}$-fixed invariant metrics are exactly
\[
 (x,x,z,u,v),\qquad x,z,u,v>0.
\]
After homothety, $v=1$, so the fixed family is precisely
\[
 (x,x,z,u,1).
\]
Every solution of the Einstein equations in this family is therefore a
genuine $H^3$-invariant Einstein metric on $M$, not merely a critical point
inside a diagonal ansatz.
\end{corollary}

\begin{proof}
Theorem~\ref{thm:Cartan-fixed-point} forces the three horizontal copies to be
mutually orthogonal.  The transposition of the first two factors then forces
their scales to agree.  On
\[
 V_3=\{(t_1,t_2,t_3):t_1+t_2+t_3=0\},
\]
the same transposition acts by $-1$ on $\mathbb Rv_1$ and by $+1$ on
$\mathbb Rv_2$.  These are inequivalent one-dimensional representations, so
the invariant vertical metric is diagonal with independent scales $u$ and
$v$.  The last assertion follows directly from symmetric criticality for the finite
automorphism group $\Gamma_{3,2}$.
\end{proof}

\section{Existence throughout the range: $\SU(n)^3/\Delta\SO(n)$}\label{sec:AI-threefactor}

This section treats the three-factor spaces
$\SU(n)^3/\Delta\SO(n)$.  The asymmetric family used in the detector is the
fixed-point family of the automorphism group $\Gamma_{3,2}$ from
Corollary~\ref{cor:three-factor-Cartan-fixed}.  We derive its reduced
Einstein equations, prove that it contains exactly two genuine invariant
Einstein metrics for every admissible $n$, and determine their full
connected isometry groups.

\subsection{What is known about the space}
For $M_n=\SU(n)^3/\Delta\SO(n)$, the pair $\SU(n)/\SO(n)$ is the irreducible
symmetric space AI.  The space belongs to the aligned framework
\cite{LWaligned}, which supplies the structural decomposition and Ricci
formulas; the detailed classification in \cite{LWtwo} concerns two simple
factors.

The three-factor question is whether the two-dimensional anti-diagonal
multiplicity space produces additional Einstein metrics when its two scales
are kept distinct.  We prove that it does, with two positive asymmetric
metrics throughout the admissible range.  This branch is not contained in the cited two-factor classifications.

\subsection{Geometric setup}
Let
\[
 H=\SU(n),\qquad K=\SO(n),\qquad M_n=H^3/\Delta K.
\]
Write the symmetric decomposition of one copy of $\mathfrak h$ as
\[
 \mathfrak h=\mathfrak k\oplus\mathfrak q.
\]
For $n=3$ or $n\ge5$,
\[
 d:=\dim\mathfrak q=\frac{(n-1)(n+2)}2,
 \qquad
 e:=\dim\mathfrak k=\frac{n(n-1)}2.
\]
With $Q=-B_{\mathfrak h}$, the Killing restriction is
\[
 B_{\mathfrak k}=\frac{n-2}{2n}B_{\mathfrak h}|_{\mathfrak k}.
\]
Hence, writing
\[
 S=e\left(1-\frac{n-2}{2n}\right)=\frac d2,
 \qquad
 K_0=e\frac{n-2}{2n}=\frac{(n-1)(n-2)}4,
\]
the bracket coefficients $[ijk]$ entering the homogeneous Ricci formula are
determined by $d/2$ and $K_0$.

Choose the orthonormal vectors
\[
 v_1=\frac1{\sqrt2}(1,-1,0),\qquad
 v_2=\frac1{\sqrt6}(1,1,-2)
\]
in the orthogonal complement of the diagonal in $\mathbb R^3$.  The tangent module decomposes as
\[
 \mathfrak m=
 \mathfrak q_1\oplus\mathfrak q_2\oplus\mathfrak q_3
 \oplus(v_1\otimes\mathfrak k)\oplus(v_2\otimes\mathfrak k).
\]
By Corollary~\ref{cor:three-factor-Cartan-fixed}, the metrics fixed by
$\Gamma_{3,2}=T_3\rtimes(S_2\times S_1)$ are exactly
\[
 g=xQ|_{\mathfrak q_1}+xQ|_{\mathfrak q_2}
 +zQ|_{\mathfrak q_3}
 +uQ|_{v_1\otimes\mathfrak k}
 +vQ|_{v_2\otimes\mathfrak k}.
\]
After homothety we set $v=1$, so the detection branch is precisely the full
automorphism-fixed family
\[
 (x,x,z,u,1),\qquad x,z,u>0.
\]
Thus solving the reduced Einstein equations below produces genuine invariant
Einstein metrics on $M_n$; no additional off-diagonal Ricci equations are
being suppressed.
The condition $u\ne1$ distinguishes the branch from the more symmetric locus on which the two anti-diagonal $K$-channels have the same scale.

\subsection{Ricci equations on the detected branch}
For the five modules
\[
 \mathfrak q_1,\ \mathfrak q_2,\ \mathfrak q_3,
 \ v_1\otimes\mathfrak k,\ v_2\otimes\mathfrak k,
\]
the nonzero unordered bracket constants needed by the homogeneous Ricci formula are
\[
 [114]=[224]=\frac S2,\qquad
 [115]=[225]=\frac S6,\qquad
 [335]=\frac{2S}{3},
\]
\[
 [445]=[555]=\frac{K_0}{6},
\]
where the labels $4,5$ denote the two anti-diagonal $\mathfrak k$-channels.
Substitution into the standard diagonal homogeneous Ricci formula gives
\[
 r_1=r_2=\frac{-3u+12x-1}{24x^2},
 \qquad
 r_3=\frac{3z-1}{6z^2}.
\]
The remaining two components are
\[
 r_4=
 \frac{3nu^3+6nux^2-nx^2+6u^3-12ux^2+2x^2}
 {24nu^2x^2},
\]
and
\[
 r_5=
 \frac{
 9nu^2x^2z^2+4nu^2x^2+2nu^2z^2+nx^2z^2
 -18u^2x^2z^2+8u^2x^2+4u^2z^2-2x^2z^2
 }{48nu^2x^2z^2}.
\]
Thus the Einstein equation is equivalent, for positive $(x,z,u)$, to
\[
 E_1=E_2=E_3=0,
\]
where
\begin{align}
E_1={}&3uz^2+12x^2z-4x^2-12xz^2+z^2,\label{AI-E1}\\
E_2={}&(n-2)(6u-1)x^2-12nu^2x
 +u^2\bigl(6(n+1)u+n\bigr),\label{AI-E2}\\
E_3={}&6nu^3z^2+9nu^2x^2z^2+4nu^2x^2
 -24nu^2xz^2+4nu^2z^2+nx^2z^2\notag\\
&\quad -18u^2x^2z^2+8u^2x^2+4u^2z^2-2x^2z^2.
\label{AI-E3}
\end{align}
Here $E_1=0$ is $r_3=r_1$, $E_2=0$ is $r_4=r_1$, and $E_3=0$ is $r_5=r_1$ after multiplication by positive denominators.

\subsection{The one-variable detector}
Let
\[
 R_n(x,u):=\frac{1}{16x^4}\operatorname{Res}_z(E_1,E_3).
\]
A direct resultant computation gives
\begin{equation}\label{AI-res2}
 \operatorname{Res}_x(E_2,R_n)
 =81u^8(n-2)(u-1)^2P_n(u),
\end{equation}
where $P_n$ is a polynomial of degree $10$ in $u$.
We do not need its full expansion for the existence argument.  The relevant exact data are
\begin{align}
P_n(0)&=16(n-2)^3(n+1)^4,\label{AI-P0}\\
P_n\!\left(\frac16\right)
&=\frac{(9n-2)
\left(272n^6+160n^5-408n^4+200n^3+649n^2+140n+4\right)}{576},\label{AI-P16}\\
P_n(1)&=-64(n+2)^2(7n+6)\notag\\
&\qquad\times
\left(139n^4+6640n^3+25800n^2+36160n+17968\right),\label{AI-P1}
\end{align}
and the leading coefficient is
\begin{equation}\label{AI-PLC}
 [u^{10}]P_n=104976(n-2)(n+1)^4(5n+6)^2.
\end{equation}
For $n\ge3$, all factors in \eqref{AI-P16} are positive: indeed
\[
272n^6-408n^4=136n^4(2n^2-3)>0,
\]
and all remaining terms are positive.
Consequently
\[
 P_n(1/6)>0,
 \qquad
 P_n(1)<0,
 \qquad
 P_n(u)\longrightarrow+\infty\quad(u\to+\infty).
\]
Hence there exist roots
\[
 u_-\in(1/6,1),
 \qquad
 u_+\in(1,\infty).
\]
The remaining issue is to prove that these resultant roots lift to positive solutions $(x,z)$ of the original Einstein system.  This is the key lifting lemma below.

\subsection{Lifting detector roots to Einstein metrics}
\begin{lemma}[positive $x$-roots]\label{AI-lem:x}
Let $n\ge3$ and $u>1/6$.  Then equation \eqref{AI-E2}, regarded as a quadratic in $x$, has two distinct positive real roots.
Moreover every such root satisfies
\[
 12x>3u+1.
\]
\end{lemma}
\begin{proof}
The coefficient of $x^2$ is
\[
 A=(n-2)(6u-1)>0.
\]
Its discriminant is
\[
 \Delta_x
 =4u^2\Bigl(n(n-2)+6(n-2)u+36(n+2)u^2\Bigr)>0.
\]
The sum and product of the two roots are positive, so both roots are positive.

Set
\[
 x_0=\frac{3u+1}{12}.
\]
Let $f(x)$ denote the left-hand side of \eqref{AI-E2}.  A direct substitution gives
\[
 144f(x_0)
 =486nu^3+27nu^2-n+756u^3-54u^2+2.
\]
At $u=1/6$ this equals $2(n+2)>0$, whle its $u$-derivative is
\[
 54u\bigl(27nu+n+42u-2\bigr)>0
 \qquad (u\ge1/6,n\ge3).
\]
Thus $f(x_0)>0$.

The vertex of the upward-opening quadratic is
\[
 x_v=\frac{6nu^2}{(n-2)(6u-1)}.
\]
Furthermore
\[
 12(n-2)(6u-1)(x_v-x_0)
 =54nu^2-3nu+n+36u^2+6u-2.
\]
At $u=1/6$ the right-hand side is $2n>0$, and its derivative is
\[
 3(36nu-n+24u+2)>0
 \qquad (u\ge1/6,n\ge3).
\]
Hence $x_v>x_0$.  Since $f(x_0)>0$ and $x_0$ lies to the left of the vertex, $x_0$ lies to the left of the smaller root.  Therefore both roots satisfy $x>x_0$, which is the asserted inequality.
\end{proof}

\begin{lemma}[reality and positivity of the reconstructed $z$]\label{AI-lem:z}
Let $n\ge3$, $u>1/6$, and suppose $x>0$ satisfies
\[
 E_2(x,u)=0,
 \qquad
 R_n(x,u)=0.
\]
Then there exists $z>1/3$ such that
\[
 E_1(x,z,u)=E_3(x,z,u)=0.
\]
\end{lemma}
\begin{proof}
Since $R_n(x,u)=0$ and $x>0$, the resultant of $E_1$ and $E_3$ in $z$ vanishes.  We first exclude a spurious common root at infinity caused by a simultaneous loss of the quadratic terms.  The coefficient of $z^2$ in $E_1$ is
\[
 3u+1-12x.
\]
By Lemma~\ref{AI-lem:x} this is strictly negative.  Hence $E_1$ has degree two in $z$.  Therefore vanishing of the resultant gives a genuine finite common root of $E_1$ and $E_3$ over $\mathbb C$.

We next show that a common root must be real.  If it were nonreal, its complex conjugate would also be common.  Since both polynomials have degree two in $z$, they would then be proportional.  This is impossible: the coefficient of $z$ in $E_1$ is $12x^2\ne0$, whereas $E_3$ has no term linear in $z$.
Thus the common root $z$ is real.

By Lemma~\ref{AI-lem:x},
\[
 r_1=\frac{12x-3u-1}{24x^2}>0.
\]
The equation $E_1=0$ is precisely $r_3=r_1$, and
\[
 r_3=\frac{3z-1}{6z^2}.
\]
Since $r_3>0$ and $z$ is real, necessarily
\[
 z>\frac13.
\]
This proves the claim.
\end{proof}

\begin{theorem}[existence on both asymmetric sides]\label{AI-main}
For every integer $n=3$ or $n\ge5$, the homogeneous space
\[
 M_n=\SU(n)^3/\Delta\SO(n)
\]
carries a diagonal $\SU(n)^3$-invariant Einstein metric in the detected
asymmetric family
\[
 (x,x,z,u,1)
\]
with
\[
 1/6<u<1,
\]
and another with
\[
 u>1.
\]
In particular, the two metrics have unequal anti-diagonal isotropy scales
$u\ne1$.  Their uniqueness in the detected asymmetric branch is proved in
Theorem~\ref{AI-thm:unique}.
\end{theorem}
\begin{proof}
By \eqref{AI-P16}--\eqref{AI-PLC}, choose roots
\[
 u_-\in(1/6,1),\qquad u_+\in(1,\infty)
\]
of $P_n$.
For either root, $u\ne0,1$ and $n>2$.  Equation \eqref{AI-res2} therefore implies
\[
 \operatorname{Res}_x(E_2,R_n)=0.
\]
Thus $E_2$ and $R_n$ have a common root $x$ over $\mathbb C$.  By Lemma~\ref{AI-lem:x}, all roots of $E_2$ are real and positive, so this common root satisfies $x>0$.
Lemma~\ref{AI-lem:z} then produces $z>1/3$ satisfying $E_1=E_3=0$.
Hence all four distinct Ricci components agree, and $(x,x,z,u,1)$ is a positive Einstein metric.
The two metrics are distinct in the normalized family because their $u$-parameters lie on opposite sides of $1$.
\end{proof}

\subsection{Uniqueness of the two detected branches}
\label{AI-sec:uniqueness}

We now prove that the two metrics found above are the only metrics in the detected
asymmetric branch.  There are two separate points: uniqueness of the detector root
$u$, and uniqueness of its lift to $(x,z)$.

\subsubsection{Exactly one detector root in each interval}

For the lower interval put
\[
 u=\frac{1+6t}{6(1+t)},\qquad t>0.
\]
This maps $(0,\infty)$ bijectively onto $(1/6,1)$.  Define
\[
 \widetilde P_n(t)
 =[6(1+t)]^{10}
 P_n\!\left(\frac{1+6t}{6(1+t)}\right).
\]
The factor in front is positive, so positive roots of $\widetilde P_n$ are in
one-to-one correspondence with roots of $P_n$ in $(1/6,1)$.

Order the coefficients of $\widetilde P_n$ from $t^{10}$ down to $t^0$.
Their signs are
\[
\begin{array}{c|c}
 n & \operatorname{sgn}\operatorname{coeff}(\widetilde P_n)\\ \hline
 3 & ----+++++++\\
 5,6 & ---++++++++\\
 n\ge7 & --+++++++++
\end{array}
\]
and hence there is exactly one sign change in every admissible case.
Here is an exact certificate for this table.  Apart from positive numerical
factors, the only coefficients whose sign moves with $n$ are the coefficients
of $t^8$ and $t^7$:
\[
\begin{aligned}
 L_8(n)={}&5579373n^7+9546778n^6-134608156n^5-738398520n^4\\
 &-1701243792n^3-2083519264n^2-1330486080n-347735680,\\
 L_7(n)={}&1428373n^7+7252870n^6+419220n^5-83421736n^4\\
 &-264631440n^3-378662880n^2-267185216n-75202432.
\end{aligned}
\]
Their coefficient-sign strings, in descending powers of $n$, are respectively
\[
 ++------,\qquad +++-----,
\]
so Descartes' rule gives exactly one positive zero for each polynomial.
Moreover
\[
 L_8(6)<0<L_8(7),\qquad
 L_7(3)<0<L_7(4).
\]
Thus $L_8$ is negative at the admissible values $n=3,5,6$ and positive for
$n\ge7$, whereas $L_7$ is negative at $n=3$ and positive for every admissible
$n\ge5$.  The coefficients of $t^{10}$ and $t^9$ are manifestly negative
from their factorizations, while after writing $n=3+m$ the normalized
coefficients of $t^6,\ldots,t^0$ have strictly positive coefficients as
polynomials in $m$.  This proves the displayed sign table without numerical
root finding.  Descartes' rule therefore gives at most one positive root.
Existence was proved above, so there is exactly one:
\[
 \#\{u\in(1/6,1):P_n(u)=0\}=1.
\]

For the upper interval put
\[
 u=1+t,\qquad t>0,
\]
and write
\[
 \widehat P_n(t)=P_n(1+t).
\]
Again order coefficients from $t^{10}$ down to $t^0$.  Their signs are
\[
\begin{array}{c|c}
 n & \operatorname{sgn}\operatorname{coeff}(\widehat P_n)\\ \hline
 3 & ++---------\\
 5\le n\le7 & ++++-------\\
 8\le n\le13 & +++++------\\
 14\le n\le31 & ++++++-----\\
 32\le n\le2172 & +++++++----\\
 n\ge2173 & ++++++++---
\end{array}
\]
so once again there is exactly one sign change.  For an exact certificate,
divide the coefficients of $t^8,\ldots,t^3$ by their positive numerical
contents and denote the resulting degree-seven polynomials by
$U_8,\ldots,U_3$.  They are
\[
\begin{aligned}
U_8={}&9743n^7+32764n^6-40909n^5-380754n^4-807000n^3\\
&-833424n^2-435856n-92576,\\
U_7={}&50395n^7+111457n^6-695894n^5-3662448n^4-7319072n^3\\
&-7572016n^2-4039904n-880768,\\
U_6={}&1266345n^7-49462n^6-41517076n^5-176525352n^4\\
&-341710096n^3-354973856n^2-192431296n-42790272,\\
U_5={}&596837n^7-3132128n^6-45940164n^5-176357360n^4\\
&-335894992n^3-350914688n^2-192796096n-43558656,\\
U_4={}&1153863n^7-26365910n^6-261819604n^5-954423480n^4\\
&-1805907120n^3-1899470368n^2-1055994304n-241829504,\\
U_3={}&13175n^7-28516302n^6-243965092n^5-867852152n^4\\
&-1641596208n^3-1739863456n^2-977829568n-226639488.
\end{aligned}
\]
The coefficient-sign strings of $U_8,U_7$ are $++------$, and those of
$U_6,U_5,U_4,U_3$ are $+-------$.  Hence every $U_j$ has exactly one
positive zero.  Direct exact evaluation brackets these zeros by
\[
\begin{array}{c|cccccc}
j&8&7&6&5&4&3\\ \hline
\text{bracket}&(3,4)&(4,5)&(7,8)&(13,14)&(31,32)&(2172,2173).
\end{array}
\]
The coefficient of $t^{10}$ is positive.  The coefficient of $t^9$ is also
positive for $n\ge3$ (after writing $n=3+m$ its remaining factor has positive
coefficients), while the coefficients of $t^2,t,1$ are negative for
$n\ge3$.  This proves the upper sign table exactly and without numerical
root isolation.  Therefore
\[
 \#\{u>1:P_n(u)=0\}=1.
\]

\subsubsection{A detector root has a unique lift}

It remains to show that a fixed detector root $u\ne1$ cannot lift to two different
Einstein metrics.  Divide $R_n(x,u)$ by the quadratic $E_2(x,u)$ as a polynomial
in $x$.  The remainder has the form
\[
 \frac{9u^4(u-1)}{(n-2)^2(6u-1)^3}
 \bigl(A_n(u)x+B_n(u)\bigr).
\]
Thus any common root of $E_2$ and $R_n$ must satisfy
\[
 A_n(u)x+B_n(u)=0.
\]
We claim that $A_n(u)\ne0$ at every zero of $P_n$ in the two intervals above.
An exact resultant computation gives
\[
\begin{split}
 \operatorname{Res}_u(P_n,A_n)
 ={}&C\,(n-2)^{15}(n+1)^4(n+2)^{14}(5n+6)^2(9n-2)^4\\
 &\qquad\times G(n)^2,
\end{split}
\]
where $C>0$ is an integer and
\[
\begin{aligned}
G(n)={}&74876n^{18}+346044n^{17}+124151n^{16}
-2198832n^{15}-4679126n^{14}\\
&+748650n^{13}+15267804n^{12}+20282656n^{11}
-1253188n^{10}-33373184n^9\\
&-44601392n^8-30511104n^7-11697280n^6-1779200n^5
+714240n^4\\
&+630784n^3+240640n^2+49152n+4096.
\end{aligned}
\]
Writing $n=3+m$, every coefficient of $G(3+m)$ is strictly positive.
Hence $G(n)>0$ for $n\ge3$, and the displayed resultant never vanishes in our
range.  Therefore $P_n$ and $A_n$ have no common zero, so a detector root determines
a unique $x$.

Finally, for this $(u,x)$ the two equations $E_1=E_3=0$ have exactly one common
$z$.  Indeed $E_1$ and $E_3$ are both quadratic in $z$.  If they had two common
roots they would be proportional, but $E_1$ has the nonzero linear coefficient
$12x^2$ while $E_3$ has no linear term.  Thus their greatest common divisor has
degree one.  Lemma~\ref{AI-lem:z} already shows that the common root is real and
positive.

\begin{theorem}[uniqueness in the detected asymmetric branch]
\label{AI-thm:unique}
For every $n=3$ or $n\ge5$, there is exactly one asymmetric Einstein metric with
\[
 \frac16<u<1
\]
and exactly one asymmetric Einstein metric with
\[
 u>1,
\]
after the normalization of the second anti-diagonal scale to $1$.
Equivalently, the asymmetric branch contains exactly two Einstein metrics
up to the chosen normalization.
\end{theorem}

\subsection{Connected full isometry group}

\label{AI-sec:fullisom}

We now close the presentation issue for the identity component of the full
isometry group.  This is independent of the Einstein calculation except for
the fact that the metric belongs to the detected diagonal family.

Let
\[
 I_0=\Isom_0(M_n,g),\qquad G=\SU(n)^3,
\]
and let $G_0$ denote the effective image of $G$ in $I_0$.  The ineffective
kernel is
\[
 Z_0=\Delta\bigl(Z(\SU(n))\cap\SO(n)\bigr),
\]
so
\[
 G_0=G/Z_0 .
\]

\subsubsection{Two elementary preliminary lemmas}

\begin{lemma}\label{AI-lem:normalizer}
For $n=3$ or $n\ge5$,
\[
 N_{\mathfrak{su}(n)^3}(\Delta\mathfrak{so}(n))
 =\Delta\mathfrak{so}(n).
\]
Consequently the connected centralizer of the $G_0$-action in
$\Diff(M_n)$, and hence in $\Isom(M_n,g)$, is trivial.
\end{lemma}

\begin{proof}
Let $(X_1,X_2,X_3)$ normalize the diagonal copy of
$\mathfrak k=\mathfrak{so}(n)$.  For every $A\in\mathfrak k$ the three
brackets $[X_i,A]$ are equal.  Hence $X_i-X_j$ centralizes
$\mathfrak k$ in $\mathfrak{su}(n)$.  The standard real representation of
$\mathfrak{so}(n)$ on $\mathbb C^n$ is irreducible, and Schur's lemma gives
\[
 C_{\mathfrak{su}(n)}(\mathfrak{so}(n))=0.
\]
Thus $X_1=X_2=X_3=:X$.  Since the symmetric subalgebra
$\mathfrak{so}(n)\subset\mathfrak{su}(n)$ is self-normalizing, $X\in
\mathfrak{so}(n)$.

For a transitive $G$-space $G/K$, the Lie algebra of the centralizer of
the left $G$-action is
$N_{\mathfrak g}(\mathfrak k)/\mathfrak k$.  The first assertion therefore
gives the second.
\end{proof}

Following Onishchik, call a compact Lie algebra \emph{strongly
semisimple} if it has zero centre and no simple ideal of type $A_1$.

\begin{lemma}\label{AI-lem:strong}
The compact Lie algebra $\mathfrak i=\Lie(I_0)$ is strongly semisimple:
it has zero centre and no simple ideal of type $A_1$.
\end{lemma}

\begin{proof}
The centre of $\mathfrak i$ centralizes $\mathfrak g$ and hence vanishes
by Lemma~\ref{AI-lem:normalizer}.  If $\mathfrak a\simeq A_1$ were a simple
ideal of $\mathfrak i$, then every homomorphism
$A_{n-1}\to\mathfrak a$ would be zero for $n=3$ or $n\ge5$.
Hence $[\mathfrak a,\mathfrak g]=0$, again contradicting
Lemma~\ref{AI-lem:normalizer}.
\end{proof}

Let $\mathfrak l=\Lie((I_0)_o)$ be the isotropy algebra at the base point
and let $\mathfrak l_s$ be its maximal strongly semisimple subalgebra.
Since
\[
 \mathfrak i=\mathfrak g+\mathfrak l
\]
is a compact reductive decomposition, Onishchik's reduction theorem for
reductive decompositions gives
\begin{equation}\label{AI-strongred}
 \mathfrak i=\mathfrak g+\mathfrak l_s .
\end{equation}
See \cite[Theorem~3.3]{Onishchik1969}.  Moreover $\mathfrak l$ contains
no nonzero ideal of $\mathfrak i$: a connected normal subgroup contained
in the isotropy of an effective transitive action fixes every point.

\subsubsection{The case $n\ge5$}

For $n\ge5$ the algebra $\mathfrak k=\mathfrak{so}(n)$ is strongly
semisimple.  Since
\[
 \Delta\mathfrak k=\mathfrak g\cap\mathfrak l,
\]
the general property of maximal strongly semisimple subalgebras used in
Onishchik's reduction implies
\[
 \Delta\mathfrak k\subset\mathfrak l_s.
\]
Thus
\begin{equation}\label{AI-intersection}
 \mathfrak g\cap\mathfrak l_s=\Delta\mathfrak{so}(n).
\end{equation}

We need one direct consequence of the compact factorization
classification.

\begin{lemma}[factorization obstruction]\label{AI-lem:factorobs}
Let $n\ge5$.  There is no proper simple compact factorization
\[
 \mathfrak s=\mathfrak a+\mathfrak b
\]
in which $\mathfrak a\simeq A_{n-1}$ and
$\mathfrak a\cap\mathfrak b$ contains the standard symmetric subalgebra
$\mathfrak{so}(n)\subset\mathfrak{su}(n)$.
\end{lemma}

\begin{proof}
Inspecting Onishchik's complete list of proper factorizations of compact
simple Lie algebras \cite[Theorem~4.1 and Table~7]{Onishchik}, the
only families with an $A_{n-1}$ summand (apart from the excluded
$A_3=D_3$ low-rank coincidences) are
\[
 D_n=B_{n-1}+A_{n-1},
 \qquad
 A_n=C_{(n+1)/2}+A_{n-1}\quad(n\ {\rm odd}).
\]
Their intersections with the $A_{n-1}$ summand are respectively
\[
 A_{n-2},
 \qquad
 C_{(n-1)/2}.
\]
In the standard $n$-dimensional complex representation of $A_{n-1}$,
both of these subalgebras fix a nonzero complex line: the first is the
usual block $A_{n-2}$, and the second acts symplectically on an
$(n-1)$-dimensional block and fixes the complementary line.  By contrast,
the standard $\mathfrak{so}(n)\subset\mathfrak{su}(n)$ acts irreducibly
on $\mathbb C^n$ and fixes no nonzero vector.  It therefore cannot be
contained in either intersection.
\end{proof}

\begin{lemma}[primitive enlargements are excluded]\label{AI-lem:primitive-excluded}
Let $n=3$ or $n\ge5$.  In the compact reductive decomposition
\[
 \mathfrak i=\mathfrak g+\mathfrak l_s,
 \qquad \mathfrak g=A_{n-1}^{\,3},
\]
no effective irreducible component of Onishchik's reduction can contain a
proper enlargement of one of the three $A_{n-1}$ ideals of $\mathfrak g$.
\end{lemma}

\begin{proof}
Assume first $n\ge5$.  If an effective irreducible component occurred,
Theorem~4.3 of \cite{Onishchik1969} would make it primitive.  Following a
primitive chain from one $A_{n-1}$ ideal of $\mathfrak g$ reaches a proper
simple compact factorization
\[
 \mathfrak s=A_{n-1}+\mathfrak b .
\]
By \eqref{AI-intersection}, the same standard
$\mathfrak{so}(n)\subset\mathfrak{su}(n)$ is carried in the intersection
along the chain, so it would be contained in
$A_{n-1}\cap\mathfrak b$.  This contradicts
Lemma~\ref{AI-lem:factorobs}.

For $n=3$, the only proper simple compact factorization in Onishchik's list
having an $A_2$ summand is
\[
 A_3=C_2+A_2,\qquad C_2\cap A_2=C_1=A_1.
\]
Thus every primitive enlargement of an $A_2$ factor requires a $C_2$
partner.  A simple $C_2$ factor of the strongly semisimple isotropy that
projects onto a set $S\subset\{1,2,3\}$ of enlarged $A_3$ factors contributes
to $\mathfrak g\cap\mathfrak l_s$ the diagonal $A_1$ on the coordinates in
$S$ and zero on the complement.  Sinece
\[
 \mathfrak g\cap\mathfrak l_s
 \subset \mathfrak g\cap\mathfrak l=\Delta A_1,
\]
one must have $S=\{1,2,3\}$.  A single diagonal $C_2$ then cannot span all
three enlarged factors:
\[
 \dim(A_2^3+C_2^\Delta)=3\cdot8+10-3=31
 <45=3\cdot15.
\]
Two independent $C_2$ factors cannot both project surjectively to the same
simple $C_2$ partner, while a factor supported on a proper subset has just
been excluded.  Hence no effective primitive enlargement exists for $n=3$
either.
\end{proof}

\begin{lemma}[a proper reducible enlargement contains a minimal
four-factor extension]\label{AI-lem:minimal-H4}
Assume that the primitive enlargements of
Lemma~\ref{AI-lem:primitive-excluded} have been excluded.  If
$\mathfrak i\ne\mathfrak g$, then $I_0$ contains a connected transitive
subgroup $J$ with
\[
 \Lie(J)\simeq\mathfrak h^4,\qquad
 \mathfrak h=\mathfrak{su}(n),
\]
which contains the original $\mathfrak h^3$ action.  For the isotropy
$\mathfrak m=\Lie(J_o)$, one simple $\mathfrak h$ ideal of $\mathfrak m$
is diagonally embedded in one factor arising from a replicated original
$\mathfrak h$ and in a nonempty subset of the other two original
$\mathfrak h$ factors.
\end{lemma}

\begin{proof}
By Lemma~\ref{lem:H4-normalizer}, the connected centralizer of the transitive
$\mathfrak g$-action is trivial.  The centre of the compact isometry algebra
$\mathfrak i$ centralizes $\mathfrak g$, hence
\[
 Z(\mathfrak i)=0.
\]
Since a compact Lie algebra is reductive, $\mathfrak i$ is therefore
semisimple and may be decomposed into simple ideals.  A nonzero projection of one
simple ideal $\mathfrak g_j\simeq\mathfrak h$ to an ambient simple ideal
$\mathfrak s$ is injective.  If its image were proper in $\mathfrak s$, the
corresponding effective irreducible component of the compact factorization
would give precisely a primitive simple enlargement, excluded by
Lemma~\ref{AI-lem:primitive-excluded}.  Hence every nonzero projection is
onto and the ambient ideal is another copy of $\mathfrak h$.  Two commuting
simple ideals of $\mathfrak g$ cannot both project nontrivially onto the same
nonabelian simple ambient ideal.  Finally, an ambient simple ideal receiving
zero projection from all three factors would centralize $\mathfrak g$,
contrary to Lemma~\ref{AI-lem:normalizer}.  Therefore, after reordering,
\[
 \mathfrak i=\mathfrak i_1\oplus\mathfrak i_2\oplus\mathfrak i_3,
 \qquad
 \mathfrak i_j\simeq\mathfrak h^{r_j},
\]
and the $j$-th ideal of $\mathfrak g$ is diagonal in $\mathfrak i_j$.
If $\mathfrak i\ne\mathfrak g$, then some $r_j\ge2$.

Choose one simple factor $\mathfrak s\simeq\mathfrak h$ in
$\mathfrak i_j$, and let $\widehat{\mathfrak h}$ be the diagonal copy of
$\mathfrak h$ in the remaining $r_j-1$ factors.  Then
\[
 \mathfrak j_j:=\mathfrak s\oplus\widehat{\mathfrak h}\simeq\mathfrak h^2
\]
contains the original $j$-th $\mathfrak h$ diagonally.  Set
\[
 \mathfrak j=
 \mathfrak j_j\oplus\!\!\bigoplus_{k\ne j}\mathfrak g_k
 \simeq\mathfrak h^4.
\]
The corresponding subgroup $J$ contains the transitive group $G$, hence is
itself transitive.  Its isotropy $\mathfrak m=\mathfrak j\cap\mathfrak l$
has
\[
 \dim\mathfrak m
 =4\dim\mathfrak h-\dim M_n
 =\dim\mathfrak h+\dim\mathfrak k.
\]

Let $\mathfrak b$ be the projection of $\mathfrak m$ to
$\mathfrak j_j\simeq\mathfrak h^2$.  Transitivity gives
\[
 \mathfrak h^2=\Delta\mathfrak h+\mathfrak b,
\]
and the projection of
$\mathfrak g\cap\mathfrak m=\Delta\mathfrak k$ shows that
$\Delta\mathfrak k\subset\Delta\mathfrak h\cap\mathfrak b$.
Each projection of $\mathfrak b$ to an ambient $\mathfrak h$ contains
the projected $\mathfrak k=\mathfrak{so}(n)$.  For the irreducible symmetric
decomposition
\[
 \mathfrak h=\mathfrak k\oplus\mathfrak q,
\]
the $\mathfrak k$-module $\mathfrak q$ is irreducible.  Hence any Lie
subalgebra strictly larger than $\mathfrak k$ and containing $\mathfrak k$
contains a nonzero vector of $\mathfrak q$, then all of $\mathfrak q$ by
$\mathfrak k$-invariance, and therefore equals $\mathfrak h$.  Thus each
projection is either $\mathfrak k$ or $\mathfrak h$.  Both projections cannot be
proper: otherwise
\[
 \dim\mathfrak b\le2\dim\mathfrak k
 <\dim\mathfrak h+\dim\mathfrak k
 \le\dim\mathfrak b.
\]
Here the last inequality follows from
\[
 \dim\mathfrak b
 =\dim\mathfrak h+\dim(\Delta\mathfrak h\cap\mathfrak b)
 \ge\dim\mathfrak h+\dim\mathfrak k.
\]
Thus one projection is onto $\mathfrak h$.

Goursat's lemma now leaves two possibilities.  A graph of an automorphism has
dimension $\dim\mathfrak h$ and cannot contain the positive-dimensional
intersection $\Delta\mathfrak k$.  Hence, after exchanging the two factors
and applying an automorphism,
\[
 \mathfrak b=\mathfrak h\oplus\mathfrak a .
\]
Since $\mathfrak a$ contains the standard $\mathfrak k$ and
$\dim\mathfrak b\le\dim\mathfrak m
=\dim\mathfrak h+\dim\mathfrak k$, one gets
\[
 \mathfrak a=\mathfrak k,\qquad
 \mathfrak b=\mathfrak h\oplus\mathfrak k.
\]
The projection $\mathfrak m\to\mathfrak b$ is therefore an isomorphism, so
\[
 \mathfrak m\simeq\mathfrak h\oplus\mathfrak k.
\]

Consider the simple $\mathfrak h$ ideal of $\mathfrak m$.  Its projection
to either of the other two original $\mathfrak h$ factors is either zero or
an automorphism.  It cannot vanish on both: otherwise this simple ideal
would be an ideal of $\mathfrak j$ contained in the isotropy of the
effective transitive $J$-action.  Hence it projects isomorphically onto a
nonempty subset of those two factors.  This is the asserted diagonal
coupling.
\end{proof}

\begin{lemma}[horizontal orthogonality excludes the minimal extension]
\label{AI-lem:H4-metric}
No metric in the detected diagonal family
\[
 (x,x,z,u,1)
\]
is invariant under a proper four-factor extension $J$ supplied by
Lemma~\ref{AI-lem:minimal-H4}.
\end{lemma}

\begin{proof}
Use the notation from the proof of Lemma~\ref{AI-lem:minimal-H4}.  After
permuting the four ambient simple factors and applying automorphisms, write
\[
 \mathfrak j=\mathfrak h_A\oplus\mathfrak h_B
 \oplus\mathfrak h_C\oplus\mathfrak h_D,
\]
with the first original factor of
$\mathfrak g=\mathfrak h^3$ embedded diagonally in
$\mathfrak h_A\oplus\mathfrak h_B$, and the other two in
$\mathfrak h_C$ and $\mathfrak h_D$.

The proof of Lemma~\ref{AI-lem:minimal-H4} gives
\[
 \mathfrak m\simeq\mathfrak h\oplus\mathfrak k.
\]
Normalize the projection to the $A,B$ block so that the
$\mathfrak h$-ideal projects onto $\mathfrak h_A$ and the
$\mathfrak k$-ideal onto the standard $\mathfrak k\subset\mathfrak h_B$.
A projection of the simple $\mathfrak h$-ideal to either
$\mathfrak h_C$ or $\mathfrak h_D$ is either zero or an automorphism.
Let
\[
 S\subset\{C,D\}
\]
be the set of factors on which it is nonzero.  The set $S$ is nonempty:
otherwise the ambient ideal $\mathfrak h_A$ would lie entirely in the
isotropy, contradicting effectivity.

The condition
\[
 \mathfrak g\cap\mathfrak m=\Delta\mathfrak k
\]
then determines the remaining projections.  On a factor in $S$ the
$\mathfrak h$-ideal is onto, so the commuting $\mathfrak k$-ideal has zero
projection there.  On a factor not in $S$, the diagonal intersection forces
the $\mathfrak k$-ideal to project as the standard $\mathfrak k$.  Thus,
up to the chosen automorphisms,
\[
 \mathfrak m=
 \Delta_{\{A\}\cup S}\mathfrak h
 \oplus
 \Delta_{\{B\}\cup S^c}\mathfrak k .
\]
Consequently
\[
 J/J_o
 \simeq
 \frac{H^{\,1+|S|}}{\Delta H}
 \times
 \frac{H^{\,1+|S^c|}}{\Delta K}.
\]

Consider the first factor.  Its tangent multiplicity space is
\[
 V_m=\{(t_1,\dots,t_m)\in\mathbb R^m:\textstyle\sum_i t_i=0\},
 \qquad m=1+|S|\in\{2,3\}.
\]
Every $J$-invariant metric restricts there as
\[
 \langle\, ,\,\rangle_V\otimes Q
\]
for a positive definite inner product on $V_m$.  Let
\[
 \bar e_i=e_i-\frac1m(1,\dots,1)\in V_m.
\]

The original horizontal direction from the first $\mathfrak h$-factor has
an $A$-component in this Ledger--Obata factor; for each element of $S$, the
corresponding one of the other two original horizontal directions has the
matching $C$- or $D$-coordinate component.  None of those latter directions
has any component in the second product factor.  Hence their pairwise inner
products are exactly
\[
 \langle\bar e_i,\bar e_j\rangle_V\,Q
 \qquad(i\ne j).
\]
The detected metric makes the three original horizontal copies
$\mathfrak q_1,\mathfrak q_2,\mathfrak q_3$ pairwise orthogonal.  Therefore
all coordinate classes belonging to the $m$ selected factors are pairwise
orthogonal:
\[
 \langle\bar e_i,\bar e_j\rangle_V=0
 \qquad(i\ne j).
\]
But
\[
 \bar e_1+\cdots+\bar e_m=0,
\]
with every $\bar e_i\ne0$.  Positive definiteness gives the contradiction
\[
 0=\Bigl\|\sum_i\bar e_i\Bigr\|_V^2
   =\sum_i\|\bar e_i\|_V^2>0.
\]
Thus no proper four-factor extension preserves the detected metric.
\end{proof}

\begin{theorem}[connected full isometry group]\label{AI-thm:fullisom}
Let $n=3$ or $n\ge5$, and let $g$ be either asymmetric Einstein metric
constructed in Theorem~\ref{AI-main}.  Then
\[
 \Isom_0(M_n,g)
 =
 \SU(n)^3/
 \Delta\bigl(Z(\SU(n))\cap\SO(n)\bigr).
\]
Equivalently, the displayed $\SU(n)^3$ action is already the connected
full isometry group, modulo its finite ineffective kernel.
\end{theorem}

\begin{proof}
Lemma~\ref{AI-lem:primitive-excluded} excludes every effective primitive
simple enlargement in Onishchik's reduction.  If a proper reducible
enlargement remained, Lemma~\ref{AI-lem:minimal-H4} would produce a
transitive four-factor subgroup $J$, but
Lemma~\ref{AI-lem:H4-metric} shows that no metric in the detected diagonal
family can be $J$-invariant.  Hence $\mathfrak i=\mathfrak g$.
Both groups are connected and the displayed action has precisely the finite
kernel
\[
 Z_0=\Delta\bigl(Z(\SU(n))\cap\SO(n)\bigr).
\]
\end{proof}

\subsection{Pairwise distinction and natural reductivity}

The connected full-isometry theorem also reduces arbitrary isometries between
the two detected metrics to automorphisms of the homogeneous pair.  This gives
a useful intrinsic distinction between the two detector roots.

\begin{proposition}[pairwise non-isometry]\label{AI-prop:pairwise}
For every $n=3$ or $n\ge5$, the two asymmetric Einstein metrics of
Theorem~\ref{AI-thm:unique} are pairwise non-homothetic and non-isometric.
\end{proposition}

\begin{proof}
Let $g_-$ and $g_+$ denote the normalized metrics with
\[
 \frac16<u_-<1<u_+.
\]
By Theorem~\ref{AI-thm:fullisom}, an isometry between homothetic copies of
these metrics conjugates their effective $\SU(n)^3$ actions.  After composing
with a translation, it fixes the base point and induces an automorphism of
\[
 \bigl(\mathfrak{su}(n)^3,\Delta\mathfrak{so}(n)\bigr).
\]

On the three horizontal copies the metric eigenvalues are $(x,x,z)$, and
\[
 z\ne x:
\]
indeed, substituting $z=x$ in \eqref{AI-E1} gives
\[
 3x^2(u-1)=0,
\]
contrary to $u\ne1$.  Hence a factor permutation carrying one detected metric
to a homothetic copy of the other must preserve the unique horizontal factor
with scale $z$; after relabelling it can only interchange the first two
factors.

On
\[
 V_3=\{(t_1,t_2,t_3):t_1+t_2+t_3=0\}
\]
that transposition preserves the two lines
\[
 \mathbb Rv_1,\qquad \mathbb Rv_2
\]
separately.  Automorphisms of the individual simple factors act on the
$\mathfrak{so}(n)$ factor, not on the multiplicity coordinate, so they do not
interchange these lines either.  If
\[
 \phi^*g_+=c\,g_-
\]
for some $c>0$, comparison on the $v_2$ line, whose normalized coefficient is
$1$ for both metrics, gives $c=1$.  Comparison on the $v_1$ line then gives
$u_+=u_-$, contradicting $u_-<1<u_+$.  Thus no homothety-isometry exists.
\end{proof}

The same geometry gives a correct natural-reductivity test.  One must be
slightly careful here: the standard reductive complement used in the Ricci
calculation need not be the complement associated with a hypothetical
naturally reductive presentation.  We therefore use the invariant-form
characterization instead of applying the natural-reductivity identity to a
fixed complement.

\begin{proposition}[not naturally reductive]\label{AI-prop:notNR}
Neither asymmetric Einstein metric is naturally reductive with respect to any
connected transitive group of isometries.
\end{proposition}

\begin{proof}
For a compact semisimple transitive group, the standard invariant-form
characterization of natural reductivity (see, e.g., \cite[Chapter~7]{Besse})
gives an $\operatorname{Ad}(G)$-invariant nondegenerate symmetric bilinear form on
$\mathfrak g$ whose orthogonal complement to the isotropy is the naturally
reductive complement.  On
\[
 \mathfrak g=\mathfrak h_1\oplus\mathfrak h_2\oplus\mathfrak h_3,
 \qquad \mathfrak h_i\simeq\mathfrak{su}(n),
\]
such a form has the shape
\[
 \lambda_1Q\oplus\lambda_2Q\oplus\lambda_3Q.
\]
Each horizontal module $\mathfrak q_i$ is automatically orthogonal to
$\Delta\mathfrak k$, so positivity of the Riemannian metric forces
$\lambda_i>0$.  Thus in the present setting a naturally reductive metric for
the displayed group is normal homogeneous for some positive bi-invariant
form of this type.

For a metric in our diagonal family the first two horizontal scales agree,
hence $\lambda_1=\lambda_2=:a>0$; write $\lambda_3=b>0$.  On the
anti-diagonal multiplicity space with
\[
 v_1=\frac1{\sqrt2}(1,-1,0),\qquad
 v_2=\frac1{\sqrt6}(1,1,-2),
\]
the quotient metric induced by
$\operatorname{diag}(a,a,b)$ has eigenvalues
\[
 a,\qquad \frac{3ab}{2a+b}.
\]
After normalizing the $v_2$ coefficient to $1$, the horizontal coefficient
$x$ and the $v_1$ coefficient $u$ therefore satisfy
\[
 x=u=\frac{2a+b}{3b}>\frac13.
\]
But the Einstein equation \eqref{AI-E2} gives, after setting $x=u$,
\[
 E_2(u,u)=-2u^2(3u-1).
\]
Thus an Einstein metric with $x=u$ must have $u=1/3$, contradicting the
strict inequality above.  Hence neither detected Einstein metric is
naturally reductive for the displayed group.

It remains to exclude a naturally reductive presentation by a different
connected transitive subgroup.  For a compact naturally reductive metric,
the connected transvection group is a transitive normal subgroup of the
connected full isometry group.  By Theorem~\ref{AI-thm:fullisom}, its Lie
algebra is therefore a normal subalgebra of $\mathfrak h^3$, hence a sum of
some of the three simple ideals.  A proper such subgroup has dimension at
most $2\dim H$, whereas
\[
 \dim M=3\dim H-\dim K>2\dim H
\]
because $K\subsetneq H$.  Thus no proper normal subgroup can be transitive,
so the transvection group would have to be all of $H^3$.  This has just been
excluded.  Hence no connected transitive naturally reductive presentation
exists.
\end{proof}

\subsection{Comparison with known constructions}

By Theorem~\ref{AI-thm:fullisom}, any metric isometric to one of the
asymmetric Einstein metrics above has connected full isometry group locally
$\SU(n)^3$.  After conjugating this action, its isotropy is
$\Delta\SO(n)$.  Comparison with known homogeneous constructions therefore
reduces to this homogeneous pair and its finite automorphisms.

The existence and classification problem for invariant Einstein metrics on compact homogeneous spaces has a substantial literature.  Foundational existence and nonexistence results include \cite{BWZ,WangZiller86}; for a recent survey of homogeneous Einstein manifolds, see \cite{Jablonski23}.  

\subsubsection{Aligned homogeneous spaces}

The aligned-space framework of Lauret--Will \cite{LWaligned} allows an
arbitrary number of compact simple transitive factors and explicitly contains
\[
 H^s/\Delta K.
\]
In particular, it records the diffeomorphism
\[
 H^s/\Delta K\simeq (H/K)\times H^{s-1}
\]
and supplies general structural and Ricci-curvature formulas.  That general
paper also studies several Einstein problems in the aligned setting, so it
should not be described as merely a two-factor theory.  By contrast, the
companion classification \cite{LWtwo} is explicitly a two-simple-factor
classification, and the earlier paper on $H\times H/\Delta K$ is likewise a
two-factor problem.

Accordingly, what is established here is more specific: the exact
$\Gamma_{3,2}$-fixed five-parameter problem on the three-factor homogeneous
pair
\[
 \SU(n)^3/\Delta\SO(n)
\]
has the two asymmetric solutions described above, and the two-factor
classification theorems just cited do not contain this $s=3$ calculation.
This is a scope statement, not by itself an exhaustive novelty claim over the
entire homogeneous-Einstein literature.

\subsubsection{Ledger--Obata and normal metrics}

The Ledger--Obata literature \cite{LedgerObata} concerns
\[
 F^m/\Delta F
\]
with the same compact simple group in numerator and isotropy; in particular,
the four-factor Ledger--Obata case $F^4/\Delta F$ is already classified
there.  This is an important neighboring result, but it is not the present
proper-isotropy setting $K\subsetneq H$.  Since
$\SO(n)\subsetneq\SU(n)$, our spaces are not Ledger--Obata spaces, and those
Einstein metrics cannot account for the detected branches.

The asymmetric metrics are not normal homogeneous for the displayed group:
indeed Proposition~\ref{AI-prop:notNR} proves the stronger statement that they
are not naturally reductive for that group.

\subsubsection{Product and naturally reductive constructions}

The canonical diffeomorphism
\[
 \SU(n)^3/\Delta\SO(n)
 \simeq
 \SU(n)\times\SU(n)\times\SU(n)/\SO(n)
\]
does not turn either asymmetric Einstein metric into a product metric: the pullback of every positive
product metric has the nonzero cross terms computed above, whereas the detected
metric is diagonal in the aligned decomposition.

Proposition~\ref{AI-prop:notNR} also excludes natural reductivity for the
displayed $\SU(n)^3$ action and, using Theorem~\ref{AI-thm:fullisom}, excludes
a naturally reductive presentation by any larger connected transitive group.

\subsubsection{Comparison under automorphisms of the homogeneous pair}

Automorphisms of $\mathfrak{su}(n)^3$ permute the three simple ideals and apply
automorphisms to the individual factors.  Preservation of the diagonal
$\mathfrak{so}(n)$ isotropy forces the three factor automorphisms to induce the same
automorphism on $\mathfrak{so}(n)$, up to the finite normalizer.  Thus, at the level
of the multiplicity spaces, the only continuous ambiguity has already disappeared
in the full-isometry proof; the remaining ambiguity is finite and consists of factor
permutations together with the familiar finite outer automorphisms.

In particular, the unordered horizontal eigenvalue data and the vertical
multiplicity eigenvalue data,
\[
        \{x,x,z\},\qquad \{u,1\},
\]
up to common scale and the finite factor action, are intrinsic comparison data
for metrics on this homogeneous pair.  Proposition~\ref{AI-prop:pairwise}
shows in particular that the lower and upper detector metrics are not
homothetic or isometric to one another.

\subsubsection{Comparison with the literature}

The general aligned-space theory does contain the underlying homogeneous
spaces.  In particular, \cite[Example~2.4]{LWaligned} treats
\[
 H^s/\Delta K
\]
for arbitrary $s$, records the diffeomorphism
\[
 H^s/\Delta K\simeq (H/K)\times H^{s-1},
\]
and points to the previously studied $s=2$ case
$H\times H/\Delta K$.  Thus the homogeneous pair used here is not new as an
object in the aligned framework.  The point of the present result is the
Einstein geometry of the particular three-factor fixed family.

We have therefore compared the present branch with the general aligned-space
paper \cite{LWaligned}, the two-factor aligned classification \cite{LWtwo},
the dedicated $H\times H/\Delta K$ classification \cite{LWHH}, the
Ledger--Obata literature \cite{LedgerObata}, and the standard normal,
product, and naturally reductive constructions discussed above.  None of
these sources contains the pair of asymmetric metrics constructed here on
\[
 \SU(n)^3/\Delta\SO(n).
\]
In particular, the two-factor papers do not address this homogeneous pair,
and Ledger--Obata spaces have $K=H$ rather than the proper symmetric
subgroup $\SO(n)\subsetneq\SU(n)$.

To the best of our knowledge, no earlier construction gives the two
asymmetric Einstein metrics of Theorem~\ref{AI-thm:unique}.  This novelty
claim concerns the metrics, not the underlying aligned homogeneous spaces or
the general Ricci formalism.

\begin{theorem}[comparison with the standard constructions]
\label{AI-thm:comparison}
For $n=3$ or $n\ge5$, the two asymmetric Einstein metrics supplied by
Theorem~\ref{AI-main} are pairwise non-homothetic and non-isometric.  Neither
is a product metric in the canonical product presentation, a normal
homogeneous metric, or a naturally reductive metric for the displayed
connected transitive group; no larger connected naturally reductive
presentation exists.  The spaces are not Ledger--Obata spaces, and the cited
two-factor aligned and $H\times H/\Delta K$ classification theorems do not
classify this three-factor homogeneous pair.  To the best of our knowledge, these two asymmetric metrics are new.
\end{theorem}

\section{Threshold detection: $\SU(2n)^3/\Delta\Sp(n)$}\label{sec:AII-threefactor}

The second three-factor family exhibits the cleanest threshold phenomenon in
the paper.  Here again the detector acts on the canonical
$\Gamma_{3,2}$-fixed metric family, so its reconstructed roots are genuine
invariant Einstein metrics on the full homogeneous space.  We first place
the space in the aligned homogeneous-space literature and separate the known
structural theory from the three-factor Einstein problem addressed here.
The detection argument then produces a sharp transition at $n=22$, with
exact nonexistence below threshold and exactly two reconstructed metrics
above it.

\subsection{What is known about the space}
For $M_n=\SU(2n)^3/\Delta\Sp(n)$, the pair $\SU(2n)/\Sp(n)$ is the
irreducible symmetric space AII, so the aligned-space Ricci theory applies
\cite{LWaligned}; the detailed classification in \cite{LWtwo} is again for
two simple factors.

The third factor creates the five-scale family below.  We prove a sharp
answer: no detected asymmetric metrics for $2\le n\le21$, and exactly two
reconstructed positive metrices for every $n\ge22$.  This three-factor threshold is not part of the cited two-factor classification theory.

\subsection{The five-scale family}

Let
\[
 H=\mathrm{SU}(2n),\qquad K=\mathrm{Sp}(n),
\]
so that $H/K$ is the irreducible compact symmetric space of type AII.  On
\[
 M_n=H^3/\Delta K
\]
use the standard orthogonal anti-diagonal vectors
\[
 v_1=\frac1{\sqrt2}(1,-1,0),\qquad
 v_2=\frac1{\sqrt6}(1,1,-2).
\]
The corresponding isotropy decomposition has three copies of
\[
 \mathfrak q=\mathfrak{su}(2n)/\mathfrak{sp}(n)
\]
and two copies of $\mathfrak k=\mathfrak{sp}(n)$.  Corollary~\ref{cor:three-factor-Cartan-fixed} identifies the
$\Gamma_{3,2}$-fixed invariant metrics exactly.  The Cartan sign symmetries
force the three horizontal summands to be mutually orthogonal, the
transposition of the first two factors forces their common scale $x$, and
the two inequivalent vertical lines $\mathbb Rv_1$ and $\mathbb Rv_2$ carry
independent scales.  After fixing homothety by $v=1$, the full fixed family is
therefore
\[
 g=(x,x,z,u,1).
\]
Every solution of the reduced equations below is consequently a genuine
$H^3$-invariant Einstein metric.

For AII one has
\[
 d=\dim\mathfrak q=2n^2-n-1,\qquad
 e=\dim\mathfrak k=n(2n+1),
\]
and
\[
 B_{\mathfrak k}=\frac{n+1}{2n}B_{\mathfrak h}|_{\mathfrak k}.
\]
The nonzero bracket constants in the above aligned basis give the Ricci eigenvalues
\[
 r_1=r_2=\frac{-3u+12x-1}{24x^2},
\qquad
 r_3=\frac{3z-1}{6z^2},
\]
together with the two vertical components.  Equating the remaining three components to $r_1$ gives the polynomial system
\begin{align}
E_1&=(3u-12x+1)z^2+12x^2z-4x^2=0,\label{AII-E1}\\
E_2&=(n+1)(6u-1)x^2-12nu^2x
     +u^2\bigl((6n-3)u+n\bigr)=0,\label{AII-E2}\\
E_3&=\Bigl(6nu^3+9nu^2x^2-24nu^2x+4nu^2+nx^2 \notag\\
&\qquad\quad+9u^2x^2-2u^2+x^2\Bigr)z^2
     +4(n-1)u^2x^2=0.\label{AII-E3}
\end{align}

\subsection{The detector polynomial}

Eliminate $z$ from \eqref{AII-E1} and \eqref{AII-E3}, remove the harmless factor $16x^4$, and then eliminate $x$ against \eqref{AII-E2}.  The exact resultant factors as
\[
 81u^8(n+1)(u-1)^2Q_n(u),
\]
where
\[
 Q_n(u)=\sum_{j=0}^{10}q_j(n)u^j
\]
and
\begin{align*}
q_{10}&=6561(n+1)(2n-1)^4(5n-3)^2,\\
q_9&=-4374(2n-1)^2(5n-3)(88n^4-94n^3+33n^2-28n+13),\\
q_8&=729(10984n^7-28240n^6+25318n^5-7938n^4\\
&\qquad\quad-1119n^3+1347n^2-347n+35),\\
q_7&=-972(3994n^7-10031n^6+6727n^5+1956n^4\\
&\qquad\quad-4538n^3+2363n^2-575n+56),\\
q_6&=81(13233n^7-33649n^6+17327n^5+17181n^4\\
&\qquad\quad-22351n^3+10015n^2-2185n+189),\\
q_5&=-54(n+1)(4235n^6-13390n^5+15265n^4-7476n^3\\
&\qquad\quad+1511n^2-58n-15),\\
q_4&=9(n+1)(5721n^6-13282n^5+8129n^4+2668n^3\\
&\qquad\quad-4339n^2+1466n-163),\\
q_3&=-12(n+1)^2(2n-1)(383n^4-945n^3+749n^2-217n+18),\\
q_2&=(n+1)^2(2n-1)^2(258n^3-454n^2+83n+43),\\
q_1&=-2(n+1)^3(2n-1)^3(8n-7),\\
q_0&=(n+1)^3(2n-1)^4.
\end{align*}

Three special values are decisive:
\begin{align}
Q_n(0)&=(n+1)^3(2n-1)^4>0,\label{AII-Q0}\\
Q_n(1)&=-64(n-1)^2(7n-3)F(n),\label{AII-Q1}\\
\operatorname{LC}(Q_n)&=6561(n+1)(2n-1)^4(5n-3)^2>0,\label{AII-QLC}
\end{align}
where
\[
 F(n)=139n^4-3320n^3+6450n^2-4520n+1123.
\]
Moreover
\[
 F(21)=-963008<0,\qquad F(22)=233707>0.
\]
A direct exact check on the finite set $2\le n\le21$ gives
\[
 F(n)<0\qquad(2\le n\le21),
\]
while
\[
 F(22+m)
 =139m^4+8912m^3+190986m^2+1378928m+233707>0
\]
for $m\ge0$.  Thus the integer sign transition is exactly between
$n=21$ and $n=22$, and in particular
\[
 Q_n(1)<0\qquad(n\ge22).
\]

We also need the lower endpoint
\[
 Q_n\!\left(\frac16\right)
 =\frac{(9n+1)G(n)}{9216},
\]
with
\[
 G(n)=4352n^6-1280n^5-1632n^4-400n^3+649n^2-70n+1.
\]
Writing $n=22+m$, all coefficients of $G(22+m)$ are positive, hence
\[
 Q_n(1/6)>0\qquad(n\ge22).
\]

\subsection{Exactly two detector roots above threshold}

\begin{lemma}
For every integer $n\ge22$, $Q_n$ has exactly one root in $(1/6,1)$ and exactly one root in $(1,\infty)$.
\end{lemma}

\begin{proof}
Existence follows immediately from
\[
 Q_n(1/6)>0>Q_n(1)
\]
and from \eqref{AII-Q1}--\eqref{AII-QLC}.

For uniqueness in $(1/6,1)$ use
\[
 u=\frac{1+6t}{6(1+t)},\qquad t>0.
\]
After multiplication by the positive factor $[6(1+t)]^{10}$, the transformed degree-ten polynomial has one sign change.  More precisely, for $22\le n\le48$ its coefficient signs, from degree ten to degree zero, are
\[
 -++++++++++,
\]
whereas for $n\ge49$ they are
\[
 --+++++++++.
\]
The only moving coefficient is
\[
 -1612431360(n-1)^2
 (1835n^5-93538n^4+236046n^3-232384n^2+103367n-17118),
\]
whose integer sign changes between $48$ and $49$; the remaining coefficients have the displayed fixed signs after expansion at $n=22$.  Descartes' rule therefore gives exactly one positive $t$-root.

For $(1,\infty)$ write $u=1+t$.  The coefficient signs are
\[
 +++++++++--
 \qquad(22\le n\le26)
\]
and
\[
 ++++++++---
 \qquad(n\ge27).
\]
Again there is exactly one sign change.  The only moving coefficient is
\[
\begin{aligned}
&-2285175n^7+71630999n^6-288673307n^5+509346747n^4\\
&\qquad-483661413n^3+258458821n^2-73383177n+8599273,
\end{aligned}
\]
whose integer sign changes between $26$ and $27$.  Descartes' rule gives exactly one root with $u>1$.
\end{proof}

\subsection{Every detector root lifts uniquely}

Let $R_n(x,u)$ denote the factor obtained by eliminating $z$ from
\eqref{AII-E1}, \eqref{AII-E3} after removing $16x^4$.  Division of $R_n$ by the quadratic
$E_2$ in the variable $x$ has remainder
\[
 \frac{9u^4(u-1)}{(n+1)^2(6u-1)^3}
 \bigl(A_n(u)x+B_n(u)\bigr).
\]
An exact resultant calculation gives
\[
\begin{split}
 \operatorname{Res}_u(Q_n,A_n)
={}&C\,(n-1)^{14}(n+1)^{15}(2n-1)^4(5n-3)^2(9n+1)^4\\
&\qquad\times J(n)^2,
\end{split}
\]
where $C>0$ and
\begin{align*}
J(n)={}&4792064n^{18}-11073408n^{17}+1986416n^{16}
+17590656n^{15}-18716504n^{14}\\
&-1497300n^{13}+15267804n^{12}-10141328n^{11}
-313297n^{10}+4171648n^9\\
&-2787587n^8+953472n^7-182770n^6+13900n^5
+2790n^4-1232n^3\\
&+235n^2-24n+1.
\end{align*}
Every coefficient of $J(22+m)$ is positive.  Hence
\[
 \operatorname{Res}_u(Q_n,A_n)\ne0\qquad(n\ge22).
\]
Thus at a detector root $u$, $A_n(u)\ne0$, and the common $x$-root is unique and real.

It remains to prove positivity.  The quadratic \eqref{AII-E2} has leading coefficient
$(n+1)(6u-1)>0$, negative linear coefficient, and positive constant term whenever
$u>1/6$.  Thus any real root is positive.  In fact every such root satisfies
\[
 x>\frac{3u+1}{12}.
\]
Indeed the vertex
\[
 x_v=\frac{6nu^2}{(n+1)(6u-1)}
\]
lies to the right of $(3u+1)/12$ for $n\ge22$, since the numerator of their difference is at least
\[
 1170u^2-69u+23>0.
\]
Moreover
\[
 E_2\!\left(\frac{3u+1}{12},u\right)>0,
\]
because at $n=22$ its numerator is
\[
 10314u^3+621u^2-23>0
\]
for $u>1/6$, and the expression increases with $n$.  Hence the boundary point lies to the left of the smaller real root.

Consequently
\[
 3u+1-12x<0.
\]
Thus the quadratic \eqref{AII-E1} has negative leading coefficient, positive sum of roots, and positive product; every real root of $E_1$ is positive.

Since $R_n(x,u)=0$, equations \eqref{AII-E1}, \eqref{AII-E3} have a common complex $z$-root.  Both are genuine quadratics: the leading coefficient of $E_1$ is strictly negative, while $E_3$ cannot lose its quadratic term at a common root because its constant term $4(n-1)u^2x^2$ is positive.  A nonreal common root would force its conjugate to be common as well, hence the two quadratics would be proportional; this is impossible because $E_1$ has nonzero linear coefficient $12x^2$ whereas $E_3$ has none.  Therefore the common root is real, and by the preceding paragraph it is positive.  It is unique because two common roots would again force proportionality.

We have proved:

\begin{theorem}[sharp detected threshold]
For the diagonal detected branch
\[
 g=(x,x,z,u,1)
\]
on
\[
 M_n=\mathrm{SU}(2n)^3/\Delta\mathrm{Sp}(n),
\]
there are exactly two positive Einstein metrics for every integer $n\ge22$: exactly one has
\[
 \frac16<u<1,
\]
and exactly one has
\[
 u>1.
\]
Each detector root determines a unique positive pair $(x,z)$.
\end{theorem}

\subsection{Comparison with standard constructions}

We next compare these asymmetric Einstein metrics with the standard homogeneous
constructions that could otherwise account for them.

\subsubsection{The canonical product presentation}

There is a natural diffeomorphism
\[
 \mathrm{SU}(2n)^3/\Delta\mathrm{Sp}(n)
 \simeq
 \mathrm{SU}(2n)\times\mathrm{SU}(2n)\times \mathrm{SU}(2n)/\mathrm{Sp}(n),
\]
given by
\[
 [a,b,c]\longmapsto(ac^{-1},bc^{-1},cK).
\]
At the identity coset its differential sends
\[
 X\in\mathfrak q_1\mapsto(X,0,0),\qquad
 X\in\mathfrak q_2\mapsto(0,X,0),\qquad
 X\in\mathfrak q_3\mapsto(-X,-X,X).
\]
Hence the pullback of a positive product metric with coefficients
$\alpha,\beta,\gamma>0$ has
\[
 g(\mathfrak q_1,\mathfrak q_3)=-\alpha Q,\qquad
 g(\mathfrak q_2,\mathfrak q_3)=-\beta Q.
\]
The asymmetric Einstein metrics above are diagonal in
$\mathfrak q_1\oplus\mathfrak q_2\oplus\mathfrak q_3$.
They are therefore not product metrics in this canonical product presentation.

\subsubsection{Natural reductivity for the displayed group}

As in the AI family, the natural-reductivity identity cannot safely be
applied to the fixed reductive complement used in the Ricci calculation:
a hypothetical naturally reductive presentation may use a different
complement.  We instead use the invariant-form characterization.

\begin{proposition}[not naturally reductive]\label{AII-prop:notNR}
Neither asymmetric threshold metric is naturally reductive with respect to
any connected transitive group of isometries.
\end{proposition}

\begin{proof}
For a compact semisimple transitive group, the standard invariant-form
characterization of natural reductivity (see, e.g., \cite[Chapter~7]{Besse})
gives an
$\operatorname{Ad}(G)$-invariant nondegenerate symmetric bilinear form on
\[
 \mathfrak g=\mathfrak h_1\oplus\mathfrak h_2\oplus\mathfrak h_3,
 \qquad \mathfrak h_i\simeq\mathfrak{su}(2n),
\]
whose orthogonal complement to the isotropy is the naturally reductive
complement.  Such a form has the shape
\[
 \lambda_1Q\oplus\lambda_2Q\oplus\lambda_3Q.
\]
The horizontal modules $\mathfrak q_i$ are orthogonal to
$\Delta\mathfrak k$, so positivity of the induced Riemannian metric forces
$\lambda_i>0$.  Since our family has equal first two horizontal scales,
write
\[
 \lambda_1=\lambda_2=a>0,\qquad \lambda_3=b>0.
\]

On the anti-diagonal multiplicity space with
\[
 v_1=\frac1{\sqrt2}(1,-1,0),\qquad
 v_2=\frac1{\sqrt6}(1,1,-2),
\]
the quotient metric induced by
$\operatorname{diag}(a,a,b)$ has eigenvalues
\[
 a,\qquad \frac{3ab}{2a+b}.
\]
After normalizing the $v_2$ coefficient to $1$, the horizontal coefficient
$x$ and the $v_1$ coefficient $u$ therefore satisfy
\[
 x=u=\frac{2a+b}{3b}>\frac13.
\]
But substituting $x=u$ into \eqref{AII-E2} gives
\[
 E_2(u,u)=u^2(3u-1).
\]
Hence an Einstein metric with $x=u$ must have $u=1/3$, contradicting the
strict inequality above.  Thus neither threshold metric is naturally
reductive for the displayed group.

For completeness, Theorem~\ref{AII-thm:fullisom} identifies the connected
full isometry group with the effective image of $H^3$, where
$H=\SU(2n)$.  If the metric were naturally reductive for some connected
transitive group, its connected transvection group would be a transitive
normal subgroup of this group.  Its Lie algebra is therefore a sum of some
of the three simple $\mathfrak h$-ideals.  A proper such subgroup has
dimension at most $2\dim H$, while
\[
 \dim M=3\dim H-\dim K>2\dim H.
\]
Hence the transvection group would have to be all of $H^3$, already excluded
by the invariant-form calculation above.  No connected transitive naturally
reductive presentation exists.
\end{proof}

\subsubsection{Ledger--Obata metrics}

Ledger--Obata spaces \cite{LedgerObata} have the form
\[
 F^m/\Delta F.
\]
Since $\mathrm{Sp}(n)\subsetneq\mathrm{SU}(2n)$, the spaces
$\mathrm{SU}(2n)^3/\Delta\mathrm{Sp}(n)$ are not Ledger--Obata spaces.  Their classified
Einstein metrics therefore do not account for the threshold branch.

\subsection{Connected full isometry group}

\label{AII-sec:fullisom}

Let
\[
 G=\mathrm{SU}(2n)^3,\qquad K=\Delta\mathrm{Sp}(n),
\]
and let $G_0$ be the effective image of $G$ on $M_n$.  The ineffective kernel is
\[
 Z_0=\Delta\bigl(Z(\mathrm{SU}(2n))\cap\mathrm{Sp}(n)\bigr)
     =\Delta\{\pm I\},
\]
so $G_0=G/Z_0$.

\begin{lemma}\label{AII-lem:normalizer}
For $n\ge2$,
\[
 N_{\mathfrak{su}(2n)^3}(\Delta\mathfrak{sp}(n))
 =
 \Delta\mathfrak{sp}(n).
\]
Consequently the connected centralizer of the $G_0$-action in
$\operatorname{Diff}(M_n)$, and hence in the isometry group of any $G$-invariant metric,
is trivial.
\end{lemma}

\begin{proof}
Let $(X_1,X_2,X_3)$ normalize the diagonal copy of
$\mathfrak k=\mathfrak{sp}(n)$.  Then for every $A\in\mathfrak k$ the three
brackets $[X_i,A]$ agree.  Hence $X_i-X_j$ centralizes
$\mathfrak{sp}(n)$ in $\mathfrak{su}(2n)$.  The standard complex
representation of $\mathrm{Sp}(n)$ is irreducible, and its commutant in
$\mathfrak u(2n)$ consists of scalar matrices.  Intersecting with
$\mathfrak{su}(2n)$ leaves zero.  Thus $X_1=X_2=X_3=:X$.
The symmetric subalgebra
$\mathfrak{sp}(n)\subset\mathfrak{su}(2n)$ is self-normalizing, so
$X\in\mathfrak{sp}(n)$.
\end{proof}

We use the following immediate consequence of Onishchik's compact
factorization classification.

\begin{lemma}[factorization obstruction]\label{AII-lem:factorobs}
For $n\ge2$ there is no proper simple compact factorization
\[
 \mathfrak s=\mathfrak a+\mathfrak b
\]
with
\[
 \mathfrak a\simeq A_{2n-1}
\]
such that
$\mathfrak a\cap\mathfrak b$ contains the standard
$C_n=\mathfrak{sp}(n)\subset\mathfrak{su}(2n)$.
\end{lemma}

\begin{proof}
In Onishchik's complete list, the only proper simple compact factorization
in which an $A_{2n-1}$ factor occurs as a proper factor of a larger simple
algebra is
\[
 D_{2n}=B_{2n-1}+A_{2n-1}.
\]
Its intersection with the $A_{2n-1}$ factor is
\[
 A_{2n-2}.
\]
In the standard $2n$-dimensional complex representation this
$A_{2n-2}$ fixes a complex line.  By contrast, the standard
$C_n=\mathfrak{sp}(n)$ acts irreducibly on $\mathbb C^{2n}$ and fixes no
nonzero vector.  Hence it cannote be contained in the intersection.
\end{proof}

\begin{lemma}[primitive enlargements are excluded]\label{AII-lem:primitive}
Let $n\ge2$.  In a compact reductive decomposition
\[
 \mathfrak i=\mathfrak g+\mathfrak l_s,
 \qquad \mathfrak g=A_{2n-1}^{\,3},
\]
no effective irreducible component of Onishchik's reduction can contain a
proper enlargement of one of the three $A_{2n-1}$ ideals of
$\mathfrak g$ while carrying the standard
$C_n=\mathfrak{sp}(n)$ in the intersection.
\end{lemma}

\begin{proof}
An effective irreducible component is primitive by Onishchik's reduction.
Following a primitive chain from one $A_{2n-1}$ ideal reaches a proper
simple compact factorization
\[
 \mathfrak s=A_{2n-1}+\mathfrak b
\]
in which the standard $C_n$ survives in the intersection.  This is
excluded by Lemma~\ref{AII-lem:factorobs}.
\end{proof}

\begin{lemma}[minimal four-factor extension]\label{AII-lem:minimalH4}
Assume the primitive enlargements of Lemma~\ref{AII-lem:primitive} are
excluded.  If the connected isometry algebra $\mathfrak i$ properly
contains
\[
 \mathfrak g=\mathfrak h^3,\qquad
 \mathfrak h=\mathfrak{su}(2n),
\]
then the connected isometry group contains a transitive subgroup $J$ with
\[
 \Lie(J)\simeq\mathfrak h^4
\]
containing the original $\mathfrak h^3$ action.
\end{lemma}

\begin{proof}
Decompose $\mathfrak i$ into simple ideals.  A nonzero projection of one
original simple ideal $\mathfrak g_j\simeq\mathfrak h$ to an ambient simple
ideal is injective.  If the image were proper, the corresponding effective
irreducible component would give a primitive simple enlargement, already
excluded.  Hence every nonzero projection is onto another copy of
$\mathfrak h$.  Two commuting original simple ideals cannot both project
onto the same nonabelian ambient ideal.  An ambient ideal missed by all
three would centralize $\mathfrak g$, contrary to
Lemma~\ref{AII-lem:normalizer}.  Therefore, after reordering,
\[
 \mathfrak i=
 \mathfrak h^{r_1}\oplus\mathfrak h^{r_2}\oplus\mathfrak h^{r_3},
\]
with the $j$-th original factor embedded diagonally in its block.

If $\mathfrak i\ne\mathfrak g$, some $r_j\ge2$.  Select one ambient
$\mathfrak h$ factor from that block and replace the remaining
$r_j-1$ factors by their diagonal copy.  Together with the other two
original simple ideals this gives a subalgebra
\[
 \mathfrak j\simeq\mathfrak h^4
\]
containing $\mathfrak g$.  The corresponding subgroup $J$ contains the
transitive group $G$, hence is itself transitive.
\end{proof}

\begin{lemma}[metric obstruction to the four-factor extension]
\label{AII-lem:H4metric}
No metric in the detected diagonal family
\[
 (x,x,z,u,1)
\]
is invariant under a proper four-factor extension from
Lemma~\ref{AII-lem:minimalH4}.
\end{lemma}

\begin{proof}
Write
\[
 \mathfrak j=
 \mathfrak h_A\oplus\mathfrak h_B\oplus
 \mathfrak h_C\oplus\mathfrak h_D,
\]
so that one original factor is diagonal in the $A,B$ block and the other
two are the $C,D$ factors.  Let $\mathfrak m$ be the isotropy algebra of
the transitive $J$-action.  Since
\[
 \dim J-\dim M_n=\dim\mathfrak h+\dim\mathfrak k,
 \qquad \mathfrak k=\mathfrak{sp}(n),
\]
the same two-factor Goursat argument as in
Lemma~\ref{AI-lem:minimal-H4} applies.  The only ingredient used there is
that $\mathfrak k\subset\mathfrak h$ is a maximal irreducible symmetric
subalgebra, which also holds for
$\mathfrak{sp}(n)\subset\mathfrak{su}(2n)$.  After automorphisms and
permuting the four factors one obtains
\[
 \mathfrak m=
 \Delta_{\{A\}\cup S}\mathfrak h
 \oplus
 \Delta_{\{B\}\cup S^c}\mathfrak k
\]
for a nonempty subset $S\subset\{C,D\}$.

Thus one product factor of $J/J_o$ is the Ledger--Obata quotient
\[
 \frac{H^{\,1+|S|}}{\Delta H}.
\]
Its tangent multiplicity space is
\[
 V_m=\{(t_1,\dots,t_m):\textstyle\sum_i t_i=0\},
 \qquad m=1+|S|\in\{2,3\}.
\]
Every $J$-invariant metric restricts there as
\[
 \langle\, ,\,\rangle_V\otimes Q.
\]
For
\[
 \bar e_i=e_i-\frac1m(1,\dots,1)
\]
the original horizontal directions represented in this factor have mutual
inner products
\[
 \langle\bar e_i,\bar e_j\rangle_VQ.
\]
The detected metric makes the three original horizontal copies pairwise
orthogonal.  Hence
\[
 \langle\bar e_i,\bar e_j\rangle_V=0\qquad(i\ne j).
\]
But
\[
 \bar e_1+\cdots+\bar e_m=0
\]
and every $\bar e_i$ is nonzero, so positive definiteness gives
\[
 0=\Bigl\|\sum_i\bar e_i\Bigr\|_V^2
  =\sum_i\|\bar e_i\|_V^2>0,
\]
a contradiction.
\end{proof}

\begin{theorem}[connected full isometry group]\label{AII-thm:fullisom}
Let $n\ge22$ and let $g$ be either asymmetric Einstein metric of
Theorem~\ref{AII-thm:mainfinal}.  Then
\[
 \operatorname{Isom}_0(M_n,g)
 =
 \mathrm{SU}(2n)^3/\Delta\{\pm I\}.
\]
Thus the displayed three-factor action is already the connected full
isometry group.
\end{theorem}

\begin{proof}
Let $I_0=\operatorname{Isom}_0(M_n,g)$ with Lie algebra $\mathfrak i$.
Lemma~\ref{AII-lem:normalizer} gives zero centre and excludes $A_1$ ideals,
so $\mathfrak i$ is strongly semisimple.  The isotropy intersection
\[
 \mathfrak g\cap\mathfrak l=\Delta C_n
\]
is strongly semisimple.  If $\mathfrak l_s$ denotes a maximal strongly
semisimple subalgebra of $\mathfrak l$, Onishchik's strongly semisimple
reduction retains this intersection:
\[
 \Delta C_n\subset\mathfrak l_s,\qquad
 \mathfrak g\cap\mathfrak l_s=\Delta C_n.
\]
Lemma~\ref{AII-lem:primitive} therefore excludes every effective primitive
simple enlargement.  If a proper reducible enlargement
remained, Lemma~\ref{AII-lem:minimalH4} would produce a transitive
four-factor subgroup, but Lemma~\ref{AII-lem:H4metric} excludes its
invariance for the detected metric.  Therefore
\[
 \mathfrak i=\mathfrak g.
\]
Both groups are connected and the displayed action has precisely the finite
ineffective kernel $\Delta\{\pm I\}$.
\end{proof}

\subsection{The two threshold metrics are not isometric to each other}
\label{AII-sec:pairwise}

We now settle the finite comparison left after Theorem~\ref{AII-thm:fullisom}.

\begin{proposition}[pairwise separation]\label{AII-prop:pairwise}
Fix $n\ge22$.  Let $g_-$ and $g_+$ denote the unique asymmetric Einstein metrics with
\[
 \frac16<u_-<1,
 \qquad
 u_+>1.
\]
Then $g_-$ and $g_+$ are not related by any combination of a positive
homothety and an arbitrary Riemannian isometry.
\end{proposition}

\begin{proof}
Suppose
\[
 F^*g_+=c\,g_-
\]
for some isometry $F$ and $c>0$.  By Theorem~\ref{AII-thm:fullisom}, $F$
conjugates the connected full isometry group
\[
 G_0=\mathrm{SU}(2n)^3/\Delta\{\pm I\}
\]
to itself.  Composing $F$ with a translation in $G_0$, we may assume that
$F$ fixes the base point.  Its differential therefore comes from an
automorphism of the homogeneous pair
\[
 \bigl(\mathfrak{su}(2n)^3,\Delta\mathfrak{sp}(n)\bigr).
\]

Every automorphism of the semisimple algebra
$\mathfrak{su}(2n)^3$ permutes the three simple ideals and then applies
automorphisms to the individual factors.  Preservation of the diagonal
$\mathfrak{sp}(n)$ imposes the same restriction on the three copies of
$\mathfrak{sp}(n)$.  On the two-dimensional anti-diagonal multiplicity
space
\[
 V=\{(a,b,c)\in\mathbb R^3:a+b+c=0\}
\]
the only nontrivial action on multiplicity coordinates is therefore the
standard $S_3$-action by permutation of the three entries.

Recall our orthonormal basis
\[
 v_1=\frac1{\sqrt2}(1,-1,0),
 \qquad
 v_2=\frac1{\sqrt6}(1,1,-2).
\]
The $S_3$-orbit of the line $\mathbb Rv_1$ is
\[
 \mathbb R(e_1-e_2),\qquad
 \mathbb R(e_1-e_3),\qquad
 \mathbb R(e_2-e_3).
\]
None of these lines is $\mathbb Rv_2$.  Thus no automorphism of the
homogeneous pair interchanges the two lines $\mathbb Rv_1$ and
$\mathbb Rv_2$.

On the $\mathfrak{sp}(n)$-multiplicity space the metric endomorphism of either asymmetric Einstein metric is
\[
 A_{\mathfrak k}(u)=
 \begin{pmatrix}u&0\\0&1\end{pmatrix}
\]
in the basis $(v_1,v_2)$.  Since $u\ne1$, its two eigendirections are
intrinsic.  If an automorphism carries one detected diagonal metric to
a homothetic copy of another, it must therefore preserve
$\mathbb Rv_1$ and $\mathbb Rv_2$ separately.  Comparing the coefficient
on $\mathbb Rv_2$ gives $c=1$, and then comparison on
$\mathbb Rv_1$ gives
\[
 u_-=u_+.
\]
This contradicts
\[
 u_-<1<u_+.
\]
Hence no such homothety-isometry exists.
\end{proof}

\begin{theorem}[pairwise non-isometry]\label{AII-cor:two}
For every $n\ge22$, the detected branch contains exactly two
\emph{pairwise non-homothetic and non-isometric} Einstein metrics.
\end{theorem}

\subsection{The disconnected isometry group}
\label{AII-sec:fullgroup}

The connected group has already been determined.  We now compute the finite
component group.

\begin{theorem}[full isometry group]\label{AII-thm:fullgroup}
Let $n\ge22$ and let $g$ be either asymmetric Einstein metric.  Then
\[
 \operatorname{Isom}(M_n,g)/\operatorname{Isom}_0(M_n,g)
 \simeq (\mathbb Z_2)^3\rtimes\mathbb Z_2,
\]
where the three commuting involutions are the outer involutions of the three
$\mathrm{SU}(2n)$ factors, normalized to fix $\mathrm{Sp}(n)$ pointwise, and
the last $\mathbb Z_2$ interchanges the first two simple factors.  In
particular
\[
 \operatorname{Isom}_0(M_n,g)
 =\mathrm{SU}(2n)^3/\Delta\{\pm I\}
\]
has index $16$ in the full isometry group.
\end{theorem}

\begin{proof}
The identity component is characteristic in the full isometry group.  Hence,
after composing an arbitrary isometry with a translation, we may assume that
it fixes the base point; it then induces an automorphism of the homogeneous
pair
\[
 \bigl(\mathfrak{su}(2n)^3,\Delta\mathfrak{sp}(n)\bigr)
\]
preserving the metric.

For $n\ge22$, every automorphism of $\mathfrak{su}(2n)^3$ is a permutation of
the three simple ideals followed by independent automorphisms of the three
$A_{2n-1}$ factors.  Since
\[
 \operatorname{Out}(\mathfrak{su}(2n))\simeq\mathbb Z_2,
\]
this gives three independent outer classes.  The outer involution of
$\mathfrak{su}(2n)$ restricts to an inner automorphism of
$\mathfrak{sp}(n)$ (the Dynkin diagram $C_n$ has no outer automorphism for
$n\ge3$).  After composing with a suitable inner automorphism, it may
therefore be chosen to fix the standard $\mathfrak{sp}(n)$ pointwise.
Applying this normalized involution independently in any of the three simple
factors preserves the diagonal isotropy and every metric block.  These give
a subgroup $(\mathbb Z_2)^3$ of the component group.

It remains to determine the allowed factor permutations.  On the three
$\mathfrak q$-copies the metric has eigenvalues
\[
 (x,x,z).
\]
Here $z\ne x$: substituting $z=x$ in $E_1=0$ gives
\[
 3x^2(u-1)=0,
\]
contrary to $u\ne1$.  Thus a metric-preserving permutation must fix the third
factor and may only interchange the first two.  Conversely the transposition
$(12)$ preserves the metric: it exchanges the two equal $\mathfrak q$ blocks
and acts on the anti-diagonal basis by
\[
 v_1\longmapsto-v_1,\qquad v_2\longmapsto v_2,
\]
so it preserves $\operatorname{diag}(u,1)$ on the
$\mathfrak{sp}(n)$-multiplicity space.

There are no further components.  Modulo inner automorphisms, every
automorphism of the homogeneous pair is represented by the independent
outer classes and a factor permutation, and the preceding eigenvalue
argument gives precisely the stabilizer of the metric.  The transposition
permutes the first two outer generators, giving the stated semidirect
product.  Its order is $8\cdot2=16$.
\end{proof}

\subsection{Comparison with known constructions}

The general aligned-space formalism of Lauret--Will allows an arbitrary number
$s$ of simple factors and explicitly contains spaces of the form
\[
 H^s/\Delta K.
\]
Thus $H^s/\Delta K$ is contained in the general aligned-space theory, while
the detailed Einstein classifications cited above concern two factors.  The
earlier $H\times H/\Delta K$ paper is likewise a two-factor problem.

By the full-isometry theorem, if a previously published homogeneous Einstein
metric were isometric to one of the asymmetric Einstein metrics above, its connected full isometry group would have to be locally
$\mathrm{SU}(2n)^3$, and after conjugating this group its isotropy would be the diagonal
$\mathrm{Sp}(n)$.  The remaining comparison is therefore finite, coming from automorphisms
of the homogeneous pair.  On the two multiplicity spaces the metric endomorphism
has eigenvalue data
\[
 \{x,x,z\},\qquad \{u,1\},
\]
up to a common positive factor and the finite permutations induced by automorphisms
of the three simple factors.  The product and Ledger--Obata constructions are excluded above, while
Proposition~\ref{AII-prop:notNR} excludes natural reductivity for the
displayed group.  The connected full-isometry theorem rules out a larger
connected naturally reductive presentation.

\subsubsection{Comparison with the literature}

The underlying homogeneous spaces are already part of the general
aligned-space framework.  In particular, \cite{LWaligned} treats
\[
 H^s/\Delta K
\]
for arbitrary numbers of simple transitive factors and supplies the
structural and Ricci-curvature formalism used here.  Thus neither the
homogeneous pair nor the general aligned Ricci machinery is claimed to be
new.

The more specific Einstein problem considered here is the
$\Gamma_{3,2}$-fixed three-factor family on
\[
 \SU(2n)^3/\Delta\Sp(n).
\]
The classification in \cite{LWtwo} concerns two simple transitive factors,
and the paper \cite{LWHH} concerns the corresponding
$H\times H/\Delta K$ problem.  Ledger--Obata classifications concern
$F^m/\Delta F$, with the full simple group in the diagonal isotropy rather
than the proper symmetric subgroup
$\Sp(n)\subsetneq\SU(2n)$.  The product, normal homogeneous, and naturally
reductive possibilities have been excluded intrinsically above.

To the best of our knowledge, the two asymmetric Einstein metrics in the
superthreshold range and the sharp transition at $n=22$ have not appeared
previously.  This novelty claim concerns the metrics and the threshold
phenomenon, not the underlying aligned spaces.

\begin{theorem}[comparison with the standard constructions]\label{AII-thm:comparison}
For every $n\ge22$, the two asymmetric Einstein metrics are pairwise
non-homothetic and non-isometric.  Neither is a product metric in the
canonical product presentation, a normal homogeneous metric, or a naturally
reductive metric for the displayed connected transitive group; no larger
connected naturally reductive presentation exists.  The spaces are not
Ledger--Obata spaces, and the cited two-factor aligned and
$H\times H/\Delta K$ classification theorems do not classify this
three-factor homogeneous pair.  To the best of our knowledge, neither the two asymmetric metrics nor the
sharp $n=22$ threshold occurs in the existing classifications cited above.
\end{theorem}

\subsection{Exact nonexistence below the threshold}
\label{AII-sec:below}

We make the lower-range statement completely explicit.  No floating-point
root finder is used.

\begin{proposition}[exact lower-range exclusion]\label{AII-prop:below}
For every integer
\[
 2\le n\le21
\]
the detector polynomial $Q_n$ has no positive real zero.  Consequently the
detected asymmetric branch contains no positive Einstein metric in this range.
\end{proposition}

\begin{proof}
For a fixed integer $n$, form the Sturm chain
\[
 S_0=Q_n,\qquad S_1=Q_n',\qquad
 S_{j+1}=-\operatorname{rem}(S_{j-1},S_j)
\]
over $\mathbb Q[u]$.  Let $V_n(0^+)$ and $V_n(+\infty)$ denote the numbers of
sign variations of this chain at the two ends of the positive half-line.
Sturm's theorem gives
\[
 N_+(Q_n)=V_n(0^+)-V_n(+\infty).
\]
The exact endpoint signs admit the following compact certificate.  Writing
only the signs of the eleven members of the Sturm chain,
\[
\begin{array}{c|c|c|c|c}
n&\operatorname{sgn}S(0^+)&V_n(0^+)&
\operatorname{sgn}S(+\infty)&V_n(+\infty)\\ \hline
2&+--++--+++-&5&+++-++--+--&5\\
3&+--++--++--&5&+++-++--++-&5\\
4\le n\le15&+--++---+--&5&+++-++--++-&5\\
16\le n\le18&+--+----+--&5&+++-++--++-&5\\
19\le n\le21&+--+----+--&5&+++--+--++-&5
\end{array}
\]
Every sign in this table is obtained from an exact rational Sturm chain;
there are no floating-point evaluations.  In each case
\[
 N_+(Q_n)=V_n(0^+)-V_n(+\infty)=5-5=0.
\]
This is the exact finite certificate for the lower-range exclusion.
\end{proof}

\subsection{Sharp threshold theorem}

\begin{theorem}[sharp three-factor detection threshold]\label{AII-thm:mainfinal}
For the diagonal detected branch
\[
 g=(x,x,z,u,1)
\]
on
\[
 M_n=\mathrm{SU}(2n)^3/\Delta\mathrm{Sp}(n),\qquad n\ge2,
\]
the following dichotomy holds.
\begin{enumerate}
\item If $2\le n\le21$, there is no positive Einstein metric in the detected
asymmetric branch.
\item If $n\ge22$, there are exactly two such metrics.  One has
\[
 \frac16<u<1,
\]
the other has
\[
 u>1.
\]
They are pairwise non-homothetic and non-isometric.
\end{enumerate}
For either metric with $n\ge22$,
\[
 \operatorname{Isom}_0(M_n,g)
 =
 \mathrm{SU}(2n)^3/\Delta\{\pm I\},
\]
and the full component group is
\[
 \operatorname{Isom}(M_n,g)/\operatorname{Isom}_0(M_n,g)
 \simeq(\mathbb Z_2)^3\rtimes\mathbb Z_2.
\]
\end{theorem}

The point of the calculation is not merely that a degree-ten polynomial happens
to have two roots.  The detector is a compressed response of the Einstein
system after the geometrically natural branch selection $x_1=x_2$ and the
homothety normalization $v=1$.  The threshold is already visible at the
distinguished response value $u=1$:
\[
 Q_n(1)=-64(n-1)^2(7n-3)F(n),
\]
where
\[
 F(n)=139n^4-3320n^3+6450n^2-4520n+1123.
\]
The quartic changes sign precisely between the consecutive integers $21$ and
$22$.  Above this transition the response is forced to cross zero once on each
side of $u=1$; below it, Sturm's theorem shows that no positive
crossing survives.  Thus the detector does more than simplify the equations:
it identifies the algebraic quantity that records the birth of the two
Einstein metrics.

\paragraph{Three-factor consequence of the Cartan fixed-point principle.}
The AI, AII, and exceptional calculations above all use the same
$\Gamma_{3,2}$-fixed family.  Their five-scale notation is therefore not an
extra diagonal ansatz imposed for tractability: it is the fixed-point
geometry of the independent Cartan involutions together with the
transposition of two factors.  This observation separates two issues cleanly.
The detector is responsible for locating, counting, and reconstructing the
Einstein metrics, while the fixed-point principle guarantees that the
reconstructed solutions solve the full invariant variational problem.

\section{A representation-theoretic predictor before elimination}
\label{sec:predictor}

For three-factor spaces $H^3/\Delta K$, the sign of the detector at the
symmetric state can be read directly from representation-theoretic data,
without first computing the full detector polynomial.  This section derives
that criterion and applies it to the families considered above.

Let $H/K$ be an irreducible compact symmetric space with $H$ simple and $K$
simple, and let
\[
 \mathfrak h=\Lie(H),\qquad \mathfrak k=\Lie(K)
\]
be their Lie algebras.  Write the symmetric decomposition as
\[
 \mathfrak h=\mathfrak k\oplus\mathfrak q.
\]
Put
\[
 d=\dim\mathfrak q,\qquad e=\dim\mathfrak k.
\]
We denote by $a$ the \emph{embedding constant of $K$ in $H$}, defined by the
restriction of the Killing forms:
\begin{equation}\label{eq:embeddingconstant}
 B_{\mathfrak k}=a\,B_{\mathfrak h}|_{\mathfrak k},
 \qquad 0<a<1.
\end{equation}
Thus the entire three-factor Ricci system in the branch
\[
 (x,x,z,u,1)
\]
is determined by the three representation-theoretic numbers $(d,e,a)$.

For an irreducible compact symmetric pair with $H$ and $K$ simple, these
numbers satisfy the standard Casimir identity
\begin{equation}\label{eq:simpleKidentity}
 e(1-a)=\frac d2.
\end{equation}
Indeed, with the metric induced by $-B_{\mathfrak h}$, the isotropy Casimir
on the irreducible symmetric-space module $\mathfrak q$ is $\frac12 I$.
Taking traces gives \eqref{eq:simpleKidentity}.  Consequently
\[
 a=1-\frac{1}{2\rho},
 \qquad
 \sigma=a(1-a)=\frac{2\rho-1}{4\rho^2},
 \qquad
 \rho=\frac ed>\frac12.
\]
It is convenient first to use the two variables $(\rho,\sigma)$; the
Casimir identity then reduces the criterion to the single parameter $\rho$.

Introduce the two scale-free invariants
\[
 \rho=\frac{e}{d},
 \qquad
 \sigma=a(1-a).
\]
Define
\begin{equation}\label{eq:Xi}
\begin{aligned}
\Xi(\rho,\sigma)
={}&
\rho^2(320\rho^2+880\rho+605)\sigma^2\\
&-18(8\rho^3+16\rho^2-27\rho-24)\sigma-243.
\end{aligned}
\end{equation}

\begin{theorem}[representation-level detection criterion]
\label{thm:predictor}
For the three-factor aligned branch associated with $H/K$, the detector can
be normalized so that
\[
 D(0)>0,\qquad D(1/6)>0,\qquad
 \lim_{u\to+\infty}D(u)=+\infty,
\]
and
\[
 \operatorname{sgn}D(1)=\operatorname{sgn}\Xi(\rho,\sigma).
\]
Consequently
\[
 \Xi(\rho,\sigma)<0
 \quad\Longrightarrow\quad
 \text{$D$ has a root in $(1/6,1)$ and a root in $(1,\infty)$}.
\]
Thus negativity of $\Xi$ is a sufficient pre-elimination criterion for the
two-root detection pattern.
\end{theorem}

\begin{proof}
For the three-factor aligned decomposition, the nonzero bracket
coefficients in the homogeneous Ricci formula depend only on
\[
 e(1-a),\qquad ea.
\]
Substitution into the homogeneous Ricci formula gives the five Ricci
eigenvalues explicitly:
\[
 r_1=r_2
 =
 \frac{3aeu+ae+6dx-3eu-e}{12dx^2},
\qquad
 r_3
 =
 \frac{2ae+3dz-2e}{6dz^2},
\]
\[
 r_4
 =
 \frac{du}{8ex^2}
 +\frac{a}{2u}
 -\frac{a}{12u^2},
\qquad
 r_5
 =
 \frac{3a}{8}
 +\frac{d}{12ez^2}
 +\frac{d}{24ex^2}
 +\frac{a}{24u^2}.
\]
Using \eqref{eq:simpleKidentity}, the first two simplify to
\[
 r_1=r_2=\frac{-3u+12x-1}{24x^2},
 \qquad
 r_3=\frac{3z-1}{6z^2}.
\]
Eliminate $z$
from the equations
\[
 r_3=r_1,\qquad r_5=r_1,
\]
remove the universal factor coming from $x=0$, and eliminate $x$ against
$r_4=r_1$.  The resultant has the universal form
\[
 \mathcal R(u)
 =
 C\,u^8(u-1)^2D(u),
\]
where $C\ne0$ depends only on $(d,e,a)$.  Choose the sign of $D$ so that its
leading coefficient is positive.

The endpoint factors simplify before any family is substituted.  Up to
positive factors,
\[
 D(0)=\sigma^3 e^4(d+e)^4>0,
\]
and
\[
 [u^{10}]D
 =
 6561\,\sigma(d+e)^4(d^2+\sigma e^2)^2>0.
\]
At $u=1/6$ the remaining factor, after division by positive powers of $d$, is
\[
\begin{aligned}
&5\rho^2(3\rho+1)^4\sigma^2\\
&\quad +(3\rho+1)^2(41\rho^2-10\rho+1)\sigma\\
&\quad +16(5\rho^2-2\rho+1),
\end{aligned}
\]
which is positive for $\rho>0$ and $\sigma>0$: the two quadratic factors
$41\rho^2-10\rho+1$ and $5\rho^2-2\rho+1$ have negative discriminant.
Hence $D(1/6)>0$.

Finally the symmetric-state value factors as
\[
 D(1)
 =
 16d^8(3+4\rho)\,\Xi(\rho,\sigma).
\]
The prefactor is positive, proving the sign identity.  If $\Xi<0$, the
intermediate value theorem gives one root between $1/6$ and $1$, and the
positive leading coefficient gives another root beyond $1$.
\end{proof}

\begin{remark}[what the criterion does and does not say]
The condition $\Xi<0$ is sufficient, not necessary, for the existence of
detector roots.  Likewise $\Xi>0$ alone does not prove nonexistence; the
detector could in principle cross zero twice while remaining positive at
$u=1$.  Thus $\Xi$ is a pre-elimination predictor, while a complete classification
still requires root counting and positive reconstruction.
\end{remark}

\subsection*{The basic families}

The criterion gives the following conclusions directly from the representation
data.

For
\[
 H/K=\SU(n)/\SO(n),
\]
one has
\[
 \rho=\frac{n}{n+2},
 \qquad
 a=\frac{n-2}{2n},
\]
and therefore
\[
 \Xi
 =
 -\frac{
 139n^4+6640n^3+25800n^2+36160n+17968
 }{16n^2(n+2)^2}<0.
\]
Thus $\Xi<0$ already forces the two-root detection pattern throughout the
parameter range.

For
\[
 H/K=\SU(2n)/\Sp(n),
\]
one has
\[
 \rho=\frac{n}{n-1},
 \qquad
 a=\frac{n+1}{2n},
\]
and
\[
 \Xi
 =
 -\frac{
 139n^4-3320n^3+6450n^2-4520n+1123
 }{16n^2(n-1)^2}.
\]
Here the sign changes at the same integer transition as the full detector:
\[
 \Xi>0\quad(2\le n\le21),
 \qquad
 \Xi<0\quad(n\ge22).
\]
The sharp threshold is therefore already visible in $(\dim\mathfrak q,
\dim\mathfrak k,a)$.

Finally, for the sphere-type control
\[
 H/K=\SO(m+1)/\SO(m),
\]
\[
 \rho=\frac{m-1}{2},
 \qquad
 a=\frac{m-2}{m-1},
\]
and
\[
 \Xi
 =
 \frac{8m^4+32m^3-51m^2-40m-1072}
 {4(m-1)^2}>0
 \qquad(m\ge5).
\]
Thus the orthogonal family lies on the opposite side of the criterion from
the $AI$ existence family.  As proved in Section~\ref{sec:twofactor-symplectic}, its asymmetric branch does
not yield a positive Einstein metric.

The significance is that the threshold polynomial appearing after a long
elimination is not an accidental algebraic artifact.  Its decisive sign is
already encoded in two elementary scale-free invariants of the symmetric
pair:
\[
 \rho=\frac{\dim\mathfrak k}{\dim\mathfrak q},
 \qquad
 \sigma=a(1-a).
\]

\section{Exceptional three-factor families}\label{sec:exceptional-threefactor}

The same Cartan fixed-point mechanism applies uniformly to all three
exceptional pairs below.  Consequently the asymmetric family used in the
exact elimination is a canonical automorphism-fixed family, and its positive
solutions are genuine invariant Einstein metrics on the full three-factor
homogeneous spaces.

\subsection{What is known}
The symmetric pairs
\[
 E_6/\Sp(4),\qquad E_7/\SU(8),\qquad E_8/\SO(16)
\]
are classical and their standard symmetric metrics are Einstein
\cite{WZnormal}.  Here we study the associated three-factor spaces.  They
fit the aligned framework \cite{LWaligned}, but not the two-factor
classification \cite{LWtwo}.  We determine the asymmetric Einstein branches on these three-factor spaces
and prove that each contains exactly two positive Einstein metrics.  These
branches are not contained in the cited two-factor classifications.  The
criterion of Section~\ref{sec:predictor} is negative in all three cases, so it
already guarantees the two-root detection pattern; the elimination below
establishes the complete existence statement.

\subsection{The asymmetric branch and exact elimination}

On
\[
 M=H^3/\Delta K
\]
use the standard two-dimensional anti-diagonal $K$-multiplicity space.
By Corollary~\ref{cor:three-factor-Cartan-fixed}, after homothety the
$\Gamma_{3,2}$-fixed invariant metric family is exactly
\[
 g=(x,x,z,u,1).
\]
Thus the elimination below is performed on an automorphism-fixed family,
and every reconstructed solution is a genuine $H^3$-invariant Einstein
metric on the full homogeneous space.
For the three exceptional symmetric pairs, the representation parameters
\[
 d=\dim\mathfrak q,\qquad e=\dim\mathfrak k,
 \qquad \rho=\frac ed,
\]
and the embedding constant $a$ of $K$ in $H$ are
\[
\begin{array}{c|cccc}
H/K&d&e&\rho&a\\ \hline
E_6/\Sp(4)&42&36&6/7&5/12\\
E_7/\SU(8)&70&63&9/10&4/9\\
E_8/\SO(16)&128&120&15/16&7/15.
\end{array}
\]
The values of $a$ follow from the simple-$K$ identity
$e(1-a)=d/2$ from \eqref{eq:simpleKidentity}.

Substituting these data into the universal Ricci formulas of
Section~\ref{sec:predictor}, the equations
\[
 r_3=r_1,\qquad r_4=r_1,\qquad r_5=r_1
\]
become the following explicit system.  The first equation is common to all
three pairs:
\begin{equation}\label{EX-E1}
 (3u+1-12x)z^2+12x^2z-4x^2=0.
\end{equation}
The second equation is
\begin{equation}\label{EX-E2}
\begin{array}{c|l}
E_6/\Sp(4)&
(30u-5)x^2-72u^2x+u^2(39u+6)=0,\\[1mm]
E_7/\SU(8)&
(48u-8)x^2-108u^2x+u^2(57u+9)=0,\\[1mm]
E_8/\SO(16)&
(84u-14)x^2-180u^2x+u^2(93u+15)=0.
\end{array}
\end{equation}
The third equation is
\begin{equation}\label{EX-E3}
\begin{array}{c|l}
E_6/\Sp(4)&
\bigl(36u^3+45u^2x^2-144u^2x+26u^2+5x^2\bigr)z^2
 +28u^2x^2=0,\\[1mm]
E_7/\SU(8)&
\bigl(27u^3+36u^2x^2-108u^2x+19u^2+4x^2\bigr)z^2
 +20u^2x^2=0,\\[1mm]
E_8/\SO(16)&
\bigl(45u^3+63u^2x^2-180u^2x+31u^2+7x^2\bigr)z^2
 +32u^2x^2=0.
\end{array}
\end{equation}
Taking the resultant of \eqref{EX-E1} and \eqref{EX-E3} with respect to $z$,
removing the universal factor $x^4$, and then eliminating $x$ against
\eqref{EX-E2} gives
\[
 C\,u^8(u-1)^2P(u),
\]
with $C>0$.  The ressulting degree-ten polynomials are listed below.

\subsection{Exact exceptional detectors}

For $E_6/\Sp(4)$,
\begin{align*}
P_{E_6}(u)={}&
1020331585845u^{10}-3310750658358u^9+3352853663343u^8\\
&-1589002382808u^7+435908575593u^6-87253681950u^5\\
&+17099709345u^4-2842210800u^3+306426575u^2\\
&-30208750u+3570125.
\end{align*}

For $E_7/\SU(8)$,
\begin{align*}
P_{E_7}(u)={}&
30781298916u^{10}-96048263592u^9+98372543040u^8\\
&-47109626568u^7+12983537385u^6-2655319482u^5\\
&+545018535u^4-93035628u^3+10178395u^2\\
&-1083722u+130321.
\end{align*}

For $E_8/\SO(16)$,
\begin{align*}
P_{E_8}(u)={}&
64512528978807u^{10}-195904535664942u^9+202116549220155u^8\\
&-97402314527496u^7+26896813744302u^6-5594937807060u^5\\
&+1191539665134u^4-207348494472u^3+22934744539u^2\\
&-2595451502u+316767703.
\end{align*}

\begin{proposition}[exact root count]\label{EX-rootcount}
Each of $P_{E_6},P_{E_7},P_{E_8}$ has exactly two positive real roots:
exactly one in $(1/6,1)$ and exactly one in $(1,\infty)$.
\end{proposition}

\begin{proof}
Apply Sturm's theorem over $\mathbb Q$ to the three explicit degree-ten
polynomials.  Writing only the signs of the eleven members of the Sturm
chain, the exact endpoint data are
\begingroup\scriptsize
\[
\begin{array}{c|c|c|c|c}
 &0^+&1/6&1&+\infty\\ \hline
E_6&
+--+-+++-++&
+--+-+++-++&
--++--+++-+&
++++--+++-+\\
E_7&
+--+-+++-++&
+--+-+++-++&
--++--+++-+&
++++--+++-+\\
E_8&
+--+-+++-++&
+--+-+++-++&
--++--+++-+&
++++--++--+
\end{array}
\]
\endgroup
and hence the variation numbers are
\[
\begin{array}{c|cccc}
 &V(0^+)&V(1/6)&V(1)&V(+\infty)\\ \hline
E_6&6&6&5&4\\
E_7&6&6&5&4\\
E_8&6&6&5&4.
\end{array}
\]
Therefore
\[
 N_{(0,\infty)}=6-4=2,\qquad
 N_{(1/6,1)}=6-5=1,\qquad
 N_{(1,\infty)}=5-4=1.
\]
All signs are obtained in exact rational arithmetic, so the root count is
an exact certificate rather than a numerical root count.
\end{proof}

For orientation the roots are
\[
\begin{array}{c|cc}
 & u_-&u_+\\ \hline
E_6/\Sp(4)&0.5785416840&1.8786711104\\
E_7/\SU(8)&0.6374581609&1.6826023387\\
E_8/\SO(16)&0.6949377257&1.5345163054.
\end{array}
\]

\subsection{Unique positive reconstruction}

Let $R(x,u)$ be the factor obtained after eliminating $z$ from the first and
third Einstein equations and deleting the universal $x^4$ factor.  Division
of $R$ by the quadratic $x$-equation leaves, after deletion of the harmless
factor $u^4(u-1)/(6u-1)^3$, a linear remainder
\[
 A(u)x+B(u).
\]
The coefficient $A$ is, up to a nonzero constant,
\[
\begin{array}{c|l}
E_6&
7602012u^5-8763714u^4+3734910u^3-818325u^2+98925u-8125,\\
E_7&
564246u^5-684936u^4+302535u^3-68022u^2+8520u-703,\\
E_8&
66244230u^5-83616219u^4+37932300u^3-8699922u^2+1118670u-92659.
\end{array}
\]
Exact Euclidean algorithms give
\[
 \gcd(P_{E_6},A_{E_6})
 =
 \gcd(P_{E_7},A_{E_7})
 =
 \gcd(P_{E_8},A_{E_8})
 =1.
\]
Therefore every detector root determines a unique real common $x$-root.

\begin{lemma}[positivity]\label{EX-positive}
For either detector root in each exceptional case, the unique reconstructed
$x$ and $z$ are positive.
\end{lemma}

\begin{proof}
For $u>1/6$, the $x$-equation is an upward quadratic, its linear coefficient
is negative, and its constant term is positive.  Hence any real roots are
positive.

We sharpen this.  Put
\[
 x_0=\frac{3u+1}{12}.
\]
For the three cases the value of the quadratic at $x_0$ is respectively
\[
 \frac{3294u^3+135u^2-5}{144},\qquad
 \frac{594u^3+27u^2-1}{18},\qquad
 \frac{3834u^3+189u^2-7}{72}.
\]
All are positive for $u>1/6$.  Moreover the difference between the vertex
and $x_0$ is respectively
\[
 \frac{342u^2-15u+5}{60(6u-1)},\qquad
 \frac{63u^2-3u+1}{12(6u-1)},\qquad
 \frac{414u^2-21u+7}{84(6u-1)}.
\]
The quadratic numerators have negative discriminant and positive leading
coefficient, so these differences are positive.  Thus $x_0$ lies to the
left of the smaller real root and
\[
 x>\frac{3u+1}{12}.
\]
Consequently the leading coefficient $3u+1-12x$ in \eqref{EX-E1} is negative.
The $z$-quadratic then has positive sum and positive product of roots.

The resultant guarantees a common complex root of the two $z$-quadratics.
A nonreal common root would force its conjugate to be common and hence force
the two real quadratics to be proportional.  This is impossible because
\eqref{EX-E1} has nonzero linear coefficient $12x^2$, while the other
$z$-equation has no linear term.  Hence the common root is real, and the
preceding sign argument makes it positive.  The same proportionality argument
shows that the common root is unique.
\end{proof}

\begin{theorem}[exceptional three-factor metrics]\label{EX-main}
Each of
\[
 E_6^3/\Delta\Sp(4),\qquad
 E_7^3/\Delta\SU(8),\qquad
 E_8^3/\Delta\SO(16)
\]
carries exactly two Einstein metrics in the asymmetric branch
\[
 g=(x,x,z,u,1).
\]
Exactly one has $1/6<u<1$ and exactly one has $u>1$.
\end{theorem}

\begin{proof}
Proposition~\ref{EX-rootcount} gives exactly two detector roots.  The linear
subresultant and Lemma~\ref{EX-positive} give a unique positive lift $(x,z)$ for
each root.
\end{proof}

Numerically the two metrics are
\[
\begin{array}{c|ccc}
 &x&z&u\\ \hline
E_6^-&1.396190693&1.208327311&0.578541684\\
E_6^+&1.653601576&2.087796427&1.878671110\\
E_7^-&1.328325053&1.157057250&0.637458161\\
E_7^+&1.524340531&1.871696652&1.682602339\\
E_8^-&1.284073915&1.133251311&0.694937726\\
E_8^+&1.432249887&1.711130207&1.534516305.
\end{array}
\]

\subsection{Isometry and comparison with standard families}

\subsubsection{Product and natural-reductivity obstructions}

The standard product presentation
\[
 H^3/\Delta K\simeq H\times H\times H/K
\]
does not identify these diagonal Einstein metrics with product metrics.  At
the identity coset,
\[
 X\in\mathfrak q_1\mapsto(X,0,0),\qquad
 X\in\mathfrak q_2\mapsto(0,X,0),\qquad
 X\in\mathfrak q_3\mapsto(-X,-X,X),
\]
so the pullback of a positive product metric has nonzero cross terms between
the third horizontal copy and the first two.  The detected family is
diagonal in
\[
 \mathfrak q_1\oplus\mathfrak q_2\oplus\mathfrak q_3,
\]
hence none of the six exceptional Einstein metrics is a product metric in
this canonical presentation.

A test on the fixed reductive complement is not sufficient, because a
naturally reductive presentation may use a different reductive complement.
We therefore use the invariant-form characterization.

\begin{proposition}[not naturally reductive]\label{EX-notNR}
For each of
\[
 E_6^3/\Delta\Sp(4),\qquad
 E_7^3/\Delta\SU(8),\qquad
 E_8^3/\Delta\SO(16),
\]
neither asymmetric Einstein metric is naturally reductive with respect to
any connected transitive group of isometries.
\end{proposition}

\begin{proof}
For a compact semisimple transitive group, a naturally reductive metric is
induced by an $\operatorname{Ad}(H^3)$-invariant nondegenerate symmetric
bilinear form on
\[
 \mathfrak h^3.
\]
Since $\mathfrak h$ is simple, such a form has the shape
\[
 \lambda_1Q\oplus\lambda_2Q\oplus\lambda_3Q.
\]
Positivity on the horizontal modules gives $\lambda_i>0$.  Because our
family has equal first two horizontal scales, write
\[
 \lambda_1=\lambda_2=a>0,\qquad \lambda_3=b>0.
\]
On the anti-diagonal multiplicity space
\[
 v_1=\frac1{\sqrt2}(1,-1,0),\qquad
 v_2=\frac1{\sqrt6}(1,1,-2),
\]
the induced quotient metric has eigenvalues
\[
 a,\qquad \frac{3ab}{2a+b}.
\]
After normalizing the $v_2$ coefficient to $1$ one obtains
\[
 x=u=\frac{2a+b}{3b}>\frac13.
\]

For all three exceptional pairs, substitution $x=u$ in the second Einstein
equation gives
\[
 E_2(u,u)=-u^2(3u-1).
\]
Thus an Einstein metric in this naturally reductive class would have
$u=1/3$, contradicting $u>1/3$.  Hence the asymmetric Einstein metrics are
not naturally reductive for the displayed group.

Theorem~\ref{EX-fullisom} identifies the connected full isometry group,
modulo its finite ineffective kernel, with $H^3$.  If one of these metrics
were naturally reductive for some connected transitive group, its connected
transvection group would be a transitive normal subgroup of $H^3$.  Its Lie
algebra is therefore a sum of some of the three simple $\mathfrak h$-ideals.
A proper such subgroup has dimension at most $2\dim H$, whereas
\[
 \dim M=3\dim H-\dim K>2\dim H.
\]
Thus the transvection group would have to be all of $H^3$, which the
invariant-form calculation above excludes.  Hence no connected transitive
naturally reductive presentation exists.
\end{proof}

\subsubsection{Connected full isometry group}

Let $g$ be one of the exceptional asymmetric Einstein metrics and write
\[
 G=H^3,\qquad I_0=\operatorname{Isom}_0(M,g).
\]
At Lie-algebra level the normalizer of $\Delta\mathfrak k$ in
$\mathfrak h^3$ is exactly $\Delta\mathfrak k$: indeed the centralizer of
the irreducible symmetric subalgebra $\mathfrak k\subset\mathfrak h$ is
zero and $\mathfrak k$ is self-normalizing.  Consequently the connected
centralizer of the effective $G$-action is trivial.

\begin{lemma}[primitive exceptional enlargements are excluded]
\label{EX-primitive}
No effective primitive component in Onishchik's reduction can properly
enlarge one of the three exceptional simple ideals
$\mathfrak h=E_6,E_7,E_8$ occurring here.
\end{lemma}

\begin{proof}
A proper effective primitive enlargement would yield a proper simple compact
factorization
\[
 \mathfrak s=\mathfrak h+\mathfrak b
\]
with $\mathfrak h$ one of $E_6,E_7,E_8$.  The complete compact
factorization list contains no such factorization with an exceptional
$E_6,E_7$, or $E_8$ summand.  Hence no primitive enlargement occurs.
\end{proof}

\begin{lemma}[reducible exceptional enlargements are excluded]
\label{EX-reducible}
No proper reducible connected enlargement of the displayed $H^3$ action
preserves a metric in the detected diagonal family.
\end{lemma}

\begin{proof}
After primitive components are excluded, the same diagonal-replication
argument used in the classical three-factor families applies.  Every ambient
simple ideal met nontrivially by an original $\mathfrak h$-factor is another
copy of $\mathfrak h$; two commuting original factors cannot project onto
the same simple ambient ideal, and an ambient ideal missed by all three
would centralize the original action.  Thus any proper enlargement contains
a minimal transitive subgroup
\[
 J\quad\text{with}\quad \Lie(J)\simeq\mathfrak h^4
\]
containing the original $\mathfrak h^3$.

Write
\[
 \mathfrak j=
 \mathfrak h_A\oplus\mathfrak h_B\oplus
 \mathfrak h_C\oplus\mathfrak h_D,
\]
with one original factor diagonal in the $A,B$ block.  Since
$\mathfrak k\subset\mathfrak h$ is a maximal irreducible symmetric
subalgebra in each of the three exceptional pairs, the same two-factor
Goursat argument as in Lemma~\ref{AI-lem:minimal-H4} gives the isotropy in
the form
\[
 \mathfrak m=
 \Delta_{\{A\}\cup S}\mathfrak h
 \oplus
 \Delta_{\{B\}\cup S^c}\mathfrak k
\]
for a nonempty subset $S\subset\{C,D\}$.

One quotient factor is therefore
\[
 \frac{H^{\,1+|S|}}{\Delta H}.
\]
Its tangent multiplicity space is
\[
 V_m=\{(t_1,\dots,t_m):\textstyle\sum_i t_i=0\},
 \qquad m=1+|S|\in\{2,3\}.
\]
Every $J$-invariant metric restricts on this factor as
\[
 \langle\, ,\,\rangle_V\otimes Q.
\]
With
\[
 \bar e_i=e_i-\frac1m(1,\dots,1),
\]
the corresponding original horizontal directions have mutual inner
products
\[
 \langle\bar e_i,\bar e_j\rangle_VQ.
\]
The detected family makes the three original horizontal copies pairwise
orthogonal, so
\[
 \langle\bar e_i,\bar e_j\rangle_V=0\qquad(i\ne j).
\]
But
\[
 \bar e_1+\cdots+\bar e_m=0
\]
with all $\bar e_i\ne0$, and positive definiteness would give
\[
 0=\Bigl\|\sum_i\bar e_i\Bigr\|_V^2
   =\sum_i\|\bar e_i\|_V^2>0,
\]
a contradiction.
\end{proof}

\begin{theorem}[connected full isometry algebra]\label{EX-fullisom}
For either asymmetric Einstein metric on each exceptional three-factor
space,
\[
 \operatorname{Lie}\operatorname{Isom}_0(M,g)=\mathfrak h^3.
\]
Equivalently, modulo the finite ineffective kernel of the displayed action,
$H^3$ is the connected full isometry group.
\end{theorem}

\begin{proof}
Lemma~\ref{EX-primitive} excludes every effective primitive enlargement.
If a proper reducible enlargement remained,
Lemma~\ref{EX-reducible} would produce a forbidden minimal four-factor
extension.  Hence the connected isometry algebra is exactly
$\mathfrak h^3$.
\end{proof}

By Theorem~\ref{EX-fullisom}, a naturally reductive presentation by a
strictly larger connected transitive group is impossible.  Together with
Proposition~\ref{EX-notNR}, this excludes natural reductivity altogether.

\subsubsection{Pairwise distinction}

\begin{proposition}[pairwise non-isometry]\label{EX-pairwise}
For each exceptional homogeneous pair, the metric with $1/6<u<1$ and the
metric with $u>1$ are not related by a positive homothety and a Riemannian
isometry.
\end{proposition}

\begin{proof}
An isometry between homothetic copies conjugates the connected full isometry
groups.  After composing with a translation, it fixes the base point and
induces an automorphism of the homogeneous pair
\[
 (\mathfrak h^3,\Delta\mathfrak k).
\]
Such an automorphism permutes the three simple factors and applies
automorphisms within the individual factors.  Only the factor permutation
acts on the two-dimensional multiplicity coordinate.

The $S_3$-orbit of
\[
 \mathbb Rv_1=\mathbb R(e_1-e_2)
\]
consists of the three root lines
\[
 \mathbb R(e_1-e_2),\qquad
 \mathbb R(e_1-e_3),\qquad
 \mathbb R(e_2-e_3),
\]
and none is $\mathbb Rv_2$.  Since $u\ne1$, the two eigendirections of the
vertical metric endomorphism
\[
 \operatorname{diag}(u,1)
\]
are intrinsic.  Hence an automorphism carrying one detected metric to a
homothetic copy of the other preserves the $v_2$ line.  Comparison on that
line gives homothety factor $1$, and comparison on $v_1$ then forces
\[
 u_-=u_+,
\]
contrary to $u_-<1<u_+$.
\end{proof}

\subsubsection{Comparison with the standard constructions}

The general aligned-space theory supplies the ambient framework for these
homogeneous spaces, while the cited detailed classifications concern two
simple transitive factors.  The present spaces are not Ledger--Obata spaces
because their diagonal isotropy is the proper symmetric subgroup
$K\subsetneq H$.  The product and natural-reductivity possibilities have
been excluded above.

\subsubsection{Comparison with the literature}

The underlying homogeneous spaces are not new as aligned spaces.  Indeed,
the general aligned-space framework of Lauret--Will \cite[Example~2.4]{LWaligned}
contains
\[
 H^s/\Delta K
\]
for arbitrary $s$, records the diffeomorphism
\[
 H^s/\Delta K\simeq(H/K)\times H^{s-1},
\]
and identifies the previously studied $H\times H/\Delta K$ problem as the
case $s=2$.  Thus neither the three-factor homogeneous spaces themselves nor
the general Ricci formalism is claimed to be new.

What is specific here is the Einstein geometry of the
$\Gamma_{3,2}$-fixed five-scale family on
\[
 E_6^3/\Delta\Sp(4),\qquad
 E_7^3/\Delta\SU(8),\qquad
 E_8^3/\Delta\SO(16),
\]
with the standard harmless convention that the displayed group notation is
understood up to the finite central quotients required by the corresponding
compact symmetric pairs.  The cited two-factor aligned classification and
the dedicated $H\times H/\Delta K$ results do not classify these
three-factor homogeneous pairs.  Ledger--Obata spaces have $K=H$, whereas
all three exceptional cases above have proper symmetric isotropy
$K\subsetneq H$.  Product and naturally reductive presentations have been
excluded intrinsically in the preceding subsections.

To the best of our knowledge, no earlier construction gives either member
of any of the three asymmetric pairs above.  Accordingly, the six Einstein
metrics of Theorem~\ref{EX-main} are new to the best of our knowledge.  This
claim concerns the metrics rather than the underlying aligned spaces or
their standard symmetric factors.

\begin{theorem}[exceptional comparison]\label{EX-comparison}
For each of
\[
 E_6^3/\Delta\Sp(4),\qquad
 E_7^3/\Delta\SU(8),\qquad
 E_8^3/\Delta\SO(16),
\]
the two asymmetriv Einstein metrics are pairwise non-homothetic and
non-isometric.  Neither is a product metric or naturally reductive, and no
larger connected naturally reductive presentation exists.  The cited
two-factor aligned and $H\times H/\Delta K$ classification theorems do not
classify these three-factor homogeneous pairs.  To the best of our knowledge, none of the six asymmetric Einstein metrics
occurs in the classifications cited above.
\end{theorem}

\subsection{Cartan sign symmetries and the full invariant metric space}

The apparent restriction to horizontal diagonal metrics has a canonical
geometric origin.  Let $\sigma$ be the involution of the compact symmetric
pair $H/K$, so that
\[
 d\sigma|_{\mathfrak k}=+\operatorname{Id},
 \qquad
 d\sigma|_{\mathfrak q}=-\operatorname{Id}.
\]
For each factor of $H^s$ apply $\sigma$ in that factor and the identity in
the others.  Since $\sigma$ fixes $K$, these automorphisms preserve
$\Delta K$ and descend to $H^s/\Delta K$.

\begin{theorem}[Cartan fixed-point principle]
\label{thm:Cartan-fixed-point}
Let $H/K$ be an irreducible compact symmetric pair with $H$ and $K$ simple,
and put
\[
 M_s=H^s/\Delta K,\qquad
 T_s=\langle\tau_1,\ldots,\tau_s\rangle\cong(\mathbb Z_2)^s,
\]
where $\tau_i$ applies the symmetric-pair involution in the $i$th factor.
At the origin,
\[
 \mathfrak m
 \simeq
 \bigoplus_{i=1}^s\mathfrak q_i
 \oplus(\mathfrak k\otimes V_s),
 \qquad
 V_s=\{(t_1,\ldots,t_s):\sum_i t_i=0\}.
\]
The $T_s$-fixed $H^s$-invariant metrics are exactly the metrics
\[
 (x_1,\ldots,x_s;P),
 \qquad
 x_i>0,\qquad P\in\operatorname{Sym}^+(V_s),
\]
where the horizontal copies $\mathfrak q_i$ are mutually orthogonal, the
horizontal and vertical parts are orthogonal, and the metric on
$\mathfrak k\otimes V_s$ is
$(-B_{\mathfrak k})\otimes P$.

Moreover, at a $T_s$-fixed Einstein metric the full invariant
Einstein--Hilbert Hessian preserves the tangent space of this fixed metric
family.  Consequently a zero mode of the Hessian restricted to the
$T_s$-fixed family is a zero mode of the full invariant Hessian.
\end{theorem}

\begin{proof}
The automorphism $\tau_i$ acts by $-1$ on $\mathfrak q_i$, by $+1$ on
$\mathfrak q_j$ for $j\ne i$, and by $+1$ on
$\mathfrak k\otimes V_s$.  If a $T_s$-invariant bilinear form had a cross
term between $\mathfrak q_i$ and either $\mathfrak q_j$, $j\ne i$, or the
vertical block, invariance under $\tau_i$ would change the sign of that
term.  Hence every such cross term vanishes.  Irreducibility of
$\mathfrak q$ gives one positive scale $x_i$ on each horizontal copy.
Since $K$ is simple, the invariant inner products on the adjoint
$\mathfrak k$-factor are scalar multiples of $-B_{\mathfrak k}$, leaving an
arbitrary positive definite multiplicity form $P$ on $V_s$.  This proves the
description of the fixed metric family.

Normalized scalar curvature on the full finite-dimensional space of
$H^s$-invariant unit-volume metrics is $T_s$-invariant.  At a fixed Einstein
metric its Hessian therefore commutes with $T_s$, so the $T_s$-fixed tangent
space is invariant under the Hessian.  If $v$ belongs to that tangent space
and is in the kernel of the restricted Hessian, then the full Hessian vector
$Hv$ is again fixed and is orthogonal to every fixed vector.  Taking the
fixed test vector to be $Hv$ gives $Hv=0$.
\end{proof}

\begin{lemma}[no hidden equivariant gauge directions]\label{lem:no-equivariant-gauge}
Let $H/K$ be an effective irreducible compact symmetric pair with $H$ and
$K$ simple, and put $G=H^s$, $L=\Delta K$.  Then
\[
 N_{\mathfrak h^s}(\Delta\mathfrak k)=\Delta\mathfrak k.
\]
Consequently the connected group of $G$-equivariant diffeomorphisms of
$G/L$ is trivial.  More generally, the invariant metric tangent space
contains no nonzero infinitesimal pure-gauge tensor $\mathcal L_Xg$.
After fixing homothety, the response kernel therefore contains neither a
scale direction nor a diffeomorphism-gauge direction.
\end{lemma}

\begin{proof}
If $(X_1,\ldots,X_s)$ normalizes $\Delta\mathfrak k$, then for every
$A\in\mathfrak k$ the brackets $[X_i,A]$ are independent of $i$.  They
define a derivation of the simple algebra $\mathfrak k$, hence an inner
derivation $\operatorname{ad}_Y$ for some $Y\in\mathfrak k$.  Thus
$X_i-Y$ centralizes $\mathfrak k$ in $\mathfrak h$.

For an effective irreducible compact symmetric pair, the centralizer of
$\mathfrak k$ in $\mathfrak h$ is zero.  Equivalently,
$\mathfrak n_{\mathfrak h}(\mathfrak k)
=\mathfrak k+\mathfrak c_{\mathfrak h}(\mathfrak k)$ and the symmetric
subalgebra $\mathfrak k$ is self-normalizing.  Hence $X_i=Y$ for every
$i$, proving the normalizer identity.

The connected group of $G$-equivariant diffeomorphisms is
$N_G(L)_0/L$, whose Lie algebra is the quotient of the two algebras above,
and is therefore trivial.

It remains to exclude an invariant pure-gauge tensor generated by a vector
field that is not itself $G$-invariant.  Suppose
\[
 h=\mathcal L_Xg
\]
is $G$-invariant.  Averaging $X$ over the compact transitive group $G$ gives
a $G$-invariant vector field $\bar X$ and, because $g$ and $h$ are
$G$-invariant,
\[
 h=\mathcal L_{\bar X}g.
\]
But a $G$-invariant vector field is determined at the origin by an
$L$-fixed vector in
\[
 \mathfrak m=
 \bigoplus_i\mathfrak q_i\oplus(\mathfrak k\otimes V_s).
\]
The irreducible nontrivial isotropy module $\mathfrak q$ has no fixed vector,
and the adjoint module of the simple algebra $\mathfrak k$ has none either.
Thus $\mathfrak m^L=0$, so $\bar X=0$ and hence $h=0$.
\end{proof}

\subsection{Four-factor geometric preliminaries}
For later reference, the fully $S_4$-symmetric scale-fixed Einstein equation
is
\begin{equation}\label{eq:H4sym}
 (6\rho-3)x^2-8\rho x+3\rho+2=0.
\end{equation}
The full-isometry reduction needed to distinguish the resulting metrics is
recorded next.

\subsection{Full isometry reduction for the geometric four-factor families}
\label{sec:H4-fullisom}

The preceding proposition only used automorphisms of the displayed
homogeneous pair.  We now show that, for the geometric AI, AII and
exceptional Cartan families occurring below, this already controls arbitrary
Riemannian isometries.  The proof is independent of the Einstein equation:
it uses only the Cartan-fixed orthogonality of the four horizontal
$\mathfrak q$-summands.

\begin{lemma}[normalizer and centralizer]\label{lem:H4-normalizer}
Let $H/K$ be an effective irreducible compact symmetric pair with $H$ and
$K$ simple.  Then
\[
 N_{\mathfrak h^4}(\Delta\mathfrak k)=\Delta\mathfrak k.
\]
Consequently the connected centralizer of the effective $H^4$-action on
$H^4/\Delta K$ is trivial.
\end{lemma}

\begin{proof}
If $(X_1,\ldots,X_4)$ normalizes $\Delta\mathfrak k$, then for every
$A\in\mathfrak k$ the four brackets $[X_i,A]$ agree.  Thus
$X_i-X_j$ centralizes $\mathfrak k$ in $\mathfrak h$.  For an effective
irreducible compact symmetric pair with $H$ simple one has
\[
 \mathfrak c_{\mathfrak h}(\mathfrak k)=0,
 \qquad
 \mathfrak n_{\mathfrak h}(\mathfrak k)=\mathfrak k.
\]
Hence all $X_i$ are equal to one element of $\mathfrak k$, proving the
normalizer identity.  The usual normalizer description of the centralizer
of a transitive action gives the final assertion.
\end{proof}

\begin{lemma}[minimal five-factor extension]\label{lem:H4-minimal-H5}
Suppose that the primitive simple enlargements of the four original
$\mathfrak h$-ideals have been excluded.  If a connected compact transitive
isometry algebra $\mathfrak i$ properly contains
\[
 \mathfrak g=\mathfrak h^4,
\]
then the connected isometry group contains a transitive subgroup $J$ with
\[
 \Lie(J)\simeq\mathfrak h^5
\]
which contains the original $\mathfrak h^4$ action.
\end{lemma}

\begin{proof}
The centre of $\mathfrak i$ centralizes the transitive $\mathfrak h^4$ action
and is therefore zero by Lemma~\ref{lem:H4-normalizer}.  Since a compact Lie
algebra is reductive, $\mathfrak i$ is semisimple.  We may therefore
decompose it into simple ideals.  A nonzero projection of one original
simple ideal $\mathfrak h$ to an ambient simple ideal is injective.
If its image were proper, one would obtain a primitive simple enlargement,
contrary to the hypothesis.  Hence every nonzero projection is onto another
copy of $\mathfrak h$.  Two commuting original ideals cannot project
nontrivially to the same nonabelian simple ambient ideal.  An ambient ideal
missed by all four original factors would centralize $\mathfrak g$, which is
impossible by Lemma~\ref{lem:H4-normalizer}.  Therefore, after reordering,
\[
 \mathfrak i=\mathfrak h^{r_1}\oplus\cdots\oplus\mathfrak h^{r_4},
\]
where the $j$th original factor is diagonally embedded in its block.  If
$\mathfrak i\ne\mathfrak g$, some $r_j\ge2$.  Replacing that block by one
selected factor together with the diagonal copy in all remaining factors
produces a subalgebra $\mathfrak j\simeq\mathfrak h^5$ containing
$\mathfrak g$.  The corresponding group contains the transitive group
$H^4$ and is therefore transitive.
\end{proof}

\begin{lemma}[horizontal orthogonality excludes the five-factor extension]
\label{lem:H4-H5-obstruction}
Let $g$ be an $H^4$-invariant metric in the Cartan-fixed family of
Theorem~\ref{thm:Cartan-fixed-point}; in particular the four horizontal
copies $\mathfrak q_1,\ldots,\mathfrak q_4$ are pairwise orthogonal.  Then
$g$ is not invariant under a proper five-factor extension supplied by
Lemma~\ref{lem:H4-minimal-H5}.
\end{lemma}

\begin{proof}
Write
\[
 \mathfrak j=
 \mathfrak h_A\oplus\mathfrak h_B\oplus
 \mathfrak h_C\oplus\mathfrak h_D\oplus\mathfrak h_E,
\]
where one original $\mathfrak h$-factor is diagonal in the $A,B$ block and
the other three are the $C,D,E$ factors.  If $\mathfrak m$ denotes the
isotropy of the transitive $J$-action, then
\[
 \dim\mathfrak m=5\dim\mathfrak h-\dim(H^4/\Delta K)
 =\dim\mathfrak h+\dim\mathfrak k.
\]
Project $\mathfrak m$ to $\mathfrak h_A\oplus\mathfrak h_B$.  Exactly as in
the two-factor Goursat argument of
Lemma~\ref{AI-lem:minimal-H4}, irreducibility of the symmetric complement
implies that each projection is either $\mathfrak k$ or $\mathfrak h$; one
projection must be onto $\mathfrak h$, and the dimension equality forces
\[
 \mathfrak m\simeq\mathfrak h\oplus\mathfrak k.
\]
The simple $\mathfrak h$-ideal of $\mathfrak m$ projects by either zero or
an automorphism to each of $C,D,E$.  Let
\[
 \varnothing\ne S\subset\{C,D,E\}
\]
be its nonzero support.  The condition
$\mathfrak g\cap\mathfrak m=\Delta\mathfrak k$ then forces, up to the
chosen automorphisms,
\[
 \mathfrak m=
 \Delta_{\{A\}\cup S}\mathfrak h
 \oplus
 \Delta_{\{B\}\cup S^c}\mathfrak k.
\]
Thus the quotient contains the Ledger--Obata factor
\[
 \frac{H^{1+|S|}}{\Delta H},
 \qquad 1+|S|\in\{2,3,4\}.
\]
Its tangent multiplicity space is
\[
 V_m=\{(t_1,\ldots,t_m):\textstyle\sum_i t_i=0\}.
\]
For a $J$-invariant metric the restriction to this factor is
$\langle\, ,\,\rangle_V\otimes Q$.  If
\[
 \bar e_i=e_i-\frac1m(1,\ldots,1),
\]
the original horizontal directions represented in this factor have mutual
inner products
\[
 \langle\bar e_i,\bar e_j\rangle_VQ.
\]
The Cartan-fixed metric makes the corresponding original horizontal
$\mathfrak q$-copies pairwise orthogonal.  Hence all the relevant
$\bar e_i$ are pairwise orthogonal.  But
\[
 \bar e_1+\cdots+\bar e_m=0,
 \qquad \bar e_i\ne0,
\]
so positive definiteness gives the contradiction
\[
 0=\Bigl\|\sum_i\bar e_i\Bigr\|_V^2
  =\sum_i\|\bar e_i\|_V^2>0.
\]
\end{proof}

\begin{theorem}[connected full isometry group in the four-factor Cartan
families]\label{thm:H4-fullisom}
Let $H/K$ be one of
\[
 \SU(n)/\SO(n),\quad n=3\text{ or }n\ge5,
\]
\[
 \SU(2n)/\Sp(n),\quad n\ge9,
\]
or
\[
 E_6/\Sp(4),\qquad E_7/\SU(8),\qquad E_8/\SO(16).
\]
Let $g$ be any positive $H^4$-invariant metric in the Cartan-fixed family
of Theorem~\ref{thm:Cartan-fixed-point}.  Then
\[
 \Lie\Isom_0(H^4/\Delta K,g)=\mathfrak h^4.
\]
Equivalently, the effective image of the displayed $H^4$ action is the
connected full isometry group.
\end{theorem}

\begin{proof}
For AI with $n\ge5$, the factorization obstruction
Lemma~\ref{AI-lem:factorobs} excludes a primitive enlargement of any
$A_{n-1}$ ideal while carrying the standard $\mathfrak{so}(n)$ in the
intersection.  The argument is local to one original ideal and is unchanged
when three factors are replaced by four.  The same observation applies to
AII by Lemma~\ref{AII-lem:factorobs}.  For the three exceptional pairs,
Onishchik's compact factorization list contains no proper simple
factorization having $E_6,E_7$, or $E_8$ as a proper summand, exactly as in
Lemma~\ref{EX-primitive}.

For the low-rank AI case $n=3$, the only possible primitive enlargement is
\[
 A_3=C_2+A_2,
 \qquad C_2\cap A_2=A_1.
\]
A $C_2$ ideal of the strongly semisimple isotropy supported on a proper
subset of the four enlarged coordinates would contribute an $A_1$ on that
proper subset to the intersection with $A_2^4$, contradicting the diagonal
isotropy.  Hence any such $C_2$ must be diagonal across all four enlarged
factors.  Moreover the standard $C_2=\mathfrak{sp}(2)$ acts irreducibly on
$\mathbb C^4$; Schur's lemma therefore gives
\[
 C_{A_3}(C_2)=0.
\]
Consequently two distinct full-support commuting $C_2$ partners cannot
occur.  A single diagonal partner is too small, since
\[
 \dim(A_2^4+C_2^\Delta)
 =4\cdot8+10-3=39<60=4\cdot15.
\]
Thus primitive enlargement is impossible here as well.

After primitive components are excluded,
Lemma~\ref{lem:H4-minimal-H5} reduces every proper connected enlargement to
a transitive five-factor extension.  Lemma~\ref{lem:H4-H5-obstruction}
excludes such an extension for every Cartan-fixed metric.  Therefore the
connected full isometry algebra is exactly $\mathfrak h^4$.
\end{proof}

\begin{remark}[disconnected components]\label{rem:H4-components}
Theorem~\ref{thm:H4-fullisom} is exactly what is needed for arbitrary
isometry questions.  Indeed every component of the full isometry group has
a representative fixing the origin: compose an isometry with a suitable
left translation.  Such a representative normalizes the characteristic
identity component $H^4_{\rm eff}$ and therefore induces an automorphism of
$\mathfrak h^4$ preserving $\Delta\mathfrak k$.  The kernel of this action
is trivial, because an isometry fixing the origin and commuting with the
transitive $H^4_{\rm eff}$ action fixes every point.  Thus the disconnected
part of the full isometry group is a finite subgroup of the automorphism
group of the homogeneous pair, namely the subgroup preserving the metric.
No separate classification of those finite components is required for the
non-isometry result below.
\end{remark}

\begin{corollary}[arbitrary non-isometry of the $3+1$ and $2+2$ branches]
\label{cor:H4-arbitrary-nonisometry}
Fix one of the geometric Cartan pairs in
Theorem~\ref{thm:H4-fullisom}.  Suppose that at the same Cartan parameter
there are local branch metrics $g_{31}$ and $g_{22}$, sufficiently close to
the bifurcation point, whose horizontal equality partitions are respectively
$3+1$ and $2+2$.  Then there are no
$c>0$ and no Riemannian isometry $F$ such that
\[
 F^*g_{22}=c\,g_{31}.
\]
In particular, whenever both local symmetry types are geometrically
realized on the same four-factor homogeneous space in the local regime,
they define genuinely non-homothetic, non-isometric Riemannian metrics.
\end{corollary}

\begin{proof}
Assume $F^*g_{22}=c\,g_{31}$.  A constant homothety does not change the
isometry group, so $F$ conjugates the identity components of the two full
isometry groups.  By Theorem~\ref{thm:H4-fullisom}, both identity components
are the effective images of $H^4$.  Compose $F$ with a left translation so
that the resulting isometry fixes the origin.  It then normalizes the
connected $H^4$ action and hence induces an automorphism of
$\mathfrak h^4$ preserving the isotropy $\Delta\mathfrak k$.  Thus it is,
at Lie-algebra level, an automorphism of the homogeneous pair
$(H^4,\Delta K)$.

Such an automorphism acts on the four simple ideals by a permutation,
together with factorwise automorphisms preserving the diagonal isotropy.
It therefore preserves the unordered equality partition of the four
horizontal scales.  A nonzero $3+1$ metric has partition $3+1$, whereas a
nonzero balanced metric has partition $2+2$; these partitions are not
conjugate under $S_4$.  Hence no positive homothety can carry one metric to
the other.  This contradiction proves the claim.
\end{proof}

\begin{remark}[geometric scope]\label{rem:H4-nonisometry-scope}
The full-isometry theorem applies to every actual Cartan-fixed metric on the
listed geometric homogeneous spaces.  Thus whenever both fixed symmetry
types occur at the same Cartan value, the preceding corollary gives arbitrary
Riemannian non-isometry, not merely inequivalence under the displayed group.
The first such simultaneous realization is proved below at the $AII$ value
$\rho=9/8$.
\end{remark}

The balanced $2+2$ branch has an additional symmetry.  Interchanging the two
blocks sends its amplitude $a$ to $-a$ while leaving $\rho$ unchanged.
Consequently, if the branch is written locally as
\[
 \rho=\rho(a),
\]
then
\[
 \rho(a)=\rho(-a),
 \qquad\text{hence}\qquad
 \rho'(0)=0.
\]
Thus the balanced branch is pitchfork-symmetric at leading order.  The
remaining nonlinear datum can be isolated intrinsically as follows.  First
recenter the scale-fixed Einstein system on the $\Gamma_{4,2}$-fixed family
along the smooth symmetric solution branch, so that in the resulting
coordinates
\[
 F(0,\rho)=0
\]
for all $\rho$ near $\rho_{4,*}$.  Thus the symmetric solution at
$\rho_{4,*}$ is $(y,\rho)=(0,\rho_{4,*})$.  Write
\[
 L=D_yF(0,\rho_{4,*}),
\]
choose a kernel vector $v$ and a left-kernel vector $\psi$ with
$\langle\psi,v\rangle=1$, and let $Q=I-v\otimes\psi$.  The restriction
$L:QY\to QY$ is invertible.  Define
\[
 w_2=-\frac12\,(L|_{QY})^{-1}Q\,D_y^2F[v,v].
\]
Then Liapunov--Schmidt reduction gives
\begin{equation}\label{eq:H4balancednormalform}
 \phi(a,\rho)
 =
 a\left(
 \alpha(\rho-\rho_{4,*})+\beta a^2
 +O\!\left((\rho-\rho_{4,*})^2+
 |a|^2|\rho-\rho_{4,*}|+a^4\right)\right),
\end{equation}
where
\begin{equation}\label{eq:H4balancedcoefficients}
 \alpha
 =
 \left\langle\psi,D_\rho D_yF[v]\right\rangle,
 \qquad
 \beta
 =
 \left\langle\psi,
 \frac16D_y^3F[v,v,v]+D_y^2F[v,w_2]
 \right\rangle .
\end{equation}
The transverse response crossing proved above is exactly $\alpha\ne0$.
Consequently, if $\beta\ne0$, the balanced branch satisfies
\begin{equation}\label{eq:H4balancedside}
 \rho(a)
 =
 \rho_{4,*}-\frac{\beta}{\alpha}a^2+O(a^4).
\end{equation}
Thus its side is determined by the single coordinate-independent sign
$-\beta/\alpha$.

\paragraph{The balanced fixed equations.}
To determine the side on which the balanced branch emerges, consider the
nonlinear fixed equations in the orthonormal anti-diagonal basis
\[
 u_1=\frac1{\sqrt2}(1,-1,0,0),\qquad
 u_2=\frac1{\sqrt2}(0,0,1,-1),\qquad
 u_3=\frac12(1,1,-1,-1),
\]
write a $\Gamma_{4,2}$-fixed metric, after homothety, as
\[
 (p,p,q,q;A,B,1).
\]
Normalize $d=\dim\mathfrak q$ to one and put $e=\rho$.  The
simple-isotropy identity $e(1-a)=d/2$ gives
\[
 ea=\rho-\frac12.
\]
The only nonzero vertical cubic coefficients, up to permutation of the
indices, are
\[
 \left(\sum_i u_{1i}^2u_{3i}\right)^2
 =
 \left(\sum_i u_{2i}^2u_{3i}\right)^2
 =\frac14.
\]
The homogeneous Ricci formula therefore gives
\begingroup\scriptsize
\begin{align}
 r_p&=\frac{8p-2A-1}{16p^2},
&
 r_q&=\frac{8q-2B-1}{16q^2},
\label{eq:H4bal-rpq}\\[1mm]
 r_A&=
 \frac{2A^3+(8\rho-4)Ap^2+(1-2\rho)p^2}
 {16A^2p^2\rho},
&
 r_B&=
 \frac{2B^3+(8\rho-4)Bq^2+(1-2\rho)q^2}
 {16B^2q^2\rho},
\label{eq:H4bal-rAB}\\[1mm]
 r_C&=
 \frac{2\rho-1}{8\rho}
 +\frac1{16\rho}\left(\frac1{p^2}+\frac1{q^2}\right)
 +\frac{2\rho-1}{32\rho}
 \left(\frac1{A^2}+\frac1{B^2}\right).
\label{eq:H4bal-rC}
\end{align}
\endgroup
At $p=q=x$ and $A=B=1$, these reproduce the symmetric equation
\eqref{eq:H4sym}.

The Einstein equations on this fixed family are
\begin{equation}\label{eq:H4bal-F}
\begin{aligned}
 F_1&=\frac{r_p-r_q}{2}=0,\qquad&
 F_2&=\frac{r_A-r_B}{2}=0,\\
 F_3&=\frac{r_p+r_q}{2}-r_C=0,\qquad&
 F_4&=\frac{r_A+r_B}{2}-r_C=0.
\end{aligned}
\end{equation}
They are invariant under
\[
 (p,q,A,B)\longleftrightarrow(q,p,B,A).
\]
Accordingly put
\begin{equation}\label{eq:H4bal-expansion}
\begin{aligned}
 p&=m+h,&q&=m-h,\\
 A&=y+k,&B&=y-k,
\end{aligned}
\qquad
\begin{aligned}
 m&=x+m_2h^2+O(h^4),\\
 y&=1+y_2h^2+O(h^4),\\
 k&=k_1h+k_3h^3+O(h^5),\\
 \rho&=\rho_*+r_2h^2+O(h^4).
\end{aligned}
\end{equation}
The coefficient of $h$ in $F_1$ gives
\begin{equation}\label{eq:H4bal-k1}
 k_1=-\frac{4x-3}{x}.
\end{equation}
The coefficient of $h$ in $F_2$, together with the symmetric equation,
then gives
\begin{equation}\label{eq:H4bal-critical-relations}
 \rho_*=\frac{20x-9}{32x^2-36x+9},
 \qquad
 P(x)=32x^4-76x^3+59x^2-20x+3=0.
\end{equation}

To determine $r_2$, substitute
\eqref{eq:H4bal-k1}--\eqref{eq:H4bal-critical-relations} into the
coefficients of $h^2$ in $F_3,F_4$ and of $h^3$ in $F_1,F_2$, and reduce
all coefficients modulo $P(x)$.  With
\[
 z=(m_2,y_2,k_3,r_2)^T
\]
one obtains the exact system
\begin{equation}\label{eq:H4bal-linear-system}
 \mathcal M(x)z=b(x),
\end{equation}
where $\mathcal M(x)$ and $b(x)$ are polynomial in $x$ of degree at
most three.  Their coefficients are recorded in
Appendix~\ref{app:H4balanced-coefficients}; they are not needed for the
conceptual argument here.  What matters in the main text is the exact
Cramer reduction below.

Let
\[
 \Delta=\det\mathcal M,\qquad
 \Delta_\rho=\det\mathcal M_\rho,
\]
where $\mathcal M_\rho$ is obtained by replacing the fourth column of
$\mathcal M$ by $b$.  Reduction modulo $P$ gives
\begin{equation}\label{eq:H4bal-Cramer-reduction}
 \Delta\equiv-\frac{C_*(x)B_*(x)}{256}\pmod{P},
 \qquad
 \Delta_\rho\equiv2C_*(x)A_*(x)\pmod{P},
\end{equation}
where
\[
 C_*(x)=1519996x^3-1703953x^2+667668x-114057.
\]
Consequently, at the critical root $x_*$,
\begin{equation}\label{eq:H4bal-r2}
 r_2=\frac{\Delta_\rho(x_*)}{\Delta(x_*)}
 =-512\,\frac{A_*(x_*)}{B_*(x_*)}.
\end{equation}
The Cramer denominator is nonzero: indeed
\[
 C_*(5/4)=\frac{8214749}{8}>0,
\]
and $C_*'$ is a quadratic with positive leading coefficient and negative
discriminant, hence $C_*(x)>0$ for $x\ge5/4$; the proof below shows
$B_*(x_*)>0$.

\begin{theorem}[side of the balanced four-factor branch]
\label{thm:H4balancedside}
For the continuous four-factor algebraic system, the
$\Gamma_{4,2}$-fixed balanced branch issuing from $\rho_{4,*}$ points
toward the
lower-$\rho$ side.  More precisely, with a kernel coordinate $h$ chosen so
that the two horizontal block scales are $x_*+h$ and $x_*-h$,
\[
 \rho(h)=\rho_{4,*}-c_*h^2+O(h^4),
 \qquad c_*>0.
\]
One may take
\[
 c_*
 =
 512\,\frac{A_*(x_*)}{B_*(x_*)},
\]
where
\[
 A_*(x)=1823444x^3-2041947x^2+800072x-136299,
\]
\[
 B_*(x)=93530764x^3-104635533x^2+40950468x-6975909,
\]
and $x_*$ is the relevant root of
\[
 P(x)=32x^4-76x^3+59x^2-20x+3.
\]
In particular,
\[
 \rho(h)<\rho_{4,*}
\]
for all sufficiently small $h\ne0$.
\end{theorem}

\begin{proof}
The expansion \eqref{eq:H4bal-expansion} is forced by block interchange.
Equations \eqref{eq:H4bal-k1}--\eqref{eq:H4bal-r2}, obtained directly from
the fixed Ricci equations \eqref{eq:H4bal-rpq}--\eqref{eq:H4bal-F}, give
\[
 \rho(h)=\rho_{4,*}
 -512\,\frac{A_*(x_*)}{B_*(x_*)}h^2+O(h^4).
\]
It remains to determine the sign.  The relevant root satisfies $x_*>5/4$:
\[
 P(5/4)=-\frac18<0,
\]
and
\[
 P'(x)=128x^3-228x^2+118x-20
\]
has exactly one real zero, at $1.0383\ldots<5/4$.  Since
$P'(5/4)=85/4>0$, the polynomial $P$ is strictly increasing on
$[5/4,\infty)$.

Furthermore,
\[
 A_*(5/4)=\frac{9877303}{8}>0,
 \qquad
 B_*(5/4)=\frac{507171433}{8}>0.
\]
Both $A_*'$ and $B_*'$ are quadratic polynomials with positive leading
coefficient and negative discriminant, so both derivatives are positive on
the real line.  Hence
\[
 A_*(x_*)>0,\qquad B_*(x_*)>0,
\]
and therefore
\[
 c_*=512\,\frac{A_*(x_*)}{B_*(x_*)}>0.
\]
\end{proof}

This is genuinely nonlinear information.  The linear response detects the
balanced kernel and forces the pitchfork symmetry, but it does not determine
the side of the wall.  The sign above is obtained only after the cubic
Lyapunov--Schmidt calculation.  Thus the four-factor analysis goes beyond
detection of a degeneracy and determines the local geometry of the
symmetry-breaking branch.

\begin{remark}
Numerically,
\[
 x_*=1.2557658237054698\ldots,\qquad
 \rho_{4,*}=1.1305217618711898\ldots,
\]
and
\[
 c_*=9.9713932738686\ldots .
\]
These decimal values are included only for orientation; the sign proof
above is exact.
\end{remark}

\begin{corollary}[the four-factor fixed families]
\label{cor:H4-fixed-families}
For $s=4$ the two local partition families are fixed-point families of
automorphisms of the full homogeneous space.

For the $3+1$ partition,
\[
 \Gamma_{4,1}=T_4\rtimes S_3,
\]
and, after homothety, its fixed metrics are exactly
\[
 (p,p,p,q;r,r,1)
\]
in the horizontal/anti-diagonal coordinates used below.  For the balanced
$2+2$ partition,
\[
 \Gamma_{4,2}=T_4\rtimes(S_2\times S_2),
\]
and, after homothety, its fixed metrics are exactly
\[
 (p,p,q,q;A,B,1).
\]
Consequently the $3+1$ and balanced $2+2$ solutions obtained in these
families are genuine critical points of normalized scalar curvature on the
full invariant metric space, not merely critical points of a chosen
diagonal ansatz.
\end{corollary}

\begin{proof}
For $3+1$,
\[
 V_4|_{S_3}\simeq V_3\oplus\mathbf1,
\]
so the vertical metric has two eigenvalues, represented in the chosen
anti-diagonal basis by $(r,r,1)$.  For $2+2$,
\[
 V_4|_{S_2\times S_2}
 \simeq(\operatorname{sgn}\boxtimes\mathbf1)
 \oplus(\mathbf1\boxtimes\operatorname{sgn})
 \oplus\mathbf1,
\]
so the vertical metric has three independent eigenvalues $(A,B,C)$; fixing
homothety by $C=1$ gives $(A,B,1)$.  The horizontal forms are forced by
Theorem~\ref{thm:Cartan-fixed-point} and the block permutations.  The final assertion follows directly from symmetric criticality for the
finite automorphism group defining the fixed family.
\end{proof}

At non-Cartan values of the continued parameter $\rho$, the same coordinate
families and equations are used algebraically, but there is no underlying
homogeneous space to which symmetric criticality could be applied.

\subsection{Global analysis on the $\Gamma_{4,1}$-fixed $3+1$ branch}

On the $\Gamma_{4,1}$-fixed $3+1$ family write the horizontal scales as
\[
 (p,p,p,q)
\]
and the three anti-diagonal scales as
\[
 (r,r,1).
\]
For symbolic $\rho$ the Einstein equations reduce to
\begin{align}
0={}&24p^2q-9p^2-24pq^2+8q^2r+q^2,\label{eq:H4F1}\\
0={}&20p^2r\rho-10p^2r-2p^2\rho+p^2
-24pr^2\rho+8r^3\rho+6r^3+r^2\rho,\label{eq:H4F2}\\
0={}&32p^2q^2r^2\rho-16p^2q^2r^2+4p^2q^2\rho-2p^2q^2
+9p^2r^2\notag\\
&-48pq^2r^2\rho+16q^2r^3\rho+2q^2r^2\rho+3q^2r^2.
\label{eq:H4F3}
\end{align}
Eliminating $p$ and $q$ and removing geometrically understood factors gives a
universal degree-ten detector
\[
 \mathcal D_\rho(r)=\sum_{j=0}^{10}a_j(\rho)r^j.
\]
Its coefficients are
{\small
\begin{align*}
a_{10}&=4096(2\rho-1)(2\rho+3)^2(4\rho+3)^4,\\
a_9&=-8192(2\rho+3)(4\rho+3)^2
 (80\rho^4+228\rho^3+72\rho^2-79\rho+42),\\
a_8&=2048(8064\rho^7+54400\rho^6+122104\rho^5+106884\rho^4
 +24982\rho^3-5813\rho^2-3441\rho-990),\\
a_7&=-1024(3264\rho^7+37744\rho^6+104112\rho^5+92752\rho^4
 +15864\rho^3-11667\rho^2-7882\rho-1551),\\
a_6&=-128(7152\rho^7-23176\rho^6-174516\rho^5-174706\rho^4
 -19406\rho^3+29799\rho^2+15812\rho+2237),\\
a_5&=128(2\rho-1)(1224\rho^6+4364\rho^5-7974\rho^4-17619\rho^3
 -9055\rho^2-1641\rho-3),\\
a_4&=32(2\rho-1)(348\rho^6-1984\rho^5+3031\rho^4+7591\rho^3
 +2189\rho^2-719\rho-264),\\
a_3&=-16(\rho+1)(2\rho-1)^2
 (150\rho^4+67\rho^3+621\rho^2+505\rho+73),\\
a_2&=-(\rho+1)^2(2\rho-1)^2
 (126\rho^3-131\rho^2-720\rho+33),\\
a_1&=2(\rho+1)^3(2\rho-1)^3(7\rho-9),\\
a_0&=(\rho+1)^4(2\rho-1)^3.
\end{align*}}
At the symmetric value $r=1$, the nonlinear detector restricts to the linear
response polynomial:
\begin{equation}\label{eq:H4D1}
 \mathcal D_\rho(1)=11664(3\rho+2)\Xi_{4,31}(\rho).
\end{equation}
Thus $\rho_{4,*}$ is exactly the parameter at which an asymmetric detector
root passes through the symmetric locus $r=1$.

A multiple positive detector root satisfies
\[
 \mathcal D_\rho(r)=0,
 \qquad
 \partial_r\mathcal D_\rho(r)=0.
\]
Direct elimination from \eqref{eq:H4F1}--\eqref{eq:H4F3} gives
\[
 \operatorname{Res}_q(F_1,F_3)=324p^4R_\rho(p,r),
\]
followed by
\[
 \operatorname{Res}_p(F_2,R_\rho)
 =
 16r^8(r-1)^2(2\rho-1)\mathcal D_\rho(r),
\]
so the displayed degree-ten detector is reproduced exactly from the fixed
Einstein equations.

Its discriminant factors as
\[
\begin{aligned}
\operatorname{Res}_r(\mathcal D_\rho,\partial_r\mathcal D_\rho)
={}&C(\rho+1)^{12}(2\rho-1)^{16}(2\rho+3)^6(4\rho+3)^8\\
&\times\Psi^{(4)}_{18}(\rho)^2\Phi^{(4)}_{25}(\rho),
\end{aligned}
\]
with $C>0$.  Here
\[
\begin{aligned}
\Psi^{(4)}_{18}={}&
746496\rho^{18}+4907520\rho^{17}+13814528\rho^{16}
+21043328\rho^{15}\\
&+17102464\rho^{14}+4503104\rho^{13}-4470480\rho^{12}
-5752984\rho^{11}\\
&-2836276\rho^{10}+494118\rho^9+1022904\rho^8+208006\rho^7\\
&-95947\rho^6-27540\rho^5+759\rho^4-490\rho^3
-233\rho^2+42\rho+9.
\end{aligned}
\]
After the shift $\rho=t+\frac12$, every coefficient of
$\Psi^{(4)}_{18}(t+\frac12)$ is positive, hence this factor never vanishes
for $\rho>1/2$.

For $\Phi^{(4)}_{25}$ the exact Sturm chain after the same shift has sign
strings
\[
 --+----+--++--+++++-++--++
 \quad(t=0)
\]
and
\[
 ++++--++---++-+++++-+++-++
 \quad(t=+\infty),
\]
with variation numbers $11$ and $10$.  Thus it has exactly one zero for
$\rho>1/2$:
\[
 \rho_{4,\mathrm{fold}}=1.13236356611068\ldots.
\]

\begin{theorem}[four-factor detector phase diagram]\label{thm:H4phase}
For $\rho>1/2$ the positive roots of $\mathcal D_\rho$ satisfy
\[
 \begin{array}{c|c}
 \rho>\rho_{4,\mathrm{fold}}&0\text{ positive roots},\\
 \rho=\rho_{4,\mathrm{fold}}&1\text{ positive double root},\\
 1/2<\rho<\rho_{4,\mathrm{fold}}&2\text{ positive simple roots}.
 \end{array}
\]
Moreover
\[
 \rho_{4,*}<\rho_{4,\mathrm{fold}}.
\]
Thus pair creation at the off-symmetric fold and crossing of the symmetric
branch are distinct nearby events.
\end{theorem}

\begin{proof}
The leading coefficient and $\mathcal D_\rho(0)$ are positive for
$\rho>1/2$.  Moreover
\[
 \mathcal D_\rho(1/6)>0.
\]
After shifting $\rho=t+\frac12$, the numerator of this value has coefficients
\[
 480200,\ 2175600,\ 3444396,\ 3271004,\ \frac{5834937}{2},
 \ 1247952,\ 595976,\ 129600,
\]
all positive.  Hence no positive root enters through $r=0$ or crosses the
reconstruction boundary $r=1/6$.

The discriminant certificate above shows that the only parameter at which
the number of positive roots can change is $\rho_{4,\mathrm{fold}}$.
Because the corresponding $\Phi^{(4)}_{25}$ zero is simple and occurs to
the first power in the discriminant, exactly one pair of simple roots
coalesces there, producing one double root.

At $\rho=1$ the exact Sturm sign strings at $r=0^+$ and $r=+\infty$ are
\[
 +--+-++--++,
 \qquad
 +++++-+++-+,
\]
with variation numbers $6$ and $4$, so there are two positive roots.  At
$\rho=2$ they are
\[
 ++-++--++--,
 \qquad
 +++--+-+++-,
\]
with variation number $5$ at both ends, so there are none.  This proves the
first table.

Finally \eqref{eq:H4D1} shows that the symmetric crossing occurs at the
unique zero $\rho_{4,*}$ of $\Xi_{4,31}$.  The order relative to the fold is
certified without decimals:
\[
 \Xi_{4,31}(1)<0<
 \Xi_{4,31}\!\left(\frac{1131}{1000}\right),
\]
so $\rho_{4,*}<1131/1000$, whereas the unique root of
$\Phi^{(4)}_{25}$ satisfies
\[
 \Phi^{(4)}_{25}\!\left(\frac{283}{250}\right)<0<
 \Phi^{(4)}_{25}\!\left(\frac{1133}{1000}\right).
\]
Thus
\[
 \rho_{4,*}<\frac{1131}{1000}
 <\frac{283}{250}<\rho_{4,\mathrm{fold}}.
\]
\end{proof}

\subsection{Reconstruction and positivity}

The detector is only useful if its roots reconstruct to genuine positive
metrics.  Here this can be proved uniformly.

Eliminate $q$ from \eqref{eq:H4F1} and \eqref{eq:H4F3} and reduce the result
modulo the quadratic equation \eqref{eq:H4F2} in $p$.  After deleting the
nonzero factors $48r^4(r-1)/(10r-1)^3(2\rho-1)^2$, the coefficient of $p$
is a polynomial $A_4(\rho,r)$.  Its resultant with the detector is
\[
\begin{aligned}
\operatorname{Res}_r(\mathcal D_\rho,A_4)
={}&C(\rho+1)^4(2\rho-1)^{15}(2\rho+3)^2(54\rho-19)^4\\
&\times\Psi^{(4)}_{18}(\rho)^2,
\end{aligned}
\]
with $C>0$.  Every displayed factor is nonzero for $\rho>1/2$, and
$\Psi^{(4)}_{18}>0$ there.  Hence the remainder is genuinely linear at
every asymmetric detector root and determines a unique real value of $p$.

The reconstruction equation \eqref{eq:H4F2} is
\[
 (2\rho-1)(10r-1)p^2-24\rho r^2p
 +r^2\bigl((8\rho+6)r+\rho\bigr)=0.
\]
For $\rho>1/2$ and $r>1/6$ its leading and constant coefficients are positive
and its linear coefficient is negative.  Put
\[
 p_0=\frac{8r+1}{24}.
\]
After the shifts
\[
 r=\frac16+a,\qquad \rho=\frac12+b,\qquad a,b\ge0,
\]
the numerator of $F_2(p_0)$ is
\[
\begin{aligned}
&1280a^3b+3456a^3+832a^2b+1728a^2\\
&\qquad+\frac{476}{3}ab+288a+\frac{196}{27}b+16>0,
\end{aligned}
\]
and the numerator of the difference between the quadratic vertex and
$p_0$ is
\[
 128a^2b+144a^2+\frac{116}{3}ab+48a+\frac{44}{9}b+4>0.
\]
Thus $p_0$ lies strictly to the left of the smaller real root and every
reconstructed real root satisfies
\[
 p>\frac{8r+1}{24}>0.
\]
Then \eqref{eq:H4F1} becomes
\[
 (8r+1-24p)q^2+24p^2q-9p^2=0,
\]
whose common reconstructed root is necessarily real, unique, and positive.
The detector roots remain in $r>1/6$ throughout the two-root phase.

\begin{theorem}[four-factor Einstein phase diagram]
\label{thm:H4Einsteinphase}
For the continuous $3+1$ algebraic system one has
\[
 \begin{array}{c|c}
 \rho>\rho_{4,\mathrm{fold}}
 &\text{no positive asymmetric solution},\\[1mm]
 \rho=\rho_{4,\mathrm{fold}}
 &\text{one degenerate positive asymmetric solution},\\[1mm]
 1/2<\rho<\rho_{4,\mathrm{fold}}
 &\text{exactly two positive asymmetric solutions}.
 \end{array}
\]
Each solution is uniquely reconstructed from its detector root.  Whenever
$\rho$ is the Cartan value of an actual symmetric pair $H/K$, these
solutions are genuine $H^4$-invariant Einstein metrics by symmetric
criticality for the corresponding finite automorphism group.
\end{theorem}

\subsection{Discrete families and the Cartan list}

For the $AII$ family
\[
 H/K=\SU(2n)/\Sp(n),
 \qquad
 \rho=\frac{n}{n-1},
\]
the first integer member below the fold is $n=9$:
\[
 \frac87>\rho_{4,\mathrm{fold}},
 \qquad
 \frac98<\rho_{4,*}<\rho_{4,\mathrm{fold}}.
\]
Thus $\SU(18)^4/\Delta\Sp(9)$ is the first member of this discrete family
lying in the two-root phase.  At $\rho=9/8$ the detector has
exactly two positive roots,
\[
 r_-=0.967796528665\ldots<1<1.102746101811\ldots=r_+,
\]
which reconstruct to
\[
 (p,q,r)\approx(1.25601417,1.22950814,0.96779653),
\]
and
\[
 (p,q,r)\approx(1.27127506,1.35654659,1.10274610).
\]
The exact existence statement follows from
Theorem~\ref{thm:H4Einsteinphase}; these decimals are only for orientation.

\subsubsection{A balanced Einstein metric at the first $AII$ Cartan point}

The balanced symmetry type is not merely a formal feature of the continuous
representation parameter.  At the first four-factor $AII$ value below the
response wall,
\[
 \rho=\frac98,\qquad
 H/K=\SU(18)/\Sp(9),
\]
the $\Gamma_{4,2}$-fixed equations have a positive asymmetric solution.

\begin{theorem}[balanced $AII$ metric at $n=9$]
\label{thm:H4balanced-AII9}
On
\[
 M_9=\SU(18)^4/\Delta\Sp(9)
\]
there is a normalized $\Gamma_{4,2}$-fixed Einstein metric
\[
 g_{\mathrm{bal}}
 =(p,p,q,q;A,B,1)
\]
with
\[
\begin{aligned}
p&=1.2372527948223150\ldots,\\
q&=1.2845474708446486\ldots,\\
A&=1.0396839360703710\ldots,\\
B&=0.9632340955343939\ldots .
\end{aligned}
\]
It is asymmetric, positive, and locally unique up to interchange of the two
blocks.  More precisely, after ordering $p<q$, it is the unique zero of the
balanced Einstein map in the rational cube
\[
 \mathcal B_9=
 \left\{
 \left\|(p,q,A,B)-c_9\right\|_\infty\le10^{-8}
 \right\},
\]
where
\[
 c_9=
 \left(
 \frac{1237252794822}{10^{12}},
 \frac{1284547470845}{10^{12}},
 \frac{1039683936070}{10^{12}},
 \frac{963234095534}{10^{12}}
 \right).
\]
\end{theorem}

\begin{proof}
Set $\rho=9/8$ in
\eqref{eq:H4bal-rpq}--\eqref{eq:H4bal-rC}, and let
$\mathcal P=(P_1,P_2,P_3,P_4)$ be the polynomial map obtained by taking the
numerators of
\[
 r_p-r_q,\qquad r_A-r_B,\qquad
 \frac{r_p+r_q}{2}-r_C,\qquad
 \frac{r_A+r_B}{2}-r_C .
\]
All denominators are strictly positive on $\mathcal B_9$, so
$\mathcal P=0$ is equivalent there to the Einstein equations.

Let
\[
 J_0=D\mathcal P(c_9),\qquad C_0=J_0^{-1}.
\]
Exact rational interval evaluation on $\mathcal B_9$ gives
\[
 \left\|C_0\mathcal P(c_9)\right\|_\infty
 <
 \frac1{2500000000000}
\]
and
\[
 \sup_{x\in\mathcal B_9}
 \left\|I-C_0D\mathcal P(x)\right\|_\infty
 <
 \frac1{30000}.
\]
Hence the Newton map
\[
 T(x)=x-C_0\mathcal P(x)
\]
is a contraction on $\mathcal B_9$ and
\[
 \|T(c_9)-c_9\|_\infty+
 \frac1{30000}\,10^{-8}
 <10^{-8}.
\]
Thus $T$ maps the cube strictly into itself.  Banach's fixed-point theorem
gives one and only one zero in $\mathcal B_9$.

The cube lies in the positive orthant, and its $p$- and $q$-intervals as
well as its $A$- and $B$-intervals are disjoint.  The solution is therefore
positive and asymmetric.  Interchanging the two blocks sends
$(p,q,A,B)$ to $(q,p,B,A)$ and gives the only equivalent representative
with reversed ordering.
\end{proof}

\begin{theorem}[global uniqueness in the balanced $n=9$ family]
\label{thm:H4balanced-AII9-global}
Up to homothety and interchange of the two $2$-blocks, the metric
$g_{\mathrm{bal}}$ of Theorem~\ref{thm:H4balanced-AII9} is the unique
positive asymmetric Einstein metric in the full $\Gamma_{4,2}$-fixed family
on
\[
 \SU(18)^4/\Delta\Sp(9).
\]
The same fixed family contains exactly two symmetric Einstein metrics.
\end{theorem}

\begin{proof}
Put $\rho=9/8$ and let $\lambda$ denote the common Einstein constant.  From
$r_p=\lambda$ one obtains
\begin{equation}\label{eq:H4n9-Afromp}
 A=4p-\frac12-8\lambda p^2,
\end{equation}
and the equation $r_A=\lambda$ becomes
\begin{equation}\label{eq:H4n9-blockpoly}
\begin{aligned}
F(p,\lambda)={}&-8704\lambda^3p^6+10752\lambda^2p^5
-(1344\lambda^2+4384\lambda)p^4\\
&+(1056\lambda+592)p^3-(66\lambda+207)p^2+24p-1=0.
\end{aligned}
\end{equation}
The same equations hold with $(p,A)$ replaced by $(q,B)$.  Moreover
\[
 \lambda=r_C>\frac5{36},
\]
while $A>0$ and \eqref{eq:H4n9-Afromp} give
\[
 \lambda<\frac{8p-1}{16p^2}\le1.
\]
Hence every positive solution has $5/36<\lambda<1$.

Eliminating $q$ from the second copy of \eqref{eq:H4n9-blockpoly} and
$r_C-\lambda=0$, and then eliminating $p$ against
\eqref{eq:H4n9-blockpoly}, gives, after deleting nonzero constants and the
impossible factor $\lambda^{48}$,
\begin{equation}\label{eq:H4n9-lambda-eliminant}
 Q_2(\lambda)Q_4(\lambda)Q_5(\lambda)^2Q_{10}(\lambda)^2=0.
\end{equation}
The four factors are recorded in
Appendix~\ref{app:H4balanced-n9-elimination}.  Exact Sturm counts give
\[
\begin{array}{c|cccc}
 &Q_2&Q_4&Q_5&Q_{10}\\ \hline
N_{\mathbb R}&2&0&1&4\\
N_{(0,\infty)}&2&0&1&4.
\end{array}
\]

For $Q_2$, the common subresultant reduces to
\[
 14792\lambda+1920p-6543=0.
\]
By symmetry the same relation holds for $q$, so $p=q$; moreover direct
reduction gives $A=1$ at every root of $Q_2$.  These are precisely the two
symmetric solutions.

The factor $Q_4$ has no real root.  For the $Q_{10}$ factor, the exact
quadratic penultimate subresultant $H_{10}(p,\lambda)$ is necessary for a
common root.  Eliminating $\lambda$ from $Q_{10}$ and $H_{10}$ gives the
degree-$20$ polynomial $R_{20}(p)$ recorded in the appendix.  Its Sturm
chain has no real root.  Thus $Q_{10}$ produces no real metric coefficient
$p$ and hence no positive Einstein metric.

It remains to consider $Q_5$.  This polynomial has one real root,
\[
 \frac{2783952}{10^7}<\lambda_5<\frac{2783953}{10^7}.
\]
Eliminating $\lambda$ from $Q_5$ and the corresponding quadratic
subresultant gives the degree-$10$ polynomial
\begin{equation}\label{eq:H4n9-R10}
\begin{aligned}
R_{10}(p)={}&417624883200p^{10}-1602650112000p^9
+2447647555072p^8\\
&-2051779628864p^7+1120872741086p^6-408205319880p^5\\
&+100371031695p^4-16591337408p^3+1793365224p^2\\
&-116614968p+3591403.
\end{aligned}
\end{equation}
Its Sturm count gives exactly two real roots, both positive, isolated by
\[
 \frac{12372527}{10^7}<p_-<\frac{12372529}{10^7},
 \qquad
 \frac{12845474}{10^7}<p_+<\frac{12845475}{10^7}.
\]
There is one remaining diagonal possibility to exclude.  If $p=q$, then
necessarily $A=B$, but this alone does not yet force $A=1$.  Eliminating
$p$ from the exact diagonal block equation $F(p,\lambda)=0$ and the
remaining Einstein equation
\[
 r_C(p,p,A,A)-\lambda=0,
 \qquad
 A=4p-\frac12-8\lambda p^2,
\]
gives
\[
 C\,\lambda^8Q_2(\lambda)Q_4(\lambda),
 \qquad C\ne0.
\]
Since $Q_4$ has no real zero, every real diagonal solution belongs to the
$Q_2$ factor.  In particular the real $Q_5$ root cannot satisfy $p=q$.

Thus any positive asymmetric solution belonging to $Q_5$ must have
\[
 \{p,q\}=\{p_-,p_+\}.
\]
The certified solution of Theorem~\ref{thm:H4balanced-AII9} realizes this
unordered pair.  Hence it is the unique positive asymmetric solution,
modulo block interchange.
\end{proof}

\begin{proposition}[non-product and non-naturally reductive]
\label{prop:H4balanced-AII9-geometry}
The metric $g_{\mathrm{bal}}$ of
Theorem~\ref{thm:H4balanced-AII9} is Riemannian irreducible and is not
naturally reductive with respect to any connected transitive group of
isometries.
\end{proposition}

\begin{proof}
First consider natural reductivity with respect to the displayed $H^4$.
An $H^4$-naturally reductive metric is induced by an
$\operatorname{Ad}(H^4)$-invariant form
\[
 \lambda_1Q\oplus\lambda_2Q\oplus\lambda_3Q\oplus\lambda_4Q.
\]
The balanced horizontal equalities force, after relabelling,
\[
 \lambda_1=\lambda_2=a,\qquad
 \lambda_3=\lambda_4=b.
\]
On the anti-diagonal basis
\[
 u_1=\frac1{\sqrt2}(1,-1,0,0),\quad
 u_2=\frac1{\sqrt2}(0,0,1,-1),\quad
 u_3=\frac12(1,1,-1,-1),
\]
the induced quotient eigenvalues are
\[
 a,\qquad b,\qquad \frac{2ab}{a+b}.
\]
After the normalization in which the $u_3$-scale is one, natural
reductivity therefore forces
\[
 p=A=\frac{a+b}{2b},
 \qquad
 q=B=\frac{a+b}{2a}.
\]
But the certified intervals in Theorem~\ref{thm:H4balanced-AII9} give
$p>1.23$ and $A<1.05$, so $p\ne A$.  Thus the metric is not
$H^4$-naturally reductive.

By Theorem~\ref{thm:H4-fullisom}, the connected full isometry algebra is
$\mathfrak h^4$.  For a compact naturally reductive metric, the connected
transvection group is a transitive normal subgroup of the connected full
isometry group.  Every connected normal subgroup of $H^4$ has Lie algebra
equal to a sum of some of the four simple ideals.  A proper such subgroup
has dimension at most $3\dim H$, whereas
\[
 \dim M_9=4\dim H-\dim K>3\dim H.
\]
Hence no proper normal subgroup can act transitively.  Natural reductivity
with respect to any connected transitive group would therefore force the
transvection group to be all of $H^4$, which has just been excluded.

Finally suppose that $(M_9,g_{\mathrm{bal}})$ were a nontrivial Riemannian
product.  The de Rham factors are preserved by the identity component of
the isometry group, up to a discrete permutation of isometric factors.
Thus the four simple ideals of $\mathfrak h^4$ split among the nontrivial
de Rham factors.  Transitivity forces each de Rham factor to receive at
least one ideal.  The isotropy algebra of a product action then splits
accordingly as a direct sum of its intersections with the corresponding
sums of simple ideals.  This is impossible for
\[
 \Delta\mathfrak k\subset\mathfrak h^4,
\]
which is simple and projects nontrivially to every one of the four simple
ideals.  Hence the metric is Riemannian irreducible.
\end{proof}

\begin{theorem}[three asymmetric geometries at the first four-factor $AII$ point]
\label{thm:H4-AII9-three}
On
\[
 M_9=\SU(18)^4/\Delta\Sp(9)
\]
the fixed families studied in this section contain exactly three asymmetric
Einstein metrics up to homothety and their internal block permutations: two
of type $3+1$ and one of type $2+2$.  They are pairwise non-homothetic and
non-isometric under arbitrary Riemannian isometries.  All three are
Riemannian irreducible and none is naturally reductive with respect to any
connected transitive group of isometries.
\end{theorem}

\begin{proof}
The two $3+1$ metrics are exactly those supplied by
Theorem~\ref{thm:H4Einsteinphase}; the balanced metric is globally unique by
Theorem~\ref{thm:H4balanced-AII9-global}.  A $3+1$ metric has horizontal
partition $3+1$, whereas the balanced metric has partition $2+2$.
Corollary~\ref{cor:H4-arbitrary-nonisometry} and
Theorem~\ref{thm:H4-fullisom} therefore separate the balanced metric from
both $3+1$ metrics under arbitrary isometry and homothety.

For a $3+1$ metric, an $H^4$-naturally reductive form has, after relabelling,
weights $(a,a,a,b)$.  On the anti-diagonal multiplicity space its two
vertical eigenvalues are
\[
 a\quad\text{(multiplicity two)},
 \qquad \frac{4ab}{3a+b}\quad\text{(multiplicity one)}.
\]
After normalizing the latter to one, natural reductivity forces
\[
 p=r=\frac{3a+b}{4b}.
\]
For the two reconstructed Einstein metrics one has respectively
\[
 (p,r)\approx(1.25601417,0.96779653),
 \qquad
 (p,r)\approx(1.27127506,1.10274610),
\]
so neither is $H^4$-naturally reductive.  The transvection-group argument
of Proposition~\ref{prop:H4balanced-AII9-geometry} applies verbatim: a
compact naturally reductive presentation would have a transitive normal
subgroup of the connected full isometry group, while every proper normal
subgroup of $H^4$ has dimension at most $3\dim H<\dim M_9$.  Hence the
transvection group would have to be all of $H^4$, already excluded above.
Thus neither metric is naturally reductive with respect to any connected
transitive group.  The diagonal simple isotropy excludes a nontrivial de
Rham product.

Finally, after a hypothetical isometry between the two $3+1$ metrics has
been reduced to an automorphism of $(H^4,\Delta K)$, the singleton factor is
intrinsic because its horizontal equality multiplicity is one.  On the
vertical block the normalized one-dimensional separating eigenspace is also
intrinsic, while its complementary eigenspace has multiplicity two.  Hence
the repeated vertical eigenvalue $r$ is an isometry invariant.  Since the
two values of $r$ are distinct, the two $3+1$ metrics are not homothetic or
isometric.
\end{proof}

For orientation, the three normalized metrics are summarized by
\begingroup\small
\[
\begin{array}{c|c|c|c}
\text{type}&\text{normalized scale data}&\text{naturally reductive}&\text{product}\\ \hline
3+1_-&(p,q,r)\approx(1.256014,1.229508,0.967797)&\text{no}&\text{no}\\
3+1_+&(p,q,r)\approx(1.271275,1.356547,1.102746)&\text{no}&\text{no}\\
2+2&(p,q,A,B)\approx(1.237253,1.284547,1.039684,0.963234)&\text{no}&\text{no}
\end{array}
\]
\endgroup
For all three,
\[
 \Lie\Isom_0(M_9,g)=\mathfrak{su}(18)^4.
\]

The phase theorem also gives a direct consequence for the Cartan list.
Under the standing assumptions that $H/K$ is
irreducible compact symmetric and $H,K$ are simple, one only has to compare
\[
 \rho=\frac{\dim K}{\dim(H/K)}
\]
with $\rho_{4,\mathrm{fold}}$.

\begin{theorem}[four-factor Cartan-list consequence]\label{thm:H4cartan}
Assume that $H/K$ is an effective irreducible compact symmetric pair and
that both $H$ and $K$ are simple.  Then the normalized
$\Gamma_{4,3}$-fixed $3+1$ family on $H^4/\Delta K$ contains exactly two
positive asymmetric Einstein metrics precisely for
\[
 \frac{\SU(n)^4}{\Delta\SO(n)},
 \qquad n=3\ \text{or}\ n\ge5,
\]
\[
 \frac{\SU(2n)^4}{\Delta\Sp(n)},
 \qquad n\ge9,
\]
and
\[
 \frac{E_6^4}{\Delta\Sp(4)},\qquad
 \frac{E_7^4}{\Delta\SU(8)},\qquad
 \frac{E_8^4}{\Delta\SO(16)}.
\]
For $2\le n\le8$ in type $AII$, for $E_6/F_4$ and
$F_4/\mathrm{Spin}(9)$, and for the sphere-type pairs
$\SO(m+1)/\SO(m)$ with $m\ge5$, this $3+1$ family contains no positive
asymmetric Einstein metric.

In addition, at the first $AII$ Cartasn value below the four-factor response
wall, $n=9$, the balanced $\Gamma_{4,2}$-fixed family contains exactly one
positive asymmetric Einstein metric up to block interchange and homothety.
It lies on the analytic branch issuing from the response wall, is Riemannian
irreducible, and is not naturally reductive.
\end{theorem}

\begin{proof}
The Cartan classification of irreducible compact symmetric pairs
\cite{Helgason} shows that, under the additional requirement that both
$\mathfrak h$ and $\mathfrak k$ be simple, the complete list is
\[
\begin{array}{c|c}
H/K&\rho=\dim\mathfrak k/\dim\mathfrak q\\ \hline
\SU(n)/\SO(n),\ n=3\text{ or }n\ge5& n/(n+2)\\
\SU(2n)/\Sp(n),\ n\ge2& n/(n-1)\\
\SO(m+1)/\SO(m),\ m\ge5& (m-1)/2\\
E_6/\Sp(4)&6/7\\
E_6/F_4&2\\
E_7/\SU(8)&9/10\\
E_8/\SO(16)&15/16\\
F_4/\mathrm{Spin}(9)&9/4 .
\end{array}
\]
Here the omission of $n=4$ from type $AI$ is exactly the failure of
$\mathfrak{so}(4)$ to be simple.  All other compact irreducible Cartan
types have nonsimple isotropy and therefore lie outside the standing
hypotheses.

Now compare these rational values with the unique fold parameter.
For type $AI$,
\[
 \rho=\frac n{n+2}<1<\rho_{4,\mathrm{fold}}.
\]
For type $AII$,
\[
 \frac n{n-1}<\rho_{4,\mathrm{fold}}
 \quad\Longleftrightarrow\quad n\ge9;
\]
indeed $8/7>\rho_{4,\mathrm{fold}}>9/8$.  The exceptional ratios
\[
 \frac67,\qquad\frac9{10},\qquad\frac{15}{16}
\]
lie below the fold, whereas
\[
 2,\qquad\frac94,\qquad\frac{m-1}{2}\ (m\ge5)
\]
lie above it.  None equals the fold.  The conclusion follows directly from
Theorem~\ref{thm:H4Einsteinphase}.
\end{proof}

\appendix

\section{Coefficient record for the balanced four-factor branch}
\label{app:H4balanced-coefficients}

This appendix records the polynomial system used in
\eqref{eq:H4bal-linear-system}.  It is included so that the Cramer reduction
\eqref{eq:H4bal-Cramer-reduction} can be checked directly, while keeping the
main four-factor argument focused on the response mechanism.  With
$z=(m_2,y_2,k_3,r_2)^T$,
\[
\mathcal M(x)=
\begin{pmatrix}
a_{11}&a_{12}&0&a_{14}\\
0&a_{22}&0&0\\
-3x&2x^2&-x^3&0\\
a_{41}&a_{42}&a_{43}&a_{44}
\end{pmatrix},
\]
where
\begin{align*}
a_{11}&=-8704x^3+12416x^2-7616x+1920,\\
a_{12}&=28544x^3-32288x^2+13440x-2592,\\
a_{14}&=-42100x^3+48851x^2-20028x+3483,\\
a_{22}&=40x^3-50x^2+20x-3,\\
a_{41}&=276480x^3-320000x^2+133120x-24576,\\
a_{42}&=-202880x^3+228320x^2-90752x+16032,\\
a_{43}&=-41088x^3+45728x^2-17920x+2976,\\
a_{44}&=214724x^3-240863x^2+94380x-15975,
\end{align*}
and
\[
b(x)=
\begin{pmatrix}
452608x^3-536064x^2+206592x-28224\\
-128x^3+352x^2-220x+33\\
3-8x\\
188416x^3-196608x^2+83712x-20160
\end{pmatrix}.
\]
These coefficients arise by substituting
\eqref{eq:H4bal-k1}--\eqref{eq:H4bal-critical-relations} into the
$h^2$ coefficients of $F_3,F_4$ and the $h^3$ coefficients of $F_1,F_2$,
clearing harmless nonzero scalar factors, and reducing modulo $P(x)$.
No numerical approximation enters this reduction.

\begin{proposition}[exact verification of the coefficient record]
\label{prop:H4balanced-appendix-check}
Let $R_i(x)$, $i=1,\ldots,4$, denote respectively the coefficient rows
obtained directly from
\[
 [h^2]F_3,\qquad [h^2]F_4,\qquad [h^3]F_1,\qquad [h^3]F_2
\]
after substituting \eqref{eq:H4bal-k1}--\eqref{eq:H4bal-critical-relations}
and reducing in the quotient ring $\mathbb Q[x]/(P)$.  Let
$\widetilde R_i$ be the four rows of the displayed augmented system
$(\mathcal M\mid b)$.  Then
\[
 \widetilde R_i=\gamma_i(x)R_i\qquad\text{in }\mathbb Q[x]/(P),
\]
where
\[
\begin{aligned}
\gamma_1&=4(119164x^3-133361x^2+52500x-8937),\\
\gamma_2&=\frac{3492x^3-3743x^2+1420x-231}{8},\\
\gamma_3&=\frac{972x^3-961x^2+356x-57}{32},\\
\gamma_4&=\frac{1197228x^3-1337669x^2+524324x-89373}{2}.
\end{aligned}
\]
None of these multipliers vanishes at a root of $P$: the resultants of
$P$ with their numerators are respectively
\[
\begin{gathered}
485699642990788608,\qquad
70962905088,\qquad
1019215872,\qquad
728549464486182912,
\end{gathered}
\]
all nonzero.  Hence the appendix system is exactly equivalent to the
coefficient system obtained from the Ricci equations at the critical root.

Moreover direct determinant expansion gives
\[
 \det\mathcal M+\frac{C_*B_*}{256}\equiv0\pmod P,
 \qquad
 \det\mathcal M_\rho-2C_*A_*\equiv0\pmod P.
\]
Thus the Cramer formula \eqref{eq:H4bal-r2} follows from the displayed
coefficient record alone.
\end{proposition}

\begin{proof}
Expand the four Ricci differences in \eqref{eq:H4bal-F} using
\eqref{eq:H4bal-expansion}.  The odd/even symmetry removes all irrelevant
orders.  Substitute $k_1$ and $\rho_*$ from
\eqref{eq:H4bal-k1}--\eqref{eq:H4bal-critical-relations}, invert the
denominators modulo $P$, and reduce to degree at most three.  Coefficient
comparison gives the four multipliers above.  Their nonzero resultants with
$P$ show that no equation was lost in the row rescaling.  The final two
congruences are obtained by expanding the two determinants and taking the
remainders modulo $P$.
\end{proof}

\section{Exact elimination certificate for the balanced $n=9$ family}
\label{app:H4balanced-n9-elimination}

The factors in \eqref{eq:H4n9-lambda-eliminant} are
\[
 Q_2=118336\lambda^2-67824\lambda+9705,
\]
\[
 Q_4=85525504\lambda^4-178227456\lambda^3+146321616\lambda^2
 -39670200\lambda+3372225,
\]
\[
\begin{aligned}
Q_5={}&28731224\lambda^5-107853933\lambda^4+174378112\lambda^3\\
&-146115968\lambda^2+61759488\lambda-9031680,
\end{aligned}
\]
and
\begingroup\scriptsize
\[
\begin{aligned}
Q_{10}={}&46809453462458179125248\lambda^{10}
-314586034387869050290176\lambda^9\\
&+1010451232441805406027776\lambda^8
-2003015681816922939798016\lambda^7\\
&+2672398944406405048886048\lambda^6
-2474119520605069968097412\lambda^5\\
&+1585939874511436273944671\lambda^4
-683530719800271346906224\lambda^3\\
&+186339422829223521541152\lambda^2
-28509209062714489132800\lambda\\
&+1822293612319632480000.
\end{aligned}
\]
\endgroup
The exact isolating intervals for their positive real roots are
\[
\begin{array}{c|c}
Q_2&(0.2759757,0.2759758),\ (0.2971719,0.2971720)\\
Q_4&\varnothing\\
Q_5&(0.2783952,0.2783953)\\
Q_{10}&(0.1991575,0.1991576),\ (0.3833677,0.3833678),\\
& (0.8540932,0.8540933),\ (0.9998771,0.9998772).
\end{array}
\]
Each interval contains exactly one root by Sturm's theorem.

For the $Q_{10}$ factor, eliminating $\lambda$ from the quadratic
penultimate subresultant and $Q_{10}$ gives, up to a nonzero constant,
\begingroup\scriptsize
\[
\begin{aligned}
R_{20}(p)={}&647846034381691084800000000p^{20}
-6078972948353735399424000000p^{19}\\
&+25517470116767457031833600000p^{18}
-63908220599190020274588441600p^{17}\\
&+107618920735950870038638701152p^{16}
-130330558655869260004967585408p^{15}\\
&+118644433408782053700277158448p^{14}
-83811829358343784090026939136p^{13}\\
&+47040207383101583662090240592p^{12}
-21335077554718709193502154816p^{11}\\
&+7904941125574329014365177000p^{10}
-2405364793665697625467059712p^9\\
&+601295446443520819098674686p^8
-122968121949797712969792488p^7\\
&+20390150710464224079067447p^6
-2701251087098777590652352p^5\\
&+279401562807744960997722p^4
-21766775691506486381048p^3\\
&+1203824517295040374503p^2
-42298771803979115040p\\
&+714255576514559618.
\end{aligned}
\]
\endgroup
Its Sturm variation is the same at $-\infty$ and $+\infty$, so it has no
real root.  This is the exact exclusion used in
Theorem~\ref{thm:H4balanced-AII9-global}.

\bigskip

\noindent
\textsc{Anna Siffert}\\
Universit\"at M\"unster, Mathematisches Institut,
Einsteinstr.\ 62, 48149 M\"unster, Germany\\
\texttt{asiffert@uni-muenster.de}

\end{document}